\documentclass[11pt, oneside]{amsart} 
\usepackage[utf8]{inputenc}
\usepackage{geometry}
\usepackage[english]{babel}
\usepackage{amsmath,mathtools,amssymb,graphicx,enumerate,bbm,xcolor,amsthm}
\usepackage[all]{xy}
\usepackage{hyperref}
\usepackage{tikz-cd}
\usetikzlibrary{decorations.pathmorphing}
\newtheorem{theorem}{Theorem}[section]
\newtheorem{corollary}[theorem]{Corollary}
\newtheorem{definition}[theorem]{Definition}
\newtheorem{lemma}[theorem]{Lemma}
\newtheorem{proposition}[theorem]{Proposition}
\newtheorem{conjecture}[theorem]{Conjecture}
\newtheorem{condition}[theorem]{Condition}
\newtheorem{remark}[theorem]{Remark}

\newcommand\GL{{\mathrm {GL}}}

\newcommand\Spec{{\mathrm {Spec}}}

\theoremstyle{plain}

\theoremstyle{remark}

\numberwithin{equation}{section}

\newcommand\Q{{\mathbb Q}}

\newcommand{\Lc}{\mathcal{L}}

\begin{document}
	
	\title{Selmer groups of symmetric powers of ordinary modular Galois
	representations}
	\author{Xiaoyu ZHANG}
	\address{LAGA, Institut Galil\'{e}e, U. Paris 13, av. J.-B. Cl\'{e}ment, Villetaneuse 93430, France.}
	\email{zhang@math.univ-paris13.fr}
	\maketitle
	\begin{abstract}
		Let $p$ be a fixed odd prime number,
		$\mu$ be a Hida family over 
		the Iwasawa algebra of one variable,
		$\rho_{\mu}$ 
		its Galois representation,
		$\Q_\infty/\mathbb{Q}$ the $p$-cyclotomic tower and $S$ the variable
		of the cyclotomic Iwasawa algebra.
		We compare, for $n\leq 4$ and under certain assumptions, 
		the characteristic power series
		$L(S)$
		of the dual of Selmer groups
		$\mathrm{Sel}(\mathbb{Q}_{\infty},\mathrm{Sym}^{2n}\otimes\mathrm{det}^{-n}\rho_{\mu})$ 
		to certain congruence ideals. 
		The case $n=1$ has been treated by H.Hida.
		In particular, we express the first term of the Taylor expansion 
		at the trivial zero $S=0$ of
		$L(S)$
		in terms of an $\Lc$-invariant and a congruence number.
		We conjecture the non-vanishing of this $\Lc$-invariant;
		this implies therefore that these Selmer groups are cotorsion.
		We also show that our $\Lc$-invariants coincide with Greenberg's
		$\mathcal{L}$-invariants calculated 
		by R.Harron and A.Jorza.
	\end{abstract}
	
	\tableofcontents
	
	\section{Introduction}
	In this paper we study the relation between Selmer groups and congruence ideals over a cyclotomic tower for a certain type of Galois representations.

	Let's first recall some history.
	The congruences between modular forms/automorphic forms 
	have been investigated by many authors. 
	One kind of congruence is between elliptic cusp forms and Eisenstein series.
	This kind of congruence has received much attention. 
	One can use this kind of congruence to study the Selmer groups of the Galois representation associated to a certain Dirichlet character 
	(see for example \cite{Ribet}, \cite{Wiles}). 
	The Eisenstein series can be seen as a kind of lift of some Dirichlet character from $\mathrm{GL}_1$ to $\mathrm{GL}_2$. 
	This point of view can be generalized from $\mathrm{GL}_2$ to other groups
	for the study of Selmer groups of certain automorphic Galois representations by using congruences between Klingen-Eisenstein series and cusp forms. 
	(see for example \cite{SkinnerUrban}). 
	
	Another kind of congruence is between cusp forms. 
	This congruence  was perhaps first studied numerically in \cite{DoiOhta} 
	and then the congruence was related to special values of Dirichlet series as investigated in \cite{DoiHida}. 
	Later works include for example \cite{HidaCongruencesOfCusp}, \cite{HidaOnCongruenceDivisors} and \cite{RibetModp}. 
	During further investigations of the relation between congruence numbers and special $L$-values, 
	H.Hida developed the theory of ordinary $p$-adic modular forms in families (see \cite{HidaHilbert}).
	Then Hida used his theory to study a third kind of congruence, that is, between Hida families.
	Moreover, for an elliptic cusp form $f$, he related the congruence between 
	the Hida family containing $f$ and other Hida families 
	to the divisibility of the special values of 
	the complex adjoint $L$-function of $f$ (see \cite{Hida88b}
	and also the lecture notes \cite{HidaArithmetic}). 
	From then on, many generalizations of these two kinds of congruences 
	have been carried out for Hilbert modular forms
	(see for example \cite{Ghate2010}, \cite{Dimitrov}).
	
	In the paper \cite{HidaTilouine}, a particular type of congruence between cusp automorphic forms is considered by H.Hida and J.Tilouine. 
	This kind of congruence involves Langlands functoriality from one group to another. 
	More precisely, fix some odd prime $p>2$
	and suppose that $G$ and $H$ are two connected reductive algebraic groups defined over some number field $F$. 
	Suppose that there is a homomorphism 
	$r\colon  ^LH\to ^LG$ between the Langlands' $L$-groups 
	$^LH$ of $H$ and $^LG$ of $G$.
	We assume also that the Langlands functoriality $H\leadsto G$ for this  homomorphism $r$ is established.
	Let's write the resulting transfer map as 
	$r_\ast\colon \mathcal{A}_{\mathrm{cusp}}(G)\to\mathcal{A}_{\mathrm{cusp}}(H)$
	between the set $\mathcal{A}_{\mathrm{cusp}}(G)$, resp. $\mathcal{A}_{\mathrm{cusp}}(H)$ of cusp automorphic representations of 
	$G(\mathbb{A}_F)$, resp. $H(\mathbb{A}_F)$.
	Let $\pi$ be a cusp automorphic representation on $H$ and $\pi^G$ 
	be its image (as a packet) under the transfer map $r_\ast$.
	In \cite{HidaTilouine} the congruence between $\pi^G$ and other cusp automorphic representations on the group $G$ is studied. 
	In fact, the authors study this congruence in $p$-adic families. That is to say, suppose that we can establish Hida theory for the groups $G$ and $H$ 
	and let's write the localized big Hida-Hecke algebras of some suitable auxiliary conductor 
	(related to that of $\pi$) 
	over $H$, resp. $G$, as $\mathbf{T}^H$, resp. $\mathbf{T}^G$. 
	Then the Langlands transfer morphism
	$r_\ast\colon\mathcal{A}_{\mathrm{cusp}}(G)\to\mathcal{A}_{\mathrm{cusp}}(H)$ gives a natural morphism of algebras
	$\mathbf{T}^G\to\mathbf{T}^H$.
	Let $\mathbf{T}^H\to A$ be a Hida family passing through $\pi$. 
	As we have said, the study of congruence between the Hida family
	$\mathbf{T}^H\to A$ and other Hida families on $H$ 
	was initiated by H.Hida.
	The authors in \cite{HidaTilouine} study more specifically the congruence
	between the Hida family on $G$ given by
	$\mathbf{T}^G\to\mathbf{T}^H$ (see Definition \ref{congruenceideal}) and those not coming from $H$. 
	This is equivalent to considering the congruence ideal of the
	morphism $\mathbf{T}^G\to\mathbf{T}^H$.
	The main result of \cite{HidaTilouine} is to relate congruence ideals 
	to Selmer groups for some twisted symmetric powers of ordinary modular Galois representations by considering the Langlands functorialities such as
	$\mathrm{Sym}^{n-1}\colon\mathrm{GL}_{2/\mathbb{Q}}\leadsto\mathrm{GL}_{n/\mathbb{Q}}$ 
	(corresponding to the morphism of Langlands $L$-groups
	$\mathrm{Sym}^{n-1}\colon\mathrm{GL}_2(\mathbb{C})\to\mathrm{GL}_n(\mathbb{C})$),
	$\mathrm{Sym}^3\colon\mathrm{GL}_{2/\mathbb{Q}}\leadsto\mathrm{GSp}_{4/\mathbb{Q}}$ 
	(corresponding to $\mathrm{Sym}^3\colon\mathrm{GL}_2(\mathbb{C})\to\mathrm{GSp}_4(\mathbb{C})$)
	and
	$\mathrm{GSp}_{4/\mathbb{Q}}\leadsto\mathrm{GL}_{4/\mathbb{Q}}$
	(corresponding to
	$\mathrm{GSp}_4(\mathbb{C})\hookrightarrow\mathrm{GL}_4(\mathbb{C})$)
	together with some Clebsch-Gordan decompositions for representations of $\mathrm{GL}_2$.
	
	Our work can be seen as a generalization of \cite{HidaTilouine} in the sense 
	that instead of considering only one field $\mathbb{Q}$, 
	we consider the $p$-cyclotomic tower $\mathbb{Q}_\infty/\mathbb{Q}$. 
	In other words, we treat an Iwasawa-theoretic version of the above congruence problem for the Langlands functorialities
	$\mathrm{Sym}^{n-1}\colon\mathrm{GL}_2\leadsto\mathrm{GL}_n$ for each layer of $\mathbb{Q}_\infty/\mathbb{Q}$.
	This leads us naturally into the world of Iwasawa theory. 
	Recall that the Iwasawa-Greenberg main conjecture equates the $p$-adic $L$-function of an ordinary Galois representation 
	(whose existence is widely open in many cases) 
	to the characteristic power series of the Selmer group of this Galois representation up to multiplication by some unit (see \cite[Conjecture 2]{GreenbergTrivialZero}). 
	In the various (partially) established cases of the main conjecture, the congruence ideals play a significant role 
	(see for example \cite{Wiles}, \cite{SkinnerUrban}, \cite{UrbanGroupesDeSelmer}, \cite{WanIwasawa}). 
	The general strategy for proving that the $p$-adic $L$-function divides the characteristic power series of the dual Selmer group is to insert the congruence ideal (or its characteristic power series): 
	(1) show that the $p$-adic $L$-function divides the congruence power series; 
	(2) show that the congruence power series divides the characteristic power series of the dual Selmer group. 
	This paper treats the second part (2) of the divisibility 
	(in fact we show an equality of the two power series over each layer of $\mathbb{Q}_\infty/\mathbb{Q}$).

	Let $p$ be an odd prime number. 
	Let $\rho\colon  \Gamma_\mathbb{Q}\to \mathrm{GL}_m(W)$ 
	be an ordinary Galois representation 
	where $\Gamma_\mathbb{Q}$ is the absolute Galois group of $\mathbb{Q}$ and
	$W$ is a $p$-profinite local ring. 
	The Galois group $\mathrm{Gal}(\mathbb{Q}_\infty/\mathbb{Q})$ 
	of the $p$-cyclotomic extension $\mathbb{Q}_\infty/\mathbb{Q}$ 
	is isomorphic to $\mathbb{Z}_p$.
	We fix a topological generator $\gamma$ for this Galois group. 
	Then we have an isomorphism of $p$-profinite $\mathbb{Z}_p$-algebras
	${\mathbb{Z}_p[[\mathrm{Gal}(\mathbb{Q}_\infty/\mathbb{Q})]]\simeq\mathbb{Z}_p[[S]]}$ by sending $\gamma$ to $1+S$.
	The Greenberg-Iwasawa main conjecture for $\rho$
	relates the $p$-adic $L$-function $L^{an}_p(\rho,S)$ of $\rho$
	to the characteristic power series $L^{alg}_p(\rho,S)$ of the dual Selmer group
	$\mathrm{Sel}(\mathbb{Q}_\infty,\rho)^\ast$
	of $\rho$
	under certain conditions. 
	
	Now let's be precise what kind of Galois representations $\rho$ we are going to treat.
	Suppose that $\pi$ is a non-CM holomorphic cusp automorphic representation of $\mathrm{GL}_2(\mathbb{A}_{\mathbb{Q}})$, 
	of conductor $N$, 
	cohomological for a local system of 
	highest weight $a>0$.
	Let $p>2$ be a prime relatively prime to $N$  where $\pi$ is ordinary.
	Let $\mathbf{T}$ be the big Hida-Hecke algebra of auxiliary conductor $N$. Let $\Spec(A)$ be an irreducible component of $\mathbf{T}$;
	the Hecke eigensystem for $\pi$ can be interpolated by a Hida family $\mu\colon\mathbf{T}\to A$.
	Assuming residual irreducibility, one can associate to $\mu$
	a Galois representation $\rho_{\mu}\colon \Gamma_{\mathbb{Q}}\to \GL_2(A)$
	which yields $\rho_{\pi}$ when specialized in the weight of $\pi$. 
	For any representation $\sigma\colon\Gamma_\mathbb{Q}\to\mathrm{GL}_2(W)$, we write
	$\mathcal{A}^n(\sigma)$ for the symmetric power representation
	${(\mathrm{Sym}^{2n}\otimes\mathrm{det}^{-n})\sigma}$.
	The Galois representations $\rho$ we will study are these $\mathcal{A}^n(\rho_\mu)$.
	For the $k$-th layer $\Q_k$ of  $\Q_\infty/\mathbb{Q}$,
	we consider the Selmer group
	$\mathrm{Sel}(\mathbb{Q}_k,\mathcal{A}^n_{\mu})$
	defined as follows (see also the beginning of Section 4)
	\[\mathrm{Sel}(\mathbb{Q}_k,\mathcal{A}^n_{\mu})=\mathrm{Ker}(H^1(\mathbb{Q}_k,\mathcal{A}^n_{\mu}\otimes_{A}\tilde{A}^{\ast})\to\prod_{\mathfrak{q}\nmid p} H^1(I_\mathfrak{q},\mathcal{A}^n_{\mu}\otimes_A\tilde{A}^{\ast})\times H^1(I_\mathfrak{p},(\mathcal{A}^n_{\mu}/F^1\mathcal{A}^n_{\mu})\otimes_A\tilde{A}^{\ast}))\]
	where $(F_j\mathcal{A}^n_\mu)_j$ is the filtration on $\mathcal{A}^n_\mu$ (restricted to the decomposition group $\Gamma_\mathfrak{p}$ over the unique prime $\mathfrak{p}$ of $\mathbb{Q}_k$ above $p$) derived from the $p$-ordinarity assumption on the Hida family $\mu$. Then our main result in this paper can be summarized as: (1) showing the cotorsionness of the Selmer groups $\mathrm{Sel}(\mathbb{Q}_\infty,\mathcal{A}^n(\rho_\mu))$ assuming the non-vanishing of their $\mathcal{L}$-invariants and some other technical hypotheses; (2) interpreting the Selmer groups in terms of congruence ideals.
	
	To state our result, we need some notations. 
	In Section 2, we will define a rank $n$ 
	unitary definite algebraic group $G_{n/\mathbb{Q}}$ 
	defined over $\mathbb{Q}$ for an integer $n\geq2$. 
	Write $\Pi^{G_n}$ for the base change (packet) of $\pi$ under the Langlands functorialities
	$\mathrm{GL}_{2/\mathbb{Q}}\leadsto\mathrm{GL}_{n/\mathbb{Q}}\leadsto G_{n/\mathbb{Q}}$ 
	which correspond to the homomorphisms ${^L\mathrm{GL}_{2/\mathbb{Q}}\xrightarrow{\mathrm{Sym}^{n-1}}
	 \ ^L\mathbb{Q}_{n/\mathbb{Q}}\to\ ^LG_{n/\mathbb{Q}}}$
	of their $L$-groups. 
	We can construct the localized big Hida-Hecke algebra $\mathbf{T}^{G_n}_k$,
	resp. $\mathbf{T}^\circ_k$, of auxiliary level $N$ over
	$G_{n/\mathbb{Q}_k}:=G_{n/\mathbb{Q}}\times_\mathbb{Q}\mathbb{Q}_k$, 
	resp. $\mathrm{GL}_{2/\mathbb{Q}_k}$. 
	Assuming the Langlands functorialities $\mathrm{GL}_{2/\mathbb{Q}_k}\leadsto\mathrm{GL}_{n/\mathbb{Q}_k}
	\leadsto G_{n/\mathbb{Q}_k}$ for each $k$, 
	we have morphisms of algebras $\mathbf{T}^{G_n}_k\to\mathbf{T}^\circ_k$. 
	Tensoring each of these Hecke algebras with $\tilde{A}$ over the corresponding weight Iwasawa algebra, we get $\tilde{\mathbf{T}}^{G_n}_k$, 
	resp. $\tilde{\mathbf{T}}^\circ_k$. 
	The above morphisms induce the following
	morphisms of algebras
	$\tilde{\theta}^{G_n}_k
	\colon\tilde{\mathbf{T}}^{G_n}_k\to\tilde{\mathbf{T}}^\circ_k$,
	$\tilde{\lambda}^\circ_k\colon\tilde{\mathbf{T}}^\circ_k\to\tilde{A}$. 
	To each $\tilde{\theta}^{G_n}_k$, 
	we can associate a congruence ideal $\mathfrak{c}_{\tilde{\theta}^{G_n}_k}$ 
	which is an ideal in $\tilde{\mathbf{T}}^\circ_k$ 
	(see Definition \ref{congruenceideal}).
	For the Galois representation
	$\rho^{G_n}_\mu=\mathrm{Sym}^{n-1}\rho_\mu$, 
	we can define an ideal
	$\mathcal{L}(\mathbb{Q},\mathrm{Ad}(\rho^{G_n}_\mu))$, 
	which is an element in the fraction field of $\tilde{A}$
	(see Definition \ref{definitionofL}). 
	We introduce three conditions on the residual Galois representation
	$\overline{\rho}^{G_n}_\pi=\mathrm{Sym}^{n-1}\overline{\rho}_\pi$:
	$\textbf{Big}(\overline{\rho}^{G_n}_\pi)$, 
	$\textbf{Dist}(\overline{\rho}^{G_n}_\pi)$
	and $\textbf{RegU}(\overline{\rho}^{G_n}_\pi)$ 
	(see Condition \ref{condition} for their precise formulations). 
	These three conditions can altogether be roughly stated as: 
	the image of the representation $\overline{\rho}^{G_n}_\pi$ is big, 
	its restriction to the decomposition group $\Gamma_p$ over $p$ is upper triangular with distinct diagonal entries 
	while its restriction to the inertia subgroups $I_q$ for any $q|N$ is non-trivial.
	
	The following results are proved in 
	\cite[Corollary 3.5, Corollary 3.9, Section 6]{HidaTilouine}:
	\begin{theorem}[\cite{HidaTilouine}]
		Assume $\textbf{Big}(\overline{\rho}^{G_n}_\pi)$,
		$\textbf{Dist}(\overline{\rho}^{G_n}_\pi)$,
		$\textbf{RegU}(\overline{\rho}^{G_n}_\pi)$. 
		Then the (fractional) congruence ideal 
		$\tilde{\lambda}^\circ_0(\mathfrak{c}_{\tilde{\theta}^{G_n}_0})/\tilde{\lambda}^\circ_0(\mathfrak{c}_{\tilde{\theta}^{G_{n-1}}_0})$ in $\tilde{A}$
		is generated by the characteristic element of the dual Selmer group
		$\mathrm{Sel}(\mathbb{Q},\mathcal{A}^{n-1}(\rho_\mu))^\ast$.
	\end{theorem}
	
	Then our result is a generalization of the above result
	in the sense that we consider a cyclotomic tower
	$\mathbb{Q}_\infty/\mathbb{Q}$
	instead of a single field $\mathbb{Q}$.
	\begin{theorem}
		We write $B_k=\tilde{A}[\mathrm{Gal}(\mathbb{Q}_k/\mathbb{Q})]$,
		$B=B_0$ and $B_\infty=\tilde{A}[[S]]$.
		\begin{enumerate}
			\item Assume
			$\textbf{Big}(\overline{\rho}_\pi)$, 
			$\textbf{Dist}(\overline{\rho}_\pi)$, 
			$\textbf{RegU}(\overline{\rho}_\pi)$,
			$\mathcal{L}(\mathbb{Q},\mathrm{Ad}(\rho_\mu))\neq0$ and $B$ regular.
			Then the dual Selmer group
			$\mathrm{Sel}(\mathbb{Q}_\infty,\mathcal{A}^1_\mu)^\ast$ is a finitely generated torsion $B_\infty$-module. 
			Its characteristic power series $L^{alg}_p(\mathcal{A}^1_\mu,S)$ has a trivial zero at $S=0$ of order $1$.
			Moreover, for any $k\geq0$ and any prime ideal $P$
			in $B_k$ of height one,	
			the ideal $\mathfrak{c}_{\tilde{\lambda}^\circ_k}(B_k)_P$ 
			in the localization $(B_k)_P$
			is generated by the Fitting ideal of $\mathrm{Sel}(\mathbb{Q}_k,\mathcal{A}^1(\rho_\mu))^\ast$ 
			in $B_k$
			(up to a factor a power of $p$ if $k>0$).
			Equivalently, it is generated by the image in $B_k$ of the element
			$L^{alg}_p(\mathcal{A}^1(\rho_\mu),S)/(SI^{\mathrm{GL}_2}_k)$
			(up to a factor a power of $p$ if $k>0$).
			\item Assume $\textbf{Big}(\overline{\rho}^{G_n}_\pi)$,
			$\textbf{Big}(\overline{\rho}^{G_n}_\pi)$,
			$\textbf{RegU}(\overline{\rho}^{G_n}_\pi)$,
			$\mathcal{L}(\mathbb{Q},\mathrm{Ad}(\rho^{G_n}_\mu))\neq0$, $\mathcal{L}(\mathbb{Q},\mathrm{Ad}(\rho^{G_{n-1}}_\mu))\neq0$, and $B$ regular. 
			Then the dual Selmer group $\mathrm{Sel}(\mathbb{Q}_\infty,\mathcal{A}^{n-1}(\rho_\mu))^\ast$ 
			is a finitely generated torsion $B_\infty$-module. 
			Its characteristic power series $L^{alg}_p(\mathcal{A}^{n-1}(\rho_\mu),S)$ has a trivial zero at $S=0$ of order $1$. 
			Moreover, 
			for any $k\geq0$ and any prime ideal 
			$P$ in $B_k$ of height one,
			the (fractional) ideal
			$\tilde{\lambda}^\circ_k(\mathfrak{c}_{\tilde{\theta}^{G_n}_k})
			/\tilde{\lambda}^\circ_k(\mathfrak{c}_{\tilde{\theta}^{G_{n-1}}_k})(B_k)_P$
			is generated by the Fitting ideal of
			$\mathrm{Sel}(\mathbb{Q}_k,\mathcal{A}^{n-1}(\rho_\mu))^\ast$ 
			over $B_k$
			(up to a factor a power of $p$ if $k>0$).
			Equivalently, it is generated by the image in $B_k$ of the element
			$L^{alg}_p(\mathcal{A}^{n-1}(\rho_\mu),S)/S
			\times I^{G_{n-1}}_k/I^{G_n}_k$
		    (up to a factor a power of $p$ if $k>0$).
		\end{enumerate}
	\end{theorem}
	
	See Theorem \ref{pL1} and Theorem \ref{pLn} for precise statements. 
	See Definition \ref{definitionofL} for the elements 
	$I^{\mathrm{GL}_2}_k$, $I^{G_n}_k$, 
	which are related to the $\mathcal{L}$-invariants
	of $\mathrm{Ad}(\rho^{G_n}_\mu)$.

	As we have mentioned, 
	the congruence ideals $\mathfrak{c}_{\tilde{\lambda}^\circ_k}$, 
	$\tilde{\lambda}^\circ_k(\mathfrak{c}_{\tilde{\theta}^{G_n}_k})/\tilde{\lambda}^\circ_k(\mathfrak{c}_{\tilde{\theta}^{G_{n-1}}_k})$
	are conjectured of automorphic nature. 
	In other words, these congruence ideals are (conjecturally) related to the special values of the complex $L$-functions $L(\mathcal{A}^{n-1}(\rho_\pi),s)$ of the Galois representation $\mathcal{A}^{n-1}(\rho_\pi)$.
	One can get some evidence for this from a result of Hida in \cite[Corollary 5.8(2)]{HidaArithmetic}.
	
	\begin{theorem}
		We assume $\textbf{Big}(\overline{\rho}_\pi)$ and $p\nmid 6\varphi(N)$.
		Suppose that $\tilde{A}\to\tilde{A}/P_{\pi}$ is the specialization to $\pi$, then the ideal $\mathfrak{c}_{\tilde{\lambda}^\circ_0}(\mathrm{mod}\ P_\pi)$ in $\tilde{A}/P_\pi$ is generated by
		\[W (\pi)N^{ (a+1)/2}\frac{L (\mathrm{Ad} (\pi),1)}{\Omega (+,\pi)\Omega (-,\pi)}\]
		where $W(\pi)$ is a non-zero complex number measuring the difference between the actions of the involution $\begin{pmatrix}
		0 & -1 \\ N & 0
		\end{pmatrix}$ and the complex conjugation on the representation $\pi$ (see \cite[end of p.31]{HidaArithmetic} for the precise definition of $W(\pi)$) and $\Omega(\pm,\pi)$ are fundamental periods of $\pi$.
	\end{theorem}

	See Remark \ref{RemarkOnCongruenceIdealAndSpecialLValue} for a detailed discussion on relating congruence ideals to special $L$-values.
	
	Now let's turn to $\mathcal{L}$-invariants. In the above theorem, we have in fact obtained some information on the trivial zero $S=0$ of the characteristic power series $L^{alg}_p(\mathcal{A}^{n-1}(\rho_\mu),S)$:
	it is of order $1$ and the first term of the Taylor expansion of $L^{alg}_p (\mathcal{A}^{n-1}(\rho_\mu),S)$ at $S=0$ is
	$\mathcal{L} (\mathbb{Q},\mathcal{A}^{n-1}(\rho_\mu))\mathrm{char}(\mathrm{Sel}(\mathbb{Q},\mathcal{A}^{n-1}(\rho_\mu))^\ast)$ 
	where
	\[
	\mathcal{L}(\mathbb{Q},\mathcal{A}^{n-1}(\rho_\mu))
	=I^{G_n}_0/I^{G_{n-1}}_0.
	\]
	
	In \cite{HarronJorza}, the $\mathcal{L}$-invariants 
	$\mathcal{L}^{\mathrm{Gr}} (\mathcal{A}^n(\rho_\pi))$ 
	of Greenberg of these Galois representations $\mathcal{A}^j(\rho_\pi)$ are calculated.
	So it is natural to compare our $\mathcal{L} (\mathbb{Q},\mathcal{A}^n(\rho_\mu))$ to the formula in \cite{HarronJorza}
	(for the ordinary case, the $\mathcal{L}$-invariant
	$\mathcal{L}^{\mathrm{Gr}} (\mathbb{Q}, \mathcal{A}^n(\rho_\mu))$ is
	just the one of Greenberg in \cite{GreenbergTrivialZero},
	which can be easily deduced from the general formula in \cite{HarronJorza}). Our second result is the following
	(see Theorem \ref{compareLinvariant}).
	\begin{theorem}
		Let $n\geq 1$ be an integer. 
		Assume $\textbf{Big}(\overline{\rho}^{G_n}_\pi)$,
		$\textbf{Dist}(\overline{\rho}^{G_n}_\pi)$,
		$\textbf{RegU}(\overline{\rho}^{G_n}_\pi)$,
		${\mathcal{L}(\mathbb{Q},\mathrm{Ad}(\rho^{G_{n-1}}_\mu))\neq0}$
		and $\tilde{A}$ regular.
		Let $z_0$ be a specialization of $A_0$ to the automorphic form $\pi$.
		Then
		\[z_0(\mathcal{L} (\mathbb{Q},\mathcal{A}^n(\rho_\mu)))=\mathcal{L}^{\mathrm{Gr}} (\mathcal{A}^n(\rho_\mu)).\]
	\end{theorem}
	
	The paper is organized as follows.
	In Section 2 we recall and generalize some results on the modularity lifting on definite unitary groups in \cite{HidaTilouine}.
	In Section 3 we review some results on the functoriality on the Hecke side and the Galois side with respect to the base change.
	In Section 4 we study the cotorsionness of the Selmer groups $\mathrm{Sel} (\mathbb{Q}_{\infty},\mathcal{A}^n(\rho_\mu))$ and get a more precise structure on their characteristic power series.
	In Section 5 we compare our 
	$\mathcal{L} (\mathbb{Q},\mathcal{A}^n(\rho_\mu))$ 
	to the $\mathcal{L}$-invariants 
	$\mathcal{L}^{\mathrm{Gr}} (\mathbb{Q},\mathcal{A}^j(\rho_\mu))$ 
	in \cite{HarronJorza}.
	\\\
	
	\paragraph*{\textbf{Notations}}
	We fix an odd prime number $p>2$,
	an integer $n>0$ prime to $p$.
	The absolute Galois group of a number field $L$ is denoted 
	by $\Gamma_L$,
	the decomposition group of $L$ at a prime $\mathfrak{q}$ of $L$ 
	is denoted by
	$\Gamma_{\mathfrak{q}}$ 
	and the inertia subgroup is denoted by
	$I_{\mathfrak{q}}$.
	Write $L_1^+/L^+$ for 
	a $p$-cyclic Galois extension of totally real number fields 
	of degree $d$, 
	both totally ramified over $p$, 
	unramified outside $p$, 
	whose $p$-adic completions do not contain the $p$-th roots of unity.
	Denote by $\Delta$ the Galois group $\mathrm{Gal} (L_1^+/L^+)$.
	Let $d=[L^+\colon\mathbb{Q}]$. 
	Let $H$ be a commutative ring 
	or a module over some ring with an action of a group $\Delta$,
	then we write $H_{\Delta}$ 
	for the quotient $H/H\sum_{\delta\in\Delta} (\delta-1)H$ 
	if $H$ is a ring,
	and $H/\sum_{\delta\in\Delta} (\delta-1)H$ if $H$ is a module.
	The Galois groups of the $p$-cyclotomic extension are given as in the graph below.
	We fix a generator 
	$\gamma$ of $\Gamma$ 
	and $\gamma_k=\gamma^{p^k}$ 
	is thus a generator of $\Gamma_k$.
	
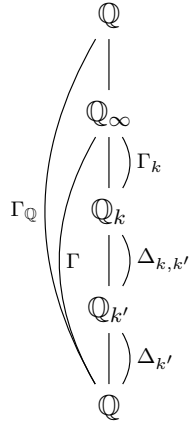
	\[
	\begin{tikzcd}
	\overline{\mathbb{Q}} \arrow[d,dash]\arrow[dddd,"\Gamma_{\mathbb{Q}}"',bend right, dash]\\
	\mathbb{Q}_\infty \arrow[d, dash]\arrow[d,"\Gamma_k", bend left, dash]\arrow[ddd,"\Gamma", bend right, dash]\\
	\mathbb{Q}_k \arrow[d,dash]\arrow[d,"\Delta_{k,k'}", bend left, dash] \\
	\mathbb{Q}_{k'} \arrow[d,dash]\arrow[d,"\Delta_{k'}", bend left, dash] \\
	\mathbb{Q}
	\end{tikzcd}
	\]

	\section{$R=T$ theorems for definite unitary groups}
	In this section we
	review some results of \cite[Section 2]{HidaTilouine}.
	
	We fix a positive integer $n$ prime to $p$, 
	a square-free integer $N=q_{1} q_{2} \ldots q_{s}$ prime to $p$, 
	where we require
	$s>1$ if $n=2m$ with $m$ an odd integer and $d$ also an odd integer,
	otherwise we just require $s>0$.
	Let $F$ be an imaginary quadratic field
	such that $F/\mathbb{Q}$ is unramified outside
	$Np$, $p$ splits in $F$, $q_{1}$ splits in $F$
	as $\mathfrak{q}_{1} \mathfrak{q}_{1}^{c}$, 
	and $q_{2}$ is inert in $F$ 
	(if $s\geq2$). 
	Let $K$ be a sufficiently large $p$-adic field and $\mathcal{O}$ its integer ring,
	$\mathbb{F}$ its residue field.
	We then follow \cite[Section 2.1]{HidaTilouine} 
	to construct a definite unitary group as follows.
	Let $D$ be a central simple algebra over $F$ of rank $n$ whose ramification set $S_D$ consists of primes above $q_1$
	(i.e. $S_D=\{\mathfrak{q}_1,\mathfrak{q}_1^c\}$).
	By \cite[(2.3) and Lemme 2.2]{Clozel1991}, one has the following observation
	\begin{enumerate}
		\item If $n$ is an odd multiple of $2$,
		then there is an involution of second kind $\dagger$ on $D$ such that $\dagger$ is positive definite at $\infty$
		and the associated unitary group $U (D,\dagger)$ 
		over $\mathbb{Q}$ is quasi-split at all inert places 
		except possibly the place $q_2$;
		\item Otherwise, there is an involution of second kind $\dagger$ on $D$
		such that $\dagger$ is positive definite at $\infty$ and the associated unitary group $U (D,\dagger)$
		is quasi-split at all inert places.
	\end{enumerate}
	
	We choose one such involution $\dagger$ as in the observation 
	and write $G'=G'_n=U (D,\dagger)$.
	
	Let $L^+$ be a totally real number field 
	totally ramified at $p$ and
	$L=L^+F$.
	The unique prime of $L^+$ over $p$ is denoted by $\mathfrak{p}_{L^+}$, 
	the primes of $L$ over $p$ are 
	then denoted by $\mathfrak{P}_{L^+},\mathfrak{P}_{L^+}^c$.
	We often drop the subscript $L^+$
	when there is no ambiguity.
	We write $L^+_\mathfrak{p}$, resp. $L_\mathfrak{P}$,
	for the completions of $L^+$, 
	resp. $L$,
	over the prime $\mathfrak{p}$, resp. $\mathfrak{P}$. 
	We denote $\mathcal{O}_\mathfrak{p}$, resp. $\mathcal{O}_\mathfrak{P}$ ,
	the integer rings of $L^+_\mathfrak{p}$, resp. $L_\mathfrak{P}$. 
	We have natural isomorphisms 
	$L^+_\mathfrak{p}\simeq L_\mathfrak{P}$ 
	and $\mathcal{O}_\mathfrak{p}\simeq\mathcal{O}_\mathfrak{P}$.
	If $L^+=\mathbb{Q}_k$, 
	we write
	$\mathfrak{p}_k=\mathfrak{p}_{\mathbb{Q}_k}$
	and $\mathfrak{P}_k=\mathfrak{P}_{\mathbb{Q}_k}$,
	$\mathfrak{P}_k^c=\mathfrak{P}_{\mathbb{Q}_k}^c$.
	Now we set 
	\[
	G=G_n=G'\times_\mathbb{Q}L^+,
	\]
	the base change of the algebraic group $G'$ from $\mathbb{Q}$ to $L^+$.
	
	Let $B_n=TN=T_nN$ be the standard Borel subgroup of $\mathrm{GL}_n$,
	where $T=T_n$ is the standard maximal torus in 
	$\mathrm{GL}_n$ and $N$ is the nilpotent radical of $B_n$. 
	Moreover, let $T^{ss}$ be 
	the intersection $T\bigcap\mathrm{SL}_n$,
	or equivalently, the semi-simple part of $T$.
	We fix a level subgroup
	$U^{\mathfrak{p}}=\prod_{\mathfrak{q}|N}U_\mathfrak{q}\times U^{Np}\subset G (\mathbb{A}_{L^+})$ of Iwahori type
	(i.e. for each prime $\mathfrak{q}$ of $L^+$ dividing $N$, 
	$U_\mathfrak{q}$ is the standard Iwahori subgroup of $G_\mathfrak{q}$ 
	if $G$ is quasi-split at $\mathfrak{q}$ 
	and is a minimal parahoric subgroup of $G_\mathfrak{q}$ otherwise).
	We set $U_{\mathfrak{p}}$ to be the image of
	$\mathrm{GL}_n (\mathcal{O}_{\mathfrak{p}})$ under the isomorphism
	$\mathrm{GL}_n (L^+_{\mathfrak{p}})\simeq G (L^+_{\mathfrak{p}})$
	(we use this isomorphism to identify $U_{\mathfrak{p}}$ as a subgroup of
	$\mathrm{GL}_n (L^+_{\mathfrak{p}})$).
	We write $U=U^{\mathfrak{p}}\times U_{\mathfrak{p}}$.
	Let $\mathrm{Iw}\subset U_{\mathfrak{p}}$ be the Iwahori subgroup.
	For any integers $c\geq b\geq0$, we set $\mathrm{Iw}^{b,c}$ to be the set of matrices which are upper triangular modulo $\mathfrak{p}^c$ and are upper triangular with diagonal entries $ (1,1,\ldots,1,\ast)$ modulo $\mathfrak{p}^b$.
	We set $U^{b,c}=U^{\mathfrak{p}}\times \mathrm{Iw}^{b,c}$.
	A regular dominant weight for the group $G$ is an element 
	$\lambda=(\lambda_\tau)_\tau\in 
	(\mathbb{Z}^n/\mathbb{Z}(1,1,\ldots,1))^{\{L\to K\}}$
	such that 
	for each embedding $\tau\colon L\to K$,
	$\lambda_{\tau,1}\geq\lambda_{\tau,2}\geq\ldots\geq\lambda_{\tau,n}$
	where 
	$\lambda_\tau
	=(\lambda_{\tau,1},\lambda_{\tau,2},\ldots,\lambda_{\tau,n})$,
	moreover, for each $\tau$, 
	there is at least one strict inequality
	$\lambda_{\tau,i}>\lambda_{\tau,i+1}$.
	Let $W_{\lambda_{\tau}} (\mathcal{O})$ be 
	the algebraic induction of 
	$w_0 (\lambda_{\tau,1}-\lambda_{\tau,n},\lambda_{\tau,2}-\lambda_{\tau,n},\ldots,0)$ 
	from $B_n$ to the whole group $\mathrm{GL}_n$,
	where $w_0$ is the longest element of the Weyl group of $\mathrm{GL}_n$ (see\cite{PoloTilouine}).
	Write then $W_{\lambda} (\mathcal{O})=\otimes_{\tau}W_{\lambda_{\tau}} (\mathcal{O})$.
	We define the $\mathcal{O}$-module $S_{\lambda} (U^{b,c};\mathcal{O})$ of cusp forms of level
	$U^{b,c}$ for $G$ to be the set of functions
	$f\colon G (L^+)\backslash G (\mathbb{A}_{L^+}^f)\to W_{\lambda} (\mathcal{O})$
	such that for any $u\in U^{b,c}$ and $x\in G (L^+)\backslash G (\mathbb{A}_{L^+}^f)$,
	$f (xu)=u^{-1}_{\mathfrak{p}}f (x)$.
	For $c>0$, we define $h_{\lambda} (U^{b,c};\mathcal{O})$ to be the $\mathcal{O}$-algebra of endomorphisms of $S_{\lambda} (U^{b,c},\mathcal{O})$ generated by the Hecke operators of the following three types
	(for any prime $\mathfrak{Q}$ of $L$ with uniformizer $\varpi_{\mathfrak{Q}}$,
	the matrix $\alpha (\varpi_{\mathfrak{Q}},i)$ is the diagonal matrix of size $n\times n$ with diagonal entries
	$ (\varpi_{\mathfrak{Q}},\ldots,\varpi_{\mathfrak{Q}},1,\ldots,1)$ where exactly the first $i$ terms are $\varpi_{\mathfrak{Q}}$ ($1\leq i\leq n$))
	\begin{enumerate}
		\item If the prime $\mathfrak{Q}$ is prime to $Np$,
		$T_{\mathfrak{Q},i}=[U^{b,c}\alpha (\varpi_{\mathfrak{Q}},i)U^{b,c}]_{\lambda}$ 
		($1\leq i\leq n$);
		\item For the prime $\mathfrak{P}$,
		$U_{\mathfrak{P},i}= (w_0\lambda) (\alpha (\varpi_{\mathfrak{P}},i))^{-1}[U^{b,c}\alpha (\varpi_{\mathfrak{P}},i)U^{b,c}]_{\lambda}$ 
		($1\leq i\leq n-1$);
		\item For any $u\in T^{ss} (\mathcal{O}_{\mathfrak{p}})$, $\langle u\rangle_\lambda=[U^{b,c}uU^{b,c}]_{\lambda}$.
	\end{enumerate}

    Here the operator $[U^{b,c}aU^{b,c}]_{\lambda}$ is defined as:
    write the coset decomposition $U^{b,c}aU^{b,c}=\bigsqcup_ja_jU^{b,c}$,
    then for any $s\in S_{\lambda} (U^{b,c};\mathcal{O})$ and $x\in G (L^+)\backslash G (\mathbb{A}^f_{L^+})$,
    $ ([U^{b,c}aU^{b,c}]_{\lambda}s) (x)=\sum_j (a_j)_{\mathfrak{p}} (s (xa_j))$
    (see \cite[Section 2.3]{Geraghty}).

	We write $eh=eh_{L^+}=\lim\limits_{\leftarrow c}eh_{\lambda} (U^{c,c};\mathcal{O})$ for the big Hida-Hecke algebra,
	where $e$ is the ordinary idempotent corresponding to $\prod_iU_{\mathfrak{P},i}$.
	It is reduced and is a finitely generated torsion-free algebra over the weight algebra $\Lambda=\Lambda_{L^+}=\mathcal{O}[[T^{ss} (\mathfrak{P})]]$ 
	where
	$T^{ss} ({\mathfrak{P}})=\mathrm{Ker} (T^{ss} (\mathcal{O}_{\mathfrak{P}})\to T^{ss} (\mathcal{O}_{\mathfrak{P}}/\mathfrak{P}))$
	(see \cite[Lemma 2.4.4 and Corollary 2.5.4]{Geraghty}).
	The $\Lambda$-algebra structure on $eh$ is via the diamond operators $\langle u\rangle_0$ for $u\in T^{ss} (\mathfrak{P})$
	(see \cite[Definition 2.6.2]{Geraghty}).
	Note that $[L^+_{\mathfrak{p}}\colon \mathbb{Q}_p]=d$, we choose a basis $u'_1,u'_2,\ldots,u'_d$
	of the free $\mathbb{Z}_p$-module $\mathcal{O}_{\mathfrak{p}}$.
	If we write $u_i=1+\varpi_{\mathfrak{p}}u'_i\in\mathcal{O}_{\mathfrak{p}}^{\times}$,
	then one can show that there is an isomorphism of $p$-profinite $\mathcal{O}$-algebras
	\[\Lambda\simeq\mathcal{O}[[X_1^{ (1)},X_2^{ (1)},\ldots,X_{n}^{ (1)},\ldots,X_{n}^{ (d)}]]/(\prod_{j=1}^nX^{(1)}_j=1,\ldots,\prod_{j=1}^nX^{(d)}_j=1)\]
	by taking the diagonal matrix
	$\mathrm{diag} (1_{j-1},u_i,1_{n-1-j},u_i^{-1})$ to the element $(1+X^{ (i)}_j)/(1+X^{(i)}_n)$.

	We write $\mathcal{G}_n$ for the algebraic group scheme over
	$\mathrm{Spec}(\mathbb{Z})$ 
	defined in \cite[Section 2.1]{ClozelHarrisTaylor}: 
	$\mathcal{G}_n$ is the semi-direct product $(\mathrm{GL}_1\times\mathrm{GL}_n)\rtimes\mathbb{Z}/2\mathbb{Z}$ of the group
	$\mathrm{GL}_1\times\mathrm{GL}_n$ 
	by $\mathbb{Z}/2\mathbb{Z}=\{1,j\}$ with the group action of $\mathbb{Z}/2\mathbb{Z}$ 
	on $\mathrm{GL}_1\times\mathrm{GL}_n$ given by
	$j(x,g)j^{-1}=(x,x(g^{-1})^t))$,
	where $(g^{-1})^t$ is the transpose of the matrix $g^{-1}\in\mathrm{GL}_n$.
	We write the projection to the first, 
	resp. second factor, of $\mathcal{G}_n$ as $\mathrm{pr}_1\colon\mathcal{G}_n\to\mathrm{GL}_1$, 
	resp. $\mathrm{pr}_2\colon\mathcal{G}_n\to\mathrm{GL}_n$.
	
	Let $\Pi$ be a cusp $p$-ordinary automorphic representation of $G$ which is Steinberg at all primes $w$ of $L$ dividing $N$. 
	Recall the following result of D.Geraghty.
	\begin{theorem}[Proposition 2.7.2, Corollary 2.7.8 of \cite{Geraghty}]\label{ExistenceOfGaloisRep}
		Let $\lambda$ be a regular dominant
		weight for $G$. 
		Suppose $\Pi$ 
		is an irreducible constituent of the 
		$G(\mathbb{A}_{L^+}^{\infty,p})\times \mathrm{Iw}$
		-representation
		$S_\lambda(\overline{\mathbb{Q}}_p)$
		with
		${\Pi^{U^{0,1}}\bigcap eS_\lambda(U^{0,1},\mathcal{O})\neq\{0\}}$.
		Let $u_{\mathfrak{P},i}$ 
		denote the eigenvalue of 
		$U_{\mathfrak{P},i}$ on 
		$\Pi^{U^{0,1}}\bigcap eS_\lambda(U^{0,1},\mathcal{O})$ for $1\leq i\leq n-1$.
		Then there exists a continuous semi-simple representation
		\[\rho_\Pi\colon\Gamma_{L^+}\to\mathcal{G}_n(\overline{\mathbb{Q}}_p)\]
		such that
		\begin{enumerate}
			\item for any prime $\mathfrak{q}$ in $L^+$ dividing $Np$ and splitting in $L$ as $\mathfrak{Q}\mathfrak{Q}^c$, 
			$\rho_\Pi$ is unramified at $\mathfrak{q}$ 
			and the characteristic polynomial of  $\rho_\Pi(\mathrm{Frob}(\mathfrak{Q}))$ is
			\[X^n-T_{\mathfrak{Q},i}X^{n-1}+\ldots+(-1)^iN(\mathfrak{Q})^{i(i-1)/2}T_{\mathfrak{Q},i}X^{n-i}+\ldots+(-1)^nN(\mathfrak{Q})^{n(n-1)/2}T_{\mathfrak{Q},n}\]
			where $N(\mathfrak{Q})$ is the norm of the prime ideal $\mathfrak{Q}$;
			\item The representation $\rho_\Pi$ is self-dual conjugate: $\rho_\Pi^c\simeq\rho_\Pi^\vee\chi^{1-n}$;
			\item For any prime $\mathfrak{q}$ in $L^+$ which is inert in $L$ and prime to $Np$, $\rho_\Pi$ is unramified at $\mathfrak{q}$;
			\item If moreover the weight $\lambda$ is regular (this means that $\lambda_{\tau,1}>0$ for some embedding $\tau\colon L\to K$), then the representation $\rho_\Pi$ is crystalline and ordinary at $\mathfrak{P}$ and $\mathfrak{P}^c$. More precisely, the restriction representation $\rho_\Pi|_{\Gamma_\mathfrak{P}}$ is conjugate to an upper triangular representation
			\[\begin{pmatrix}
			\psi_{\mathfrak{P},1} & \ast & \ldots  & \ast \\
			0 & \chi^{-1}\psi_{\mathfrak{P},2} & \ldots  & \ast \\
			\vdots  & \vdots  & \ddots  & \vdots  \\
			0 & 0 & 0 & \chi^{- (n-1)}\psi_{\mathfrak{P},n}
			\end{pmatrix}\]
			where $\psi_{\mathfrak{P},i}\circ\mathrm{Art}_\mathfrak{P}\colon L_\mathfrak{P}^\times\to K$ is characterized by: 
			(1) for any $x\in \mathcal{O}_\mathfrak{P}^\times$,
			$\psi_{\mathfrak{P},i}\circ\mathrm{Art}_\mathfrak{P}(x)=\prod_{\tau\colon L\to K}\tau(x)^{-\lambda_{\tau,n+1-i}}$, 
			(2) $\psi_{\mathfrak{P},i}\circ\mathrm{Art}_\mathfrak{P}(\varpi_\mathfrak{P})=u_{\mathfrak{P},i}/u_{\mathfrak{P},i-1}$ 
			(here we put $\lambda_{\tau,n}=0$ for any $\tau\colon L\to K$, $u_{\mathfrak{P},0}=1=u_{\mathfrak{P},n}$ 
			and
			$\mathrm{Art}_{\mathfrak{p}}$ is the local Artin reciprocity symbol). 
			Similar for $\mathfrak{P}^c$.
			
		\end{enumerate}
	\end{theorem}
	
	\begin{remark}
		Note that by 4) of the above theorem, for any
		$x\in\mathcal{O}_\mathfrak{P}^\times$,
		$\psi_{\mathfrak{P},1}\circ\mathrm{Art}_\mathfrak{P}(x)=1$. 
		This shows that the character $\psi_{\mathfrak{P},1}$ 
		is unramified when restricted to $\Gamma_\mathfrak{P}$.
	\end{remark}
	
	We recall the notion of liftings/deformations of residue Galois representations.
	
	\begin{definition}\label{deformation}
		Let $\mathrm{CNL}_{\mathcal{O}}$ be the category of complete noetherian local $\mathcal{O}$-algebras $B$ with residue field $B/\mathfrak{m}_B=\mathbb{F}$.
		For an object $B$ in $\mathrm{CNL}_{\mathcal{O}}$,
		a continuous homomorphism $r\colon \Gamma_{L^+}\to \mathcal{G}_n (B)$ is called a \textbf{lifting} of $\bar{\rho}_{\Pi}$
		if $r\equiv\bar{\rho}_{\Pi}\pmod{\mathfrak{m}_B}$.
		
		Two such liftings $r,r^{\prime}\colon \Gamma_{L^+}\to\mathcal{G}_n (B)$ are said to be equivalent
		if there exists some matrix ${g\in 1+\mathfrak{m}_BM_n (B)}$ such that $r^{\prime}=g r g^{-1}$.
	\end{definition}

    We write $\mathbf{T}=\mathbf{T}_{L^+}$ for the localization of $eh$ at the maximal ideal associated to the residue Galois representation $\overline{\rho}_{\Pi}$. Then we have the following result on the liftings of $\overline{\rho}_\Pi$ to $\mathcal{G}_n(\mathbf{T})$.
    \begin{theorem}[Proposition 2.7.4 of \cite{Geraghty}, Proposition 2.8 of \cite{HidaTilouine}]\label{ExistenceOfBigHeckeAlgebraValuedGaloisRep}
    	The same notations and assumptions as the above theorem. 
    	Assume moreover that $\overline{\rho}_\Pi$ is absolutely irreducible.
    	Then there exists a unique equivalence class of continuous liftings
    	\[\rho_\mathbf{T}\colon\Gamma_{L^+}\to\mathcal{G}_n(\mathbf{T})\]
    	of $\overline{\rho}_\Pi$ such that
    	\begin{enumerate}
    		\item For any prime $\mathfrak{q}$ in $L^+$ splitting in $L$ as $\mathfrak{Q}\mathfrak{Q}^c$, $\rho_\mathbf{T}$ is unramified at $\mathfrak{Q}$ and $\mathfrak{Q}^c$. Moreover, the characteristic polynomial of $\rho_\mathbf{T}(\mathrm{Frob}_\mathfrak{Q})$ is 
    		\[X^n-T_{\mathfrak{Q},1}X^{n-1}+\ldots+(-1)^jN(\mathfrak{Q})^{j(j-1)/2}T_{\mathfrak{Q},j}X^{n-j}+\ldots+(-1)^nN(\mathfrak{Q})^{n(n-1)/2}T_{\mathfrak{Q},n}.\]
    		\item The matrix $\rho_\mathbf{T} (\mathrm{Frob}_{\mathfrak{P}})$
    		has the following characteristic polynomial 
    		\[P (\mathrm{Frob}_{\mathfrak{P}},X)=\prod_{j=1}^n (X-N (\mathfrak{P})^{j-1}\frac{u_{\mathfrak{P},j}}{u_{\mathfrak{P},j-1}}\prod_{\tau\colon L\to K}\varpi_{\mathfrak{P}}^{\lambda_{\tau,n+1-j}}).\]
    	\end{enumerate}
    \end{theorem}
	
	We next define a deformation functor $\mathcal{D}=\mathcal{D}_{L^+}$
	from $\mathrm{CNL}_{\mathcal{O}}$ to the category $\mathrm{Sets}$.
	\begin{definition}
		The functor $\mathcal{D}$ takes an object $A$ in $\mathrm{CNL}_{\mathcal{O}}$
		to the set of equivalence classes of liftings $r\colon \Gamma_{L^+}\to \mathcal{G}_n (A)$ of $\bar{\rho}_{\Pi}$
		satisfying the following conditions:
		\begin{enumerate}
			\item For each prime $v$ of $L^+$ dividing $N$, there exists $g_v\in 1+\mathfrak{m}_AM_n (A)$ such that for any $\sigma\in I_v$, $\mathrm{pr}_2\circ r (\sigma)=g_v \exp (t_p (\sigma)N_n) g_v^{-1}$
			where $N_n$ is the following matrix
			\[N_n=\begin{pmatrix}
			0 & 1 & 0 & \ldots \\
			0 & 0 & 1 & \ldots \\
			\ldots  & \ldots  & \ldots  & \ldots  \\
			0 & 0 & 0 & 1 \\
			0 & 0 & 0 & 0
			\end{pmatrix}.\]
			\item The representation $r$ is ordinary at $\mathfrak{p}$. More precisely, there exists some matrix $g\in 1+\mathfrak{m}_AM_n (A)$ such that
			$$\mathrm{pr}_2\circ r|_{\Gamma_{\mathfrak{P}}}=g\begin{pmatrix}
			\psi_{r,\mathfrak{p},1} & \ast & \ldots  & \ast \\
			0 & \chi^{-1}\psi_{r,\mathfrak{p},2} & \ldots  & \ast \\
			\ldots  & \ldots  & \ldots  & \ldots  \\
			0 & 0 & 0 & \chi^{- (n-1)}\psi_{r,\mathfrak{p},n}
			\end{pmatrix}g^{-1}$$
			where $\psi_{r,\mathfrak{p},j}$ is a lifting of $\overline{\psi}_{\mathfrak{p},j}$.
			Moreover, we require that $\psi_{r,\mathfrak{p},1}$ is unramified when restricted to
			$\Gamma_\mathfrak{P}$;
			\item $r$ is unramified outside $Np$.
		\end{enumerate}
	\end{definition}
  
    \begin{remark}
    	By \cite[Remark 2.2 and Proposition 2.8]{HidaTilouine}, (the equivalence class of) the Galois representation $\rho_\mathbf{T}$ lies in $\mathcal{D}(\mathbf{T})$.
    \end{remark}

	We will consider the following conditions which will be used throughout the paper.
	\begin{condition}\label{condition}
		\begin{enumerate}
			\item $\textbf{Big} (\overline{\rho}_{\Pi})$: 
			there is a subfield 
			$\mathbb{F}^{\prime}\subset\mathbb{F}$ 
			such that 
			${\mathrm{Sym}^{n-1}\mathrm{SL}_2
			(\mathbb{F}^{\prime})
			\subset
			\mathrm{Im} (\overline{\rho}_{\Pi})}$;
			\item $\textbf{Dist} (\overline{\rho}_{\Pi})$: 
			the characters on the diagonal of the residue  representation
			$\overline{\rho}|_{\Gamma_{\mathfrak{p}}}$ are mutually distinct;
			\item $\textbf{RegU} (\overline{\rho}_{\Pi})$: 
			for any prime $v$ of $L^+$ dividing $N$,
			the restriction of the residue Galois representation $\overline{\rho}_{\Pi}$
			to the inertia subgroup $I_{v}$ is regular unipotent modulo $p$,
		\end{enumerate}	
	\end{condition}

    \begin{remark}
    	It is easy to see that $\textbf{Big}(\overline{\rho}_\Pi)$ implies the absolute irreducibility of the Galois representation $\overline{\rho}_\Pi$. 
    \end{remark}
	
	The following lemma describes how these conditions change 
	if we vary the base field $L^+$ 
	or the representation $\overline{\rho}_{\Pi_{L^+}}$.
	\begin{lemma}\label{basechange}
		For the extension $L_1^+/L^+$, we have
		\begin{enumerate}
			\item $\textbf{Big} (\overline{\rho}_{\Pi})$ implies $\textbf{Big} (\overline{\rho}_{\Pi}|_{\Gamma_{L_1^+}})$;
			\item $\textbf{Dist} (\overline{\rho}_{\Pi})$ implies $\textbf{Dist} (\overline{\rho}_{\Pi}|_{\Gamma_{L_1^+}})$;
			\item $\textbf{RegU} (\overline{\rho}_{\Pi})$ implies $\textbf{RegU} (\overline{\rho}_{\Pi}|_{\Gamma_{L_1^+}})$;
			\item Suppose that $L^+=\mathbb{Q}$ and $\Pi$ is an ordinary cusp automorphic representation on $\mathrm{GL}_2 (\mathbb{A}_{L^+})$,
			then $\textbf{Dist} (\mathrm{Sym}^{n+1}\overline{\rho}_{\Pi})$ implies $\textbf{Dist} (\mathrm{Sym}^n\overline{\rho}_{\Pi})$.
		\end{enumerate}
	\end{lemma}
	\begin{proof}
		We write $\overline{\rho}=\overline{\rho}_{\Pi}$ and $\overline{\rho}_1=\overline{\rho}_{\Pi}|_{\Gamma_{L_1^+}}$.
		\begin{enumerate}
			\item Assume $\textbf{Big} (\overline{\rho})$.
			Note that $\mathrm{SL}_2 (\mathbb{F})$ is a simple group,
			so is $\mathrm{Sym}^{n-1}\mathrm{SL}_2 (\mathbb{F})$.
			$\mathrm{Im} (\overline{\rho}|_{\Gamma_{\mathfrak{p}_{L_1^+}}})$ is a normal subgroup of $\mathrm{Im} (\overline{\rho}|_{\Gamma_{\mathfrak{p}_{L^+}}})$,
			so we have the following division
			\[\# (\frac{\mathrm{Sym}^{n-1}\mathrm{SL}_2 (\mathbb{F})}{\mathrm{Sym}^{n-1}\mathrm{SL}_2 (\mathbb{F})\bigcap\mathrm{Im} (\overline{\rho}|_{\Gamma_{\mathfrak{p}_{L_1^+}}})})|\# (\frac{\mathrm{Im} (\overline{\rho}|_{\Gamma_{\mathfrak{p}_{L^+}}})}{\mathrm{Im} (\overline{\rho}|_{\Gamma_{\mathfrak{p}_{L_1^+}}})}).\]
			By comparing the cardinals of both sides,
			we conclude that
			$\mathrm{Sym}^{n-1}\mathrm{SL}_2 (\mathbb{F})\subset \mathrm{Im} (\overline{\rho}|_{\Gamma_{\mathfrak{p}_{L_1^+}}});$
			\item Assume $\textbf{Dist} (\overline{\rho})$.
			Then the characters of the diagonal entries of $\overline{\rho}|_{\Gamma_{\mathfrak{p}_{L^+}}}$ take value in $\mathbb{F}^{\times}$.
			So the problem is reduced to showing that any character $\Gamma_{\mathfrak{p}_{L_1^+}}\to\mathbb{F}^{\times}$ can be extended to
			$\Gamma_{\mathfrak{p}_{L^+}}\to\mathbb{F}^{\times}$ in at most one way.
			The number of extensions is in one-to-one correspondence with the cohomological group $\mathrm{H}^2 (\Delta,\mathbb{F}^{\times})$.
			This group vanishes because $\#\mathbb{F}^{\times}$ is prime to $\#\Delta$;
			\item Assume $\textbf{RegU} (\overline{\rho})$.
			For any prime $l$ dividing $N$, since $L_1^+/L^+$ is unramified over $l$,
			so for any prime $v$ of $L^+$ over $l$ and any prime $w$ of $L_1^+$ over $v$, the inertia subgroups $I_v=I_w$.
			Therefore it is clear that $\textbf{RegU} (\overline{\rho})$ implies $\textbf{RegU} (\overline{\rho}_1)$;
			\item Suppose the diagonal entries of $\overline{\rho}|_{\Gamma_{\mathfrak{p}_{L^+}}}$ are $\overline{\psi}_1,\overline{\psi}_2$.
			Then $\textbf{Dist} (\mathrm{Sym}^{n+1}\overline{\rho})$ implies that $ (\overline{\psi}_1/\overline{\psi}_2)^{ (n+1)!}\neq1$.
			Clearly this implies $ (\overline{\psi}_1/\overline{\psi}_2)^{n!}\neq1$,
			from which we deduce
			$\textbf{Dist} (\mathrm{Sym}^n\overline{\rho})$.
		\end{enumerate}
	\end{proof}
	\begin{proposition}
		Assume $\textbf{Big} (\overline{\rho}_{\Pi})$ and
		$\textbf{Dist} (\overline{\rho}_{\Pi})$,
		then the deformation functor $\mathcal{D}$
		is representable by a universal pair $ (R=R_{L^+},\rho=\rho_{L^+})$.
	\end{proposition}
	The proof is the same as \cite[Lemma 2.6]{HidaTilouine}.
	We denote the diagonal entries of the Galois representation
	$\rho|_{\Gamma_{\mathfrak{P}}}$
	by 
	$\mathrm{diag}(\psi_1,\psi_2,\ldots,\psi_n)
	=\mathrm{diag}(\psi_{L^+,1},\psi_{L^+,2},\ldots,\psi_{L^+,n})$.
	
	We can then use the arguments in \cite[Sections 2.4-2.7]{HidaTilouine} to
	prove the following.
	
	\begin{theorem}\label{R=T}
		Assume $\textbf{Big} (\overline{\rho}_{\Pi})$,
		$\textbf{Dist} (\overline{\rho}_{\Pi})$ and
		$\textbf{RegU} (\overline{\rho}_{\Pi})$,
		then we have an isomorphism of 
		complete intersection $\Lambda$-algebras
		\[R\simeq \mathbf{T}.\]
	\end{theorem}
	\begin{remark}\label{lambdaalgebra}
		\begin{enumerate}
			\item The $\Lambda$-algebra structure on $R$ is given by the characters
			$ (\psi_{1},\psi_{2},\ldots,\psi_{n})$ restricted to the inertia subgroup
			$I_{\mathfrak{P}}$.
			More precisely, we send $1+X_i^{ (j)}\in\Lambda$ to
			$\psi_{n+1-i}\circ\mathrm{Art}_{\mathfrak{p}} (u_j)^{-1}$ for
			$i=1,2,\ldots,n$ and $j=1,2,\ldots,d$;
			\item The $\Lambda$-algebra structure on $\mathbf{T}$ is induced from
			that on $h$, i.e. it is given by the diamond operators $\langle u\rangle_0$
			for $u\in T^{ss} (\mathfrak{P})$;
			\item These two algebra structures are compatible with the isomorphism
			$R\simeq\mathbf{T}$ in Theorem \ref{R=T}.
			To prove this, it suffices to show that for each specialization
			$\mathbf{T}\to\mathbf{T}/P_{\lambda}$ of $\mathbf{T}$
			of any dominant weight $\lambda$,
			the associated Galois representation
			$\rho_{\lambda}\colon \Gamma_{L^+}\to \mathcal{G}_n (\mathbf{T}/P_{\lambda})$,
			when restricted to the decomposition group
			$\Gamma_{\mathfrak{p}_{L^+}}$,
			has diagonal entries induced from the $\Lambda$-algebra structure on
			$\mathbf{T}/P_{\lambda}$. This follows from
			\cite[Proposition 2.6.1, Corollary 2.7.8 (i)]{Geraghty}.
		\end{enumerate}
	\end{remark}
	This remark will be used later when the base field
	$L^+$ varies and show the compatibility of this isomorphism
	$R_{L^+}\simeq\mathbf{T}_{L^+}$ with respect to this base change.

	\section{Some base changes}
	In this section we review some results concerning
	cyclic base change on automorphic representations
	and cyclic base change on deformation functors and universal deformation rings.
	
	\subsection{Cyclic base change of automorphic representations}
	Let $L_1^+/L^+$ be a 
	$p$-cyclic finite extension of totally real number fields,
	such that $L_1^+$ and $L^+$ 
	are both totally ramified at $p$,
	unramified outside $p$.
	Let $\pi=\pi_{L^+}$ be a cusp automorphic representation on $\mathrm{GL}_{n/L^+}$.
	Then by \cite{Clozel1991}, there is a cusp automorphic representation
	$\Pi=\Pi_{L^+}$ on the unitary group $G_{n/L^+}$ which is the descent of
	$\pi_{L^+}$ from $\mathrm{GL}_{n/L^+}$.
	By \cite{ArthurClozel}, there is a cusp automorphic representation
	$\pi_{L_1^+}$ on $\mathrm{GL}_{n/L^+_1}$
	which is the automorphic cyclic base change
	$\mathrm{GL}_{n/L^+}\leadsto\mathrm{GL}_{n/L^+_1}$ 
	of $\pi_{L^+}$.
	Applying again \cite{Clozel1991},
	there is a cusp automorphic representation $\Pi_{L_1^+}$ on $G_{n/L^+_1}$,
	the automorphic descent of $\pi_{L^+_1}$ from $\mathrm{GL}_{n/L^+_1}$.
	\[
	\begin{tikzcd}
	\pi_{L^+}/ (\mathrm{GL}_{n/L^+})\arrow[r,dash]\arrow[d,dash] & \pi_{L_1^+}/ (\mathrm{GL}_{n/L^+_1})\arrow[d,dash] \\
	\Pi_{L^+}/ (G_{n/L^+})\arrow[r,dash] & \Pi_{L_1^+}/ (G_{n/L^+_1})
	\end{tikzcd}
	\]
	
	Recall that by Theorem \ref{ExistenceOfBigHeckeAlgebraValuedGaloisRep}, 
	there exists an ordinary Galois representation
	$\rho_\mathbf{T}\colon\Gamma_{L^+}\to\mathcal{G}_n(\mathbf{T})$ 
	lifting the residue Galois representation 
	$\overline{\rho}_{\Pi_{L^+}}$ 
	which is given by 
	Theorem \ref{ExistenceOfGaloisRep}. 
	Similarly, there exists an ordinary 
	Galois representation
	$\rho_{\mathbf{T}_{L^+}}\colon
	\Gamma_{L^+_1}\to\mathcal{G}_n(\mathbf{T}_{L^+_1})$
	lifting
	$\overline{\rho}_{\Pi_{L^+_1}}$.

	Suppose that the characteristic polynomials
	$P (\mathrm{Frob}_{\mathfrak{q}},X)$ 
	have the factorization 
	(in some finite flat extension of $\mathbf{T}_{L^+}$ 
	depending on $\mathfrak{q}$)
	\[P (\mathrm{Frob}_{\mathfrak{q}},X)
	=\prod_{j=1}^n (X-\alpha_{\mathfrak{q},j}).\]
	
	Similarly, for a prime $\mathfrak{Q}$ of $L_1$ prime to $Np$ splitting in
	$L_1/L_1^+$,
	suppose that $P (\mathrm{Frob}_{\mathfrak{Q}},X)$
	has the following factorization (in some finite flat extension of $\mathbf{T}_{L^+_1}$ depending on $\mathfrak{Q}$)
	\[P (\mathrm{Frob}_{\mathfrak{Q}},X)=\prod_{j=1}^n (X-\alpha_{\mathfrak{Q},j}).\]
	Then we have the following proposition analogous to
	\cite[Proposition 2.3]{HidaTilouine}.
	\begin{proposition}\label{cyclic}
		There exists a ring homomorphism
		$\theta\colon  \mathbf{T}_{L_1^+}\to \mathbf{T}_{L^+}$
		above the algebra homomorphism
		$\Lambda_{L_1^+}\to \Lambda_{L^+}$
		(which is induced by the norm map
		$ (L_1^+)^{\times}\to (L^+)^{\times}$).
		The morphism $\theta$ is characterized by the following fact: suppose that $\mathfrak{Q}$ is a prime of $L_1$ prime to $Np$,
		splitting in $L_1/L_1^+$, lying over a prime $\mathfrak{q}$ of $L$,
		which splits in $L/L^+$, then
		\begin{enumerate}
			\item $\theta$ takes the polynomial
			$P (\mathrm{Frob}_{\mathfrak{Q}},X)=\prod_{j=1}^n (X-\alpha_{\mathfrak{Q},j})$
			to the polynomial
			\[\prod_{j=1}^n (X-\alpha_{\mathfrak{q},j}^{f (\mathfrak{q})})\]
			where $f (\mathfrak{q})$ is the residue degree of $\mathfrak{Q}$ over
			$\mathfrak{q}$;
			\item $\theta$ takes $U_{\mathfrak{P}_{L_1},j}$ to $U_{\mathfrak{P}_L,j}$
			for $j=1,2,\ldots,n-1$ for $\mathfrak{P}_{L_1}$ dividing $\mathfrak{P}_L$.
		\end{enumerate}
	\end{proposition}
	\begin{proof}
		\begin{enumerate}
			\item We write $t_{\mathfrak{Q}}$ for the Hecke parameter associated to
			the (local factor) admissible automorphic representation
			$\Pi_{L^+_1,\mathfrak{Q}}$ of
			$\Pi_{L_1^+}$. Similar definition for $t_\mathfrak{q}$.
			Then by the result in \cite[Chapter 3]{ArthurClozel},
			we have
			\[t_\mathfrak{Q}= (t_{\mathfrak{q}})^{f (\mathfrak{q})}.\]
			Using the Satake parameters,
			we can translate this relation in terms of eigenvalues of the matrix of the Frobenius elements
			$\mathrm{Frob}_\mathfrak{Q}$
			and $\mathrm{Frob}_{\mathfrak{q}}$
			\[\alpha_{\mathfrak{Q},j}=(\alpha_{\mathfrak{q},j})^{f (\mathfrak{q})}.\]
			\item Since $L_1^+/L^+$ is totally ramified over $p$,
			in the natural isomorphism
			\[\Gamma_{\mathfrak{P}_{L_1}}/I_{\mathfrak{P}_{L_1}}\simeq\Gamma_{\mathfrak{P}_{L}}/I_{\mathfrak{P}_{L}},\]
			the Frobenius $\mathrm{Frob}_{\mathfrak{P}_{L_1}}$ is taken to
			$\mathrm{Frob}_{\mathfrak{P}_{L}}$.
			Translating back to the Hecke side, we see that
			$U_{\mathfrak{P}_{L_1},j}$ is taken by $\theta$ to
			$U_{\mathfrak{P}_{L},j}$ for $j=1,2,\ldots,n-1$.
		\end{enumerate}
	\end{proof}
	Suppose that
	$\lambda_{L^+}\colon \mathbf{T}_{L^+}\to A_{L^+}$
	is a Hida family passing through $\Pi_{L^+}$.
	This implies that $\mathrm{Spec} (A_{L^+})$ is an irreducible component of
	$\mathrm{Spec} (\mathbf{T}_{L^+})$ corresponding to a minimal prime ideal
	$P_{L^+}$ of $\mathbf{T}_{L^+}$.
	We have the following
	\begin{proposition}\label{Hidafamilybasechange}
		Let $L_1^+/L^+$ be as above.
		There exists a Hida family $\lambda_{L_1^+}\colon \mathbf{T}_{L_1^+}\to A_{L_1^+}$
		passing through $\Pi_{L_1^+}$ such that the composition
		$\mathbf{T}_{L^+_1}\to\mathbf{T}_{L^+}\to A_{L^+}$ 
		factors through $A_{L^+_1}$.
	\end{proposition}
	\begin{proof}
		Recall the morphism $\theta\colon \mathbf{T}_{L_1^+}\to\mathbf{T}_{L^+}$.
		It suffices to show that there is a minimal prime ideal $P_{L_1^+}$ of
		$\mathbf{T}_{L_1^+}$ such that $\theta(P_{L_1^+})\subset P_{L^+}$.
		Note that $\theta^{-1} (P_{L^+})$ is a prime ideal of $\mathbf{T}_{L_1^+}$.
		So there exists such a minimal prime ideal $P_{L_1^+}$ contained in
		$\theta^{-1} (P_{L^+})$ 
		(use the fact that a local Noetherian ring has finite Krull dimension).
		This minimal prime ideal $P_{L_1^+}$
		gives rise to a Hida family passing through $\Pi_{L_1^+}$
		by the following commutative diagram.
		\[
		\begin{tikzcd}
		\mathbf{T}_{L_1^+}\arrow[d]\arrow[r] & \mathbf{T}_{L_1^+}/P_{L_1^+}\arrow[d]\arrow[rd,dashrightarrow] & \\
		\mathbf{T}_{L^+}\arrow[r] & \mathbf{T}_{L^+}/P_{L^+}\arrow[r] & \mathcal{O}
		\end{tikzcd}
		\]
	\end{proof}

	\subsection{Base change of deformation functors}
	Consider the deformation functors $\mathcal{D}_{L^+}$ and $\mathcal{D}_{L_1^+}$.
	It is easy to see that if $r\colon \Gamma_{L^+}\to \mathcal{G}_n (B)$ 
	is a lifting of
	$\overline{\rho}_{\Pi}$ for some object $B$ in $\mathrm{CNL}_{\mathcal{O}}$,
	then $r|_{\Gamma_{L_1^+}}$ is a lifting of
	$\overline{\rho}_{\Pi}|_{\Gamma_{L_1^+}}$ 
	(for places dividing $N$, one uses Lemma \ref{basechange} (3)).
	So we have a natural homomorphism
	of $\mathcal{O}$-algebras
	\[\alpha\colon R_{L_1^+}\to R_{L^+}\]
	The following lemma implies that the map $\alpha$ 
	is compatible with the map of weight spaces
	\begin{lemma}
		We have the following commutative diagram of algebras
		\[
		\begin{tikzcd}
		\Lambda_{L_1^+}\arrow[d,"\mathrm{Norm}"]\arrow[r] & R_{L_1^+}\arrow[r,"\simeq"]\arrow[d,"\alpha"] & \mathbf{T}_{L_1^+}\arrow[d,"\theta"]\\
		\Lambda_{L^+}\arrow[r] & R_{L^+}\arrow[r,"\simeq"] & \mathbf{T}_{L^+}
		\end{tikzcd}
		\]
	\end{lemma}
	\begin{proof}
		Only the commutativity of the left half of the diagram needs proof (the right half follows from Remark \ref{lambdaalgebra} (3) and Proposition \ref{cyclic}).
		
		By Remark \ref{lambdaalgebra}, the $\Lambda_{L^+}$-algebra structure on
		$R_{L^+}$ is given by the universal characters
		${\psi_{L^+,j}\colon \Gamma_{\mathfrak{p}_{L^+}}\to  (R_{L^+})^{\times}}$
		for $2\leq j\leq n$.
		Similarly for the $\Lambda_{L^+_1}$-algebra structure on $R_{L^+_1}$.
		The commutativity of the following left diagram for all $2\leq j\leq n$ gives the commutativity of the following right diagram by local class field theory.
		\[
		\begin{tikzcd}
		\Gamma_{\mathfrak{p}_{L_1^+}} \arrow[r,"\psi_{L_1^+,j}"] \arrow[d] &  (R_{L_1^+})^{\times}\arrow[d,"\alpha"] & \Lambda_{L_1^+}\arrow[r,"(\psi_{L_1^+,j})_j"]\arrow[d,"\mathrm{Norm}"] & R_{L_1^+}\arrow[d,"\alpha"]\\
		\Gamma_{\mathfrak{p}_{L^+}}\arrow[r,"\psi_{L^+,j}"] &  (R_{L^+})^{\times}& \Lambda_{L^+}\arrow[r,"(\psi_{L^+,j})_j"] & R_{L^+}
		\end{tikzcd}
		\]
	\end{proof}

	The group $\Delta$
	acts on the deformation functor $\mathcal{D}_{L_1^+}$,
	the universal deformation ring $R_{L_1^+}$,
	the weight Iwasawa algebra $\Lambda_{L^+_1}$
	and the Hecke algebra $\mathbf{T}_{L^+_1}$
	in the following way
	(see \cite[Section 2.2]{HidaParis94}
	or \cite[Section 3.2]{HidaAdjointSelmer}).
	For any element $\sigma\in\Gamma_{L^+}$,
	any object $A$ in $\mathrm{CNL}_{\mathcal{O}}$
	and any deformation $[r]\in\mathcal{D}_{L_1^+}(A)$,
	we fix a lifting
	$\ell(\sigma)\in\mathcal{G}_n(A)$
	of $\overline{\rho}_\pi(\sigma)=\overline{r(\sigma)}
	\in\mathcal{G}_n(\mathbb{F})$
	(such a lifting always exists since the group
	$\mathcal{G}_n$ is smooth over $\mathbb{Z}$, 
	and thus smooth over $\mathcal{O}$).
	Then we define a map
	$r^\sigma\colon\Gamma_{L_1^+}\to\mathcal{G}_n(A)$
	by sending $g$ to
	$\ell(\sigma)^{-1}\rho(\sigma g\sigma^{-1})\ell(\sigma)$.
	It is easy to verify that
	$r^\sigma$ is a Galois representation of $\Gamma_{L_1^+}$,
	that $[r^\sigma]$ is a deformation in $\mathcal{D}_{L_1^+}(A)$,
	that $[r^\sigma]$ is independent of the choice 
	of the representative $r$ in the equivalence class
	$[r]\in\mathcal{D}_{L_1^+}(A)$
	as well as $\ell(\sigma)$,
	and that $[r^\sigma]=[r]$ if $\sigma\in\Gamma_{L_1^+}$.
	As a result, the map 
	$[r]\mapsto[r^\sigma]$ gives an action of
	$\Gamma_{L^+}$ on $\mathcal{D}_{L_1^+}$,
	and moreover this action factors through
	$\Delta=\Gamma_{L^+}/\Gamma_{L_1^+}$.
	By posing 
	$A=R_{L_1^+}$
	and 
	by the definition of the universal deformation ring $R_{L_1^+}$,
	we see that the action of $\sigma\in\Delta$
	on $\mathcal{D}_{L_1^+}$
	gives rise to an action on $R_{L_1^+}$,
	which we denote by
	$\phi_\sigma\colon R_{L_1^+}\to R_{L_1^+}$,
	an automorphism of the $\mathcal{O}$-algebra $R_{L_1^+}$.
	We next define an action of
	$\Delta$ on the weight Iwasawa algebra $\Lambda_{L_1^+}$.
	Recall that
	$\Lambda_{L_1^+}=\mathcal{O}[[T^{ss}(\mathfrak{P}_{L_1^+})]]$
	where
	$T^{ss}(\mathfrak{P}_{L_1})
	=\mathrm{Ker}(T^{ss}(\mathcal{O}_{\mathfrak{P}_{L_1}})
	\to T^{ss}(\mathcal{O}_{\mathfrak{P}_{L_1}}/\mathfrak{P}_{L_1}))$.
	By the natural isomorphism
	$\Delta=\mathrm{Gal}(L_1^+/L^+)
	\simeq\mathrm{Gal}
	((L_1)_{\mathfrak{P}_{L_1}}/L_{\mathfrak{P}_L})$,
	$\Delta$ acts on $T^{ss}(\mathfrak{P}_{L_1})$,
	and thus on
	$\Lambda_{L_1^+}$.
	If we write
	$\Lambda_{L_1^+,L^+}=\mathrm{Im} (\Lambda_{L_1^+}\to\Lambda_{L^+})$,
	then it is easy to verify that
	$\Lambda_{L_1^+,L^+}\simeq(\Lambda_{L_1^+})_\Delta$.
	Moreover, we can verify that the actions of $\Delta$
	on $\Lambda_{L_1^+}$ and on $R_{L_1^+}$
	are compatible with the algebra homomorphism
	$\Lambda_{L_1^+}\to R_{L_1^+}$
	(indeed, the action of $\Delta$ on
	$R_{L_1^+}$ induces actions on the characters
	$\psi_{L_1^+,j}$
	given by
	$(\psi_{L_1^+,j})^\sigma(g)=\psi_{L_1^+,j}(\sigma g\sigma^{-1})$
	by the assumption
	$\mathbf{Dist}(\overline{\rho}_\pi)$.
	The latter is the same as the action of
	$\Delta$ on $\mathcal{O}_{\mathfrak{P}_{L_1}}$
	by local class field theory).
	Then we define an action
	of $\Delta$ on the big Hida-Hecke algebra $\mathbf{T}_{L_1^+}$.
	Any element $\sigma\in\Delta$
	sends
	Hecke operators
	$T_{\mathfrak{Q}_{L_1},i}$
	to $T_{\sigma(\mathfrak{Q}_{L_1}),i}$,
	$U_{\mathfrak{P}_{L_1},i}$ 
	to $U_{\sigma(\mathfrak{P}_{L_1}),i}$,
	and sends the diamond operators
	$\langle u\rangle_0$ 
	to
	$\langle\sigma(u)\rangle_0$
	with $u\in T^{ss}(\mathcal{O}_{\mathfrak{p}_{L^+}})$. 
	Clearly, these actions of $\Delta$
	on $\mathbf{T}_{L_1^+}$
	and on $\Lambda_{L_1^+}$
	are compatible with the $\Lambda_{L_1^+}$-algebra 
	structure on $\mathbf{T}_{L_1^+}$.
	The compatibility of the actions
	of $\Delta$ on $R_{L_1^+}$ and on $\mathbf{T}_{L_1^+}$
	follows from the Satake isomorphism 
	for places of $L_1^+$ unramified for $\pi$
	and the assumption $\mathbf{Dist}(\overline{\rho}_\pi)$
	for the place of $L_1^+$ above $p$.
	
	Recall that we denote the quotient
	$R_{L_1^+}/\Sigma_{\sigma\in\Delta}R_{L_1^+}(\sigma-1)R_{L_1^+}$
	by $(R_{L_1^+})_\Delta$.
	We have the following control theorem on the universal deformation rings/Hecke algebras
	with respect to the base change $L_1^+/L^+$.
	\begin{theorem}\label{control}
		Assume $\textbf{Dist} (\overline{\rho}_{\Pi})$, then
		$(R_{L_1^+})_{\Delta}
		\simeq
		(\mathbf{T}_{L_1^+})_\Delta
		\simeq
		R_{L^+}
		\simeq
		\mathbf{T}_{L^+}$.
	\end{theorem}
	\begin{proof}
		We give two proofs of this proposition.
		
		The first proof is as follows: 
		by Proposition \ref{cyclic}, it is clear that
		$(\mathbf{T}_{L_1^+})_\Delta$
		injects into $\mathbf{T}_{L^+}$.
		It then suffices to show that $(R_{L_1^+})_\Delta$
		surjects onto $R_{L^+}$.
		In other words, we have to show that 
		for any two deformations
		$[r],[r']\in\mathcal{D}_{L^+}(A)$
		such that
		$[r|_{\Gamma_{L_1^+}}]=
		[r'|_{\Gamma_{L_1^+}}]\in\mathcal{D}_{L_1^+}(A)$,
		then $[r]=[r']$.
		We can show
		(see \cite[Corollary A.2.1]{HidaParis94})
		\[\{r\otimes\psi\colon \psi\in \mathrm{Hom} (\Delta,1+\mathfrak{m}_B)\}
		=
		\{r''\in\mathcal{D}_{L^+} (B)
		\colon
		[r^{\prime\prime}|_{\Gamma_{L_1^+}}]
		=
		[r|_{\Gamma_{L_1^+}}]\}.\]
		So $r'=r\otimes\psi$ for a character
		$\psi\colon\Delta\to1+\mathfrak{m}_B$.
		Note that
		$\mathrm{det}(r'|_{\Gamma{\mathfrak{p}_{L^+}}})=\chi^{-n(n-1)/2}
		=\mathrm{det}(r|_{\Gamma_{\mathfrak{p}_{L^+}}})$.
		Thus $\psi^n=1$.
		Since $\Delta$ is a $p$-group
		while $n$ is prime to $p$,
		so $\psi=1$ and $[r']=[r]$.

		The second proof follow the ideas of Hida in \cite[Theorem 2.1]{HidaParis94}.
		We consider the two following auxiliary deformation functors
		\[
		\mathcal{D}_{L_1^+,L^+}\colon  \mathrm{CNL}_{\mathcal{O}}\to \mathrm{Sets},
		\]
		\[
		A
		\mapsto
		\{
		[r|_{\Gamma_{L_1^+}}]\in\mathcal{D}_{L_1^+} (A)
		\colon [r]\in\mathcal{D}_{L^+} (B)
		\text{ for a flat }A\text{-algebra }B
		\in \mathrm{CNL}_{\mathcal{O}}
		\};
		\]
		\[
		\mathcal{D}_{L_1^+}^{\Delta}\colon \mathrm{CNL}_\mathcal{O}
		\to \mathrm{Sets},
		\]
		\[
		A\mapsto \{[r]\in\mathcal{D}_{L^+} (A)
		\colon r^{\sigma}\simeq r, \forall \sigma \in\Delta\}.
		\]
		
		One sees that there are natural transformations of functors:
		$\mathcal{D}_{L^+}\xrightarrow{\alpha_1}\mathcal{D}_{L_1^+,L^+}\xrightarrow{\alpha_2}\mathcal{D}_{L_1^+}^{\Delta}\xrightarrow{\alpha_3}\mathcal{D}_{L_1^+}$
		such that $\alpha_3\circ\alpha_2\circ\alpha_1=\alpha$.
		We first show that $\alpha_1$ and $\alpha_2$ are isomorphisms.
		Fix an arbitrary object $A$ in $\mathrm{CNL}_{\mathcal{O}}$
		and an arbitrary deformation $[r]\in \mathcal{D}_{L_1^+}^{\Delta} (A)$.
		
		To prove that $\alpha_2$ is an isomorphism,
		it is enough to show that there exists an extension
		$r'\in\mathcal{D}_{L^+} (B)$ of $r$ for some flat $A$-algebra
		$B\in\mathrm{CNL}_{\mathcal{O}}$.
		It is not hard to show that there exists a faithfully flat $A$-algebra $B$ in
		$\mathrm{CNL}_{\mathcal{O}}$ such that $r$ extends to a Galois representation
		$r'\colon \Gamma_{L^+}\to \mathcal{G}_n (B)$ with
		$r'$ lifting 
		$\overline{\rho}_{\Pi}$
		(the arguments in \cite[Corollary A.1.3]{HidaParis94} work through for
		representations taking values in $\mathcal{G}_n (A)$ 
		instead of $\mathrm{GL}_n (A)$).
		Next we verify that $r'\in\mathcal{D}_{L^+} (B)$.
		For each prime $\mathfrak{q}$ of $L^+$ dividing $N$ and a prime $\mathfrak{Q}$ of $L_1^+$ over $\mathfrak{q}$,
		we have $I_{\mathfrak{q}}=I_{\mathfrak{Q}}$.
		So the $\mathfrak{Q}$-minimality of $r$ implies the
		$\mathfrak{q}$-minimality of $r'$ and thus $r'$ satisfies the $N$-minimality condition.
		We can assume that $r|_{\Gamma{\mathfrak{p}_{L_1^+}}}$ is an upper triangular representation  (up to conjugation).
		Suppose that $r'$ acts on the $B$-free module $V$ of rank $n$ and $r$ fixes a non-zero vector $0\neq x\in V$ by $r (\sigma)x=\psi_1 (\sigma)x$ for any $\sigma\in\Gamma_{\mathrm{p}_{L_1^+}}$
		(this vector is unique up to a scalar multiplication since we have assumed that the diagonal entries of the representation $r$ are distinct).
		Now for any $\tau\in\Gamma_{\mathfrak{p}_{L^+}}$,
		we set $\sigma^{\tau}:=\tau^{-1}\sigma\tau\in\Gamma_{\mathfrak{p}_{L_1^+}}$,
		then $\sigma\tau=\tau\sigma^{\tau}$.
		Therefore
		\[r' (\sigma)r' (\tau)x=r' (\tau)r' (\sigma^{\tau})x=r' (\tau)\psi_1 (\sigma^{\tau})x=\psi_1 (\sigma^{\tau})r' (\tau)x.\]
		Modulo the two ends of the above equation by $\mathfrak{m}_B$,
		we get
		\[\overline{r'} (\sigma)\overline{r'} (\tau)\overline{x}=\overline{\psi}_1 (\sigma^{\tau})\overline{r'} (\tau)\overline{x}.\]
		Note that the residue representation
		$\overline{r'}|_{\Gamma_{\mathfrak{p}_{L^+}}}=\overline{\rho}_{\Pi}|_{\Gamma_{\mathfrak{p}_{L^+}}}$
		is upper triangular with distinct diagonal entries
		$\overline{\psi}_1,\overline{\psi}_2,\ldots,\overline{\psi}_n$.
		So \textit{a priori} we have also
		\[\overline{r'} (\sigma)\overline{r'} (\tau)\overline{x}=\overline{\psi}_1 (\sigma)\overline{r'} (\tau)\overline{x}.\]
		By assumption $r$ is invariant under the action of $\Delta$,
		so are $\psi_1$ and $\overline{\psi}_1$.
		Therefore $\overline{\psi}_1 (\sigma)=\overline{\psi}_1 (\sigma^{\tau})$.
		Since these characters $\psi_1,\psi_2,\ldots,\psi_n$ are distinct,
		we conclude that $\psi_1 (\sigma^{\tau})=\psi_1 (\sigma)$.
		This shows that $\tau$ leaves $Bx$ stable and thus acts as a non-zero scalar in $B^{\times}$.
		This shows that we can extend the character $\psi_1$ from
		$\Gamma_{\mathfrak{p}_{L_1^+}}$ to the larger group
		$\Gamma_{\mathfrak{p}_{L^+}}$
		by setting $\psi_1(\tau)$ to be the above scalar.
		Then we can consider the representation $V/Bx$ and repeat the above argument for $\psi_2$
		(induction on $\psi_i$)
		and we conclude that $r$ extends to an upper triangular representation $r$ of $\Gamma_{\mathfrak{p}_{L^+}}$.
		This shows that $\alpha_2$ is an isomorphism.
		
		To show that $\alpha_2\circ\alpha_1$ is an isomorphism,
		it is sufficient to show that the above extension $r'$ is unique.
		This is the same as the second part of the first proof.
		It is not hard to show that  the functor
		$\mathcal{D}_{L_1^+}^{\Delta}$, resp. $\mathcal{D}_{L^+,L_1^+}$
		is representable by $ (R_{L_1^+,\Delta},\rho_{L_1^+}\pmod{\mathfrak{a}})$, 
		resp. $ (\mathrm{Im} (\alpha),\alpha (\rho_{L_1^+}))$		
		(see \cite[Proposition A.2.2]{HidaParis94})
		where $\mathfrak{a}$ is the ideal
		$\sum_{\gamma\in\Delta}R_{L_1^+} (\gamma-1)R_{L_1^+}$ of $R_{L_1^+}$
		(see \cite[Proposition A.2.1]{HidaParis94}).
		So we conclude that $\mathrm{Im} (\alpha)=R_{L^+}$ and the theorem follows.
	\end{proof}

    We deduce the following result 
	on the control of the K\"{a}hler differentials.
	\begin{corollary}
		Suppose that $B$ is an $R_{L^+}$-algebra,
		then we have the following isomorphism
		\[ (\Omega_{R_{L_1^+}/\Lambda_{L_1^+}}\bigotimes_{R_{L_1^+}}B)_{\Delta}\simeq\Omega_{R_{L^+}/\Lambda_{L_1^+,L^+}}\bigotimes_{R_{L^+}}B\]
	\end{corollary}

	This corollary is a special case of the following.
	\begin{corollary}\label{controlofKahler}
		Let $C_{L_1^+}$ be an object in $\mathrm{CNL}_\mathcal{O}$ 
		with a continuous action of $\Delta$. 
		Suppose that $R_{L_1^+}$ is a $C_{L_1^+}$-algebra
		and that the actions of $\Delta$ on 
		$C_{L_1^+}$ and $R_{L_1^+}$ are compatible.
		Thus $R_{L^+}$ is a $C_{L^+}$-algebra
		with $C_{L^+}:=(C_{L_1^+})_\Delta$.
		Then we have the following isomorphism
		\[
		(\Omega_{R_{L_1^+}/C_{L_1^+}}\bigotimes_{R_{L_1^+}}B)_\Delta
		\simeq
		\Omega_{R_{L^+}/C_{L^+}}\bigotimes_{R_{L^+}}B.
		\]
	\end{corollary}

   The proof is identical to the proof of \cite[Corollary 3.4]{HidaAdjointSelmer}.

	\section{Selmer groups and congruence ideals}

	\subsection{Preliminary on commutative algebra}
	
	First recall the notion of congruence ideals as in
	\cite[Section 8.2]{HidaTilouine}.
	Let $B$ be a normal noetherian domain and $R$ is a reduced finite flat $B$-algebra with a morphism of $B$-algebras $\phi\colon R\to B$.
	Since $R$ is reduced, the fractional ring $\mathrm{Frac} (R)$
	is a product of fields which has the form
	$\mathrm{Frac} (R)=\mathrm{Frac} (B)\oplus X$.
	We write $I=\mathrm{Ker} (R\to\mathrm{Frac} (R)\to X)$.
	\begin{definition}\label{congruenceideal}
		The \textbf{congruence ideal} of $\phi\colon R\to B$ is
		\[\mathfrak{c}_{\phi}=\mathrm{Ann}_B ((R/I)\bigotimes_RB)\subset B.\]
	\end{definition}
	We have the following property of congruence ideals(see \cite[Corollary 8.6]{HidaTilouine})
	\begin{proposition}\label{DecompositionPropertyOfCongruenceIdeal}
		Let $B$ be a normal profinite local domain of residue characteristic $p$. Suppose that $R\xrightarrow{\alpha}S\xrightarrow{\beta}A$
		are surjective morphisms of reduced local finite flat Gorenstein $B$-algebras 
		($R$ is a Gorenstein $B$-algebra, i.e.,
		$\mathrm{Hom}_{B-\mathrm{alg}}(R,B)$ 
		is a free $R$-module of rank one.
		Similarly for $S$ and $A$).
		Write $\lambda=\beta\circ\alpha$.
		Then
		\[\mathfrak{c}_{\lambda}=\mathfrak{c}_{\beta}\beta (\mathfrak{c}_{\alpha}).\]
	\end{proposition}

   Next recall the notion of characteristic ideals.
   Let $B$ be a normal noetherian domain and 
   $M$ be a finitely generated torsion $B$-module. 
   Then by 
   \cite[Chapitre 7, Th\'{e}or\`{e}me 5, p.253]{BourbakiAlgebreCommutative},
   the module $M$ is pseudo-isomorphic 
   to a $B$-module $M'$ of the form
   $M'=\bigoplus_{i\in I}B/P_i^{n_i}$,
   where $I$ is a finite index set, 
   $P_i$ are prime ideals of $B$ of height one 
   and $n_i$ are positive integers. 
   Moreover, the set of the pairs $(P_i,n_i)$ 
   is determined by the module $M$.
   \begin{definition}\label{DefinitionOfCharacteristicElement}
   The \textbf{characteristic ideal} 
   $\mathrm{char}(M)=\mathrm{char}_B(M)$ 
   of $M$ is
   \[\mathrm{char}(M)=\prod_{i\in I}P_i^{n_i}.\]
   \end{definition}
   If we assume moreover that $B$ is \textbf{regular}, 
   then each $P_i$ is in fact a principal ideal 
   (see \cite[Tags.0AG0, 0AFT]{StackProject}), 
   generated by an element in $B$, say, $b_i$. 
   Then we also use $\mathrm{char}(M)=\mathrm{char}_B(M)$ 
   to denote
   the \textbf{characteristic element} of $M$ defined as 
   (it will be clear from the context whether $\mathrm{char}(M)$ 
   denotes an ideal or an element of $B$)
   \[\mathrm{char}(M)=\prod_{i\in I}b_i^{n_i}.\]
   Note that $\mathrm{char}(M)$ is only unique up to a unit in $B$. 
   We will mainly deal with the cases $B=\tilde{A}$ 
   or $B=\tilde{A}[[S]]$. 
   We will assume that $\tilde{A}$ is regular, 
   so that we can speak of characteristic elements of finitely generated torsion
   modules over $\tilde{A}$ or $\tilde{A}[[S]]$.
   We have the following simple observation on characteristic ideals
   
   \begin{lemma}\label{PropertyOfCharacteristicIdeal}
   	Let $B$ be a regular normal noetherian local domain and $B_\infty=B[[S]]$.
   	Let $M_\infty$ be a finitely generated torsion $B_\infty$-module and  $M=M_\infty/SM_\infty$ be a finitely generated torsion $B$-module. 
   	Suppose that $M_\infty$, resp. $M$, has no non-trivial $B_\infty$-submodule,
   	resp. $B$-submodule pseudo-isomorphic to $0$. 
   	Then
   	\[\mathrm{char}_{B_\infty}(M_\infty)+(S)(\mathrm{mod}\ S)=\mathrm{char}_B(M).\]
   \end{lemma}
   \begin{proof}
   	Since $B$ is regular, so is $B_\infty$. 
   	Suppose that $M_\infty$ is $B_\infty$-pseudo-isomorphic to $Q_\infty=\prod_{i\in I}B_\infty/P_i^{n_i}$ with ${\mathrm{char}_{B_\infty}(M_\infty)=\prod_{i\in I}P_i^{n_i}}$ 
   	where each $P_i$ is generated by a power series $f_i(S)\in B_\infty$.
   	Since $M_\infty$ has no non-trivial pseudo-zero $B_\infty$-submodule, we have an exact sequence of $B_\infty$-modules
   	\[0\to M_\infty\to Q_\infty\to N_\infty\to 0\]
   	such that $N_\infty$ is a pseudo-zero $B_\infty$-module.
   	It is easy to see that $N_\infty/SN_\infty$ is a pseudo-zero $B$-module 
   using the $B$-algebra morphism $B\to B_\infty$ and going-up theorem in commutative algebra.
   	
   	Modulo the above exact sequence by $S$ and using the assumption that $M$ has no non-trivial pseudo-zero $B$-submodule, 
   	we get another exact sequence of $B$-modules
   	\[0\to M\to Q_\infty/SQ_\infty\to N_\infty/SN_\infty\to0.\]
   	
   	Thus we get that $\mathrm{char}_B(M)=\mathrm{char}_B(Q_\infty/SQ_\infty)=\prod_{i\in I}f_i(0)^{n_i}B$. 
   	This concludes the proof.
   \end{proof}
   
   Then recall the notion of Fitting ideals.
   \begin{definition}
   	Let $B$ be a noetherian ring and $\mathrm{Q}(B)$ its total ring of fractions. 
   	Let $C$ be an $r\times s$-matrix with entries in $\mathrm{Q}(B)$. 
   	For any integer $i\geq0$, the $i$-th \textbf{Fitting (fractional) ideal} $\mathrm{F}^{(i)}_B(C)$
   	of $C$ over $B$ is the fractional ideal of $B$ generated 
   	by all the $(r-i)\times(r-i)$-minors of $C$ 
   	(if $r\leq i$ or $s<r$, we put $\mathrm{F}^{(i)}_B(C)=B$).
   	
   	Let $M$ be a finitely generated $B$-module. 
   	Suppose that $B^a\xrightarrow{h}B^b\to M\to0$ is a presentation of $M$.
   	Denote by $H$ the $b\times a$-matrix for the morphism $h$ under the standard 
   	$B$-basis of $B^a$ and $B^b$. Then the $i$-th \textbf{Fitting ideal}
   	$\mathrm{F}^{(i)}_B(M)$ of $M$ over $B$ is $\mathrm{F}^{(i)}_B(H)$.
   	
   	When $i=0$, we also write $\mathrm{F}_B(C)$, resp. $\mathrm{F}_B(M)$ for $\mathrm{F}^{(0)}_B(C)$, resp. $\mathrm{F}^{(0)}_B(M)$.
   \end{definition}

   It can be shown that $\mathrm{F}^{(i)}_B(M)$ 
   is independent of the choice of the presentation $h$ for $M$.
   
   Recall briefly the notion of reflexive envelopes
   for ideals in $B$.
   For any ideal $I$ of $B$, let $\tilde{I}$ be
   $\mathrm{Hom}_B(\mathrm{Hom}_B(I,B),B)$.
   $\tilde{I}$ is called the \textbf{reflexive envelope} of $I$ 
   (see \cite[Chapitre 7, \textsection4,
   $\mathrm{n}^\circ2$]{BourbakiAlgebreCommutative}).
   If $B$ is normal, then $\tilde{I}$ is isomorphic to the ideal
   $\bigcap_PI_P$, 
   the intersection of the localizations $I_P$ of $I$ 
   at the prime ideals $P$ of $B$ of height one.
   We write $\tilde{\mathrm{F}}_B(M)$ for the reflexive envelope of
   $\mathrm{F}_B(M)$.

   We then list some properties of Fitting ideals that will be used later.
   \begin{proposition}\label{PropertiesOfFittingIdeal}
   	\begin{enumerate}
   		\item If $I$ is an ideal of $B$, then
   		$\mathrm{F}_{B/I}(M/IM)=\mathrm{F}_B(M)/I\mathrm{F}_B(M).$
   		\item If $M_1$ and $M_2$ are finitely generated $B$-modules, then
   		$\mathrm{F}^{(i)}_B(M_1\oplus M_2)=\sum_{j+l=i}\mathrm{F}^{(j)}_B(M_1)\mathrm{F}^{(l)}_B(M_2)$.
   		\item If $0\to M_1\to M_2\to M_3\to0$ is an exact sequence of finitely generated $B$-modules, then
   		$\mathrm{F}^{(i)}_B(M_2)\supset\sum_{j+l=i}\mathrm{F}^{(j)}_B(M_1)\mathrm{F}^{(l)}_B(M_3)$.
   		If moreover $M_3$ has a presentation of the form 
   		${B^a\to B^a\to M_3\to0}$, 
   		then $\mathrm{F}_B(M_3)$ is a principal ideal of $B$ and
   		$\mathrm{F}_B(M_2)=\mathrm{F}_B(M_1)\mathrm{F}_B(M_3).$
   		\item If $S$ is a multiplicative subset of $B$ not containing $0$. We write the subscript $S$ to indicate localization with respect to $S$. Then 
   		$\mathrm{F}_{B_S}(M_S)=\mathrm{F}_B(M)_S.$
   		\item If $B$ is a discrete valuation ring, then
   		$\mathrm{F}_B(M)=\mathrm{char}_B(M)$.
   		If moreover $0\to M_1\to M_2\to M_3\to0$ 
   		is an exact sequence of finitely generated torsion $B$-modules,
   		then we have
   		$\mathrm{F}_B(M_1)\mathrm{F}_B(M_3)=\mathrm{F}_B(M_2)$.
   		\item If $B$ is a normal noetherian local domain 
   		and $M$ is a finitely generated torsion $B$-module, then
   		$\tilde{\mathrm{F}}_B(M)=\mathrm{char}_B(M).$
   		In other words, for any prime ideal $P$ in $B$ of height one, we have
   		$\mathrm{F}_B(M)_P=\mathrm{char}_B(M)_P.$
   		\item If $R$ is a complete intersection $B$-algebra 
   		with a morphism of $B$-algebras $R\to B$. 
   		Then the $B$-module $\Omega_{R/B}\otimes_RB$ 
   		has a Fitting ideal $\mathrm{F}_B(\Omega_{R/B}\otimes_RB)$ 
   		which is a principal ideal.
   	\end{enumerate}
   \end{proposition}

   \begin{proof}
   	\begin{enumerate}
   		\item See \cite[Appendix, 4, p.325]{MazurWiles1984}.
   		\item See \cite[Chapter 3, Exercise 3, p.61]{Northcott}.
   		\item See \cite[Chapter 3, Exercise 2, p.61]{Northcott} for the first part and \cite[Chapter 3, Theorem 22, p.80]{Northcott} for the second part.
   		\item See \cite[Chapter 3, Theorem 3, p.59]{Northcott}.
   		\item Note that the pseudo-isomorphic relation between $B$-modules is the same as the isomorphic relation. 
   		Using (2), the problem is reduced to show that for any $k\geq0$,
   		$\mathrm{F}_B(B/\varpi^kB)=\varpi^kB$
   		where $\varpi$ is a uniformizer of $B$. 
   		The latter follows from the definition of Fitting ideals.
   		\item For any prime ideal $P$ of $B$ of height one, 
   		the localization $B_P$ is a discrete valuation ring. 
   		Write $n(P,N)$ for the length of a $B_P$-module $N$. 
   		Then it is easy to see that ($B$ is regular)
   		\[\tilde{\mathrm{F}}_B(M)=\prod_{\mathrm{ht}(P)=1}
   		P^{n(P,B_P/\mathrm{F}_B(M)_P)}
   		=\prod_{\mathrm{ht}(P)=1}P^{n(P,B_P/\mathrm{F}_{B_P}(M_P))}.\]
   		For each such $P$, $P^{n(P,B_P/\mathrm{F}_{B_P}(M_P))}$ is just
   		$P^{n(P,M_P)}$ by (5). 
   		Since by definition there is
   		$\mathrm{char}_B(M)=\prod_{\mathrm{ht}(P)=1}P^{n(P,M_P)}$, 
   		we conclude.
   		\item Suppose that $R=B[[X_1,X_2,\ldots,X_r]]/(f_1,f_2,\ldots,f_r)$. Write $J=(f_1,\ldots,f_r)$. Now the second fundamental exact sequence of the K\"{a}hler differentials for the natural morphisms ${B\to B[[X_1,\ldots,X_r]]\to R}$ gives us
   		\[J/J^2\bigotimes_{B[[X_1,\ldots,X_r]]}B\to\Omega_{B[[X_1,\ldots,X_r]]/B}\bigotimes_{B[[X_1,\ldots,X_r]]}B\to\Omega_{R/B}\bigotimes_RB\to0.\]
   		This is just
   		$\oplus_{i=1}^rBdf_i\to\oplus_{i=1}^rBdX\to\Omega_{R/B}\otimes_RB\to0,$
   		which gives a presentation of $\Omega_{R/B}\otimes_RB$ and by (3), we conclude.

   	\end{enumerate}
   \end{proof}

   \begin{remark}
   	Let $R$ be a complete intersection $B$-algebra 
   	with a $B$-algebra homomorphism
   	$\phi\colon R\to B$.
   	Suppose also that $R$ is a reduced algebra.
    By a theorem of J.Tate (see \cite[Theorem 8.7]{HidaTilouine}), 
    we have 
    $\mathrm{F}_B(\Omega_{R/B}\otimes_{R,\phi}B)=\mathfrak{c}_{\phi}$.
   \end{remark}
	
	\subsection{Cotorsionness of Selmer groups}
	Let $\Lambda=\mathcal{O}[[X]]$ be the Iwasawa algebra of one variable,
	$eh$ for the big cusp Hida-Hecke algebra 
	(generated over $\mathcal{O}$ by the operators $T_q$ for $q$ prime to $Np$,
	the diamond operators $\langle u\rangle_0$ for $u\in\mathbb{Z}_p^{\times}$ and the Atkin-Lehner operator $U_p$).
	It is known that $h$ is reduced and finite torsion-free over $\Lambda$.
	The $\Lambda$-algebra structure on $h$ is given 
	by sending $1+X$ to $\langle 1+p\rangle_0$.
	Let $\pi$ be a non-CM $p$-ordinary holomorphic cusp automorphic representation of $\mathrm{GL}_2 (\mathbb{A}_{\mathbb{Q}})$,
	cohomological for a local system of highest weight $a>0$,
	of conductor $N$ and level subgroup
	$U^{ (1)}_0 (N)=\{u\in\mathrm{GL}_2 (\hat{\mathbb{Z}})\colon u\pmod{N} \text{ is upper triangular}\}$.
	Let $\mathbf{T}$ be the localization of $eh$ at the maximal ideal associated to residue Galois representation corresponding to $\pi$.
	Let $\mu\colon \mathbf{T}\to A$ be a Hida family passing through $\pi$.
	Then there is a Galois representation
	$\rho_{\mu}\colon \Gamma_{\mathbb{Q}}\to \mathrm{GL}_2 (A)$
	associated to $\mu$.
	Moreover, $\overline{\rho}_{\mu}|_{\Gamma_p}$
	is conjugate to an upper triangular representation whose diagonal entries are
	$ (\mathrm{unr} (\overline{\alpha}),\mathrm{unr} (\overline{\alpha})^{-1}\omega^{-1-a})$
	where $\alpha=\mu (U_p)$ and $\mathrm{unr} (\overline{\alpha})$ is the unramified character of $\Gamma_p$
	taking the Frobenius $\mathrm{Frob}_p$ to $\overline{\alpha}$ and $\omega$ is the modulo $p$ cyclotomic character.
	For any representation $\sigma\colon\Gamma_\mathbb{Q}\to\mathrm{GL}_2(W)$ 
	taking values in any ring $W$, we write
	$\mathcal{A}^n(\sigma)$ for the symmetric power representation
	${(\mathrm{Sym}^{2n}\otimes\mathrm{det}^{-n})\sigma}$. 
	In particular we are interested in the representations $\mathcal{A}^{j}(\rho_\mu)\colon \Gamma_\mathbb{Q}\to\mathrm{GL}_{2j+1}(A)$.
	
	Let $\mathbb{X}\colon \Gamma_{\mathbb{Q}}\to\Lambda^{\times}=\Lambda^{\times}_{\mathbb{Q},2}$
	denote the restriction to $\Gamma_p$ of the unramified deformation
	of the $p$-adic cyclotomic character $\chi$ 
	unramified outside $p\infty$
	(more precisely, if $\chi (\sigma)=\omega (\sigma)u^{\ell (\sigma)}$, then $\mathbb{X} (\sigma)=\chi (\sigma) (1+X)^{\ell (\sigma)}$).
	We write $\Phi_a=\omega^a\mathbb{X}$.
	Then $\mathcal{A}^j(\rho_\mu)|_{\Gamma_p}$ has a filtration
	$ (F^{m (a+1)}\mathcal{A}^j(\rho_\mu))_m$ with $A$-free graded pieces
	$\mathrm{gr}^{m (a+1)}\mathcal{A}^j(\rho_\mu)$ on which $\Gamma_p$ acts by
	$\mathrm{unr} (\alpha)^{2m}\Phi^m_a$. 
	For any $L^+$, the Galois representation
	$\mathcal{A}^j(\rho_\mu)|_{\Gamma_{L^+}}$ is again ordinary at
	$\mathfrak{p}_{L^+}$ and the filtration of this representation restricted to
	$\Gamma_{\mathfrak{p}_{L^+}}$ is just the restriction of the filtration
	$ (F^{m (a+1)}\mathcal{A}^j(\rho_\mu))_{m\in\mathbb{Z}}$ to $\Gamma_{\mathfrak{p}_{L^+}}$.
	
	Let $\tilde{A}$ be the normal closure of $A$ in its fractional field.
	It is a $2$-dimensional normal ring and thus is free over $\Lambda$.
	For any $\mathcal{O}$-module $M$, $M^{\ast}$ denotes its Pontryagin dual
	$\mathrm{Hom}_{\mathcal{O}} (M,K/\mathcal{O})$.
	We define the minimal $p$-ordinary Selmer group $\mathrm{Sel}(L^+,\mathcal{A}^j(\rho_\mu))$ 
	associated to $\rho'=\mathcal{A}^j(\rho_\mu)|_{\Gamma_{L^+}}$ as follows:
	\[
	\mathrm{Ker}
	(H^1 (L^+,\rho'\otimes_{A}\tilde{A}^{\ast})
	\to
	\prod_{\mathfrak{q}\neq \mathfrak{p}_{L^+}}
	H^1 (I_{\mathfrak{q}},\rho'\otimes_A\tilde{A}^{\ast})
	\times
	H^1 (I_{\mathfrak{p}_{L^+}}, (\rho'/F^1\rho')\otimes_A\tilde{A}^{\ast})).
	\]
	
	We set $G=G_{/\mathbb{Q}}=G_{n/\mathbb{Q}}$ for some $n>1$. 
	By \cite{ArthurClozel}, 
	we have the automorphic cyclic base change functoriality
	$\mathrm{GL}_{2/\mathbb{Q}}\leadsto 
	\mathrm{GL}_{2/\mathbb{Q}_k}$ for all $k$. 
	Assume that the Langlands functorialities
	$\mathrm{Sym}^{n-1}
	\colon
	\mathrm{GL}_{2/\mathbb{Q}_k}\leadsto\mathrm{GL}_{n/\mathbb{Q}_k}$ 
	are established for all $k$ 
	(corresponding to the homomorphisms of Langlands'
	$L$-groups
	$\mathrm{Sym}^{n-1}
	\colon
	\mathrm{GL}_2(\mathbb{C})
	\to
	\mathrm{GL}_n(\mathbb{C})$). 
	Note that by 
	\cite{ClozelThorne2014,ClozelThorne2015,ClozelThorne2017}, 
	for $p>7$, the functorialities
	$\mathrm{Sym}^{n-1}
	\colon
	\mathrm{GL}_{2/\mathbb{Q}_k}
	\leadsto\mathrm{GL}_{n/\mathbb{Q}_k}$ 
	are established for all $n\leq9$ and all $k\geq0$.
	By \cite{Clozel1991},
	we have the functoriality of automorphic descent
	$\mathrm{GL}_{n/\mathbb{Q}_k}\leadsto G_{/\mathbb{Q}_k}$.

	Let $\pi_k$, resp. $\Pi_k=\Pi_k^G$, 
	be the base change (packet) of $\pi$ under the Langlands functorialities
	$\mathrm{GL}_{2/\mathbb{Q}}\leadsto \mathrm{GL}_{2/\mathbb{Q}_k}$, resp.
	$\mathrm{GL}_{2/\mathbb{Q}}\leadsto\mathrm{GL}_{2/\mathbb{Q}_k}\leadsto
	\mathrm{GL}_{n/\mathbb{Q}_k}\leadsto G_{/\mathbb{Q}_k}$.
	We write $\rho^{G}_\mu=\mathrm{Sym}^{n-1}\rho_\mu$. 
	Our aim in this subsection is to show that, under certain conditions,
	the dual Selmer group
	$\mathrm{Sel} (\mathbb{Q}_{\infty},\mathrm{Ad} (\rho^{ G}_\mu))^\ast$
	is a finitely generated torsion $\tilde{A}[[S]]$-module and study its
	characteristic power series.
	Concerning Condition \ref{condition} for 
	$\overline{\rho}_\pi$ as well as $\overline{\rho}_\Pi$,
	it is easy to show that if
	$\alpha^{2c_n}\not\equiv1(\mathrm{mod}\ p)$, then
	$\mathbf{Dist}(\overline{\rho}_\Pi)$
	is satisfied.
	Here $c_n$
	is the least common multiple of the integers
	between $1$ and $2n$.
	Similarly, $\mathbf{Big}(\overline{\rho}_\pi)$,
	resp. $\mathbf{RegU}(\overline{\rho}_\pi)$,
	implies
	$\mathbf{Big}(\overline{\rho}_\Pi)$,
	resp. $\mathbf{RegU}(\overline{\rho}_\Pi)$.

	Over $\mathrm{GL}_{2/\mathbb{Q}_k}$, one has 
	the weight Iwasawa algebra $\Lambda^\circ_k$,
	the deformation functor $\mathcal{D}^\circ_k$ 
	(defined in the same way as $\mathcal{D}_k$), 
	representable by $(R^\circ_k,\rho^\circ_k)$,
	the localized big Hida-Hecke algebra $\mathbf{T}^\circ_k$ of auxiliary conductor $N$ 
	and a Hida family $\mathbf{T}^\circ_k\to A^\circ_k$ passing through
	$\pi_k$ 
	(set $\Lambda=\Lambda^\circ_0$, $\mathcal{D}=\mathcal{D}^\circ_0$, 
	$R=R^\circ_0$, $\mathbf{T}=\mathbf{T}^\circ_0$ and
	$A=A^\circ_0$).
	Over $G_{/\mathbb{Q}_k}$ (see the beginning of Section 2), 
	one has the weight Iwasawa algebra $\Lambda_k=\Lambda^G_k$,
	the deformation functor $\mathcal{D}_k$,
	representable by $(R_k,\rho_k)=(R^G_k,\rho^G_k)$,
	the localized Hida-Hecke algebra $\mathbf{T}_k=\mathbf{T}^G_k$ 
	of auxiliary conductor $N$ and
	a Hida family $\mathbf{T}_k\to A_k=A^G_k$ passing through $\Pi_k$.
    For any $k>k'$, 
    set $\Lambda_{k,k'}=\mathrm{Im} (\Lambda_{k}\to\Lambda_{k'})$.
    By the definition of $\Lambda_{k'}$, 
	$\Lambda_{k'}$ is a complete intersection
	$\Lambda_{k,k'}$-algebra.

	The cyclic base change
	${\mathrm{GL}_{2/\mathbb{Q}}\leadsto\mathrm{GL}_{2/\mathbb{Q}_k}}$ 
	gives homomorphisms of Hecke algebras
	$\mathbf{T}^\circ_k\to\mathbf{T}$
	(the same as in Proposition \ref{cyclic} 
	or using directly \cite[Chapter 3]{ArthurClozel} 
	for the correspondences on Hecke parameters), 
	resp. of Hida families $A^\circ_k\to A$,
	of Iwasawa algebras 
	$\Lambda^\circ_k\to\Lambda$
	(induced by the norm map
	$\mathcal{O}_{\mathfrak{p}_k}^\times\to\mathbb{Z}_p^\times$. 
	See Proposition \ref{cyclic}).
	The composition map
	$\mathbf{T}_k^\circ\to\mathbf{T}\to A$
	induces a morphism of $B$-algebras
	$\tilde{\lambda}_k^\circ
	\colon
	\tilde{\mathbf{T}}_k^\circ
	=\mathbf{T}^\circ_k\otimes_{\Lambda^\circ_k}\tilde{A}
	\to \tilde{A}$.
	
	The composition of Langlands functorialities
	$\mathrm{GL}_{2/\mathbb{Q}_k}\leadsto
	\mathrm{GL}_{n/\mathbb{Q}_k}\leadsto
	G_{/\mathbb{Q}_k}$ 
	gives homomorphisms of Hecke algebras 
	${\theta_k=\theta_k^G\colon\mathbf{T}_k\to\mathbf{T}^\circ_k}$
	(the same as \cite[Proposition 2.3]{HidaTilouine}),
	resp. of Hida families $A_k\to A^\circ_k$, 
	of Iwasawa algebras
	$\Lambda_k\to\Lambda^\circ_k$
	(the morphism  $\Lambda_k\to\Lambda^\circ_k$ is surjective by
	\cite[Proposition 2.3]{HidaTilouine}).
	The Galois side of $\mathrm{GL}_{2/\mathbb{Q}_k}\leadsto G_{\mathbb{Q}_k}$
	gives homomorphisms of universal deformation rings 
	$R_k\to R^\circ_k$
	which is compatible with
	$\mathbf{T}_k\to\mathbf{T}_k^\circ$.
	The composition $\Lambda_k\to\Lambda^\circ_k\to\Lambda$ shows that $A$, as well as $\tilde{A}$, is a $\Lambda_k$-algebra. 
	We set $\tilde{\mathbf{T}}_k=\mathbf{T}_k\otimes_{\Lambda_k}\tilde{A}$
	and $\tilde{R}_k=R_k\otimes_{\Lambda_k}\tilde{A}$
	This induces the natural maps
	$\tilde{\lambda}_k=\tilde{\lambda}^G_k\colon\tilde{\mathbf{T}}_k\to\tilde{A}$
	as well as
	$\tilde{\theta}_k=\tilde{\theta}^G_k
	\colon\tilde{\mathbf{T}}_k\to\tilde{\mathbf{T}}_k^\circ$.
	We also denote by $\tilde{\lambda}_k$ the composition
	$\tilde{R}_k
	\simeq\tilde{\mathbf{T}}_k\to\tilde{A}$.
	We can consult the diagram below to have a clear picture 
	of the various morphisms mentioned above.
		\[
	\begin{tikzcd}
	\pi/\mathrm{GL}_{2/\mathbb{Q}}
	\arrow[r, rightsquigarrow]
	&
	\pi_k/\mathrm{GL}_{2/\mathbb{Q}_k}
	\arrow[r, rightsquigarrow,"\mathrm{GL}_{n/\mathbb{Q}_k}"]
	&
	\Pi_k/G_{\mathbb{Q}_k}
	\\
	\tilde{\mathbf{T}}
	\arrow[rd,"\tilde{\lambda}"']
	&
	\tilde{\mathbf{T}}_k^\circ
	\arrow[l]
	\arrow[d,"\tilde{\lambda}_k^\circ"']
	&
	\tilde{\mathbf{T}}_k
	\arrow[ld,"\tilde{\lambda}_k"]
	\arrow[l,"\tilde{\theta}_k"']
	\\
	&
	\tilde{A}
	&
	\end{tikzcd}
	\]

	\begin{remark}
		There is a subtlety in choosing the coefficient field $K$ 
		(and its ring of integers $\mathcal{O}$). 
		Recall that $K$ is a subfield of $\overline{\mathbb{Q}}_p$, 
		of finite degree over $\mathbb{Q}_p$. 
		For each fixed $k$, to establish the Hida theory over $G_{\mathbb{Q}_k}$, 
		we need to assume that 
		the corresponding coefficient field 
		$K_k=K$ contains all the embeddings of
		$\mathbb{Q}_k$ into
		$\overline{\mathbb{Q}}_p$ (see \cite[Section 2]{Geraghty}).
		However, the objects such as $A$, $R_k$, $\Lambda_k$,
		$\rho^G_{\mu'}$,
		can be defined over $\mathcal{O}_{K_0}$, the ring of integers of $K_0$,
		instead of the larger ring $\mathcal{O}_{K_k}$.
		We write these objects defined over $\mathcal{O}_{K_0}$
		temporarily as $A'$, $R_k'$, $\Lambda_k'$
		and $\rho^G_{\mu'}$.
		Since $\mathcal{O}_{K_k}$ is finite flat over $\mathcal{O}_{K_0}$,
		we see that
		$A\simeq A'\otimes_{\mathcal{O}_{K_0}}\mathcal{O}_{K_k}$,
		$R_k\simeq R_k'\otimes_{\mathcal{O}_{K_0}}
		\mathcal{O}_{K_k}$,
		$\Lambda_k\simeq\Lambda_k'
		\otimes_{\mathcal{O}_{K_0}}\mathcal{O}_{K_k}$
		and
		$\rho^G_\mu\simeq\rho^G_{\mu'}\otimes
		_{\mathcal{O}_{K_0}}\mathcal{O}_{K_k}$.
		So the isomorphism $R_k\simeq\mathbf{T}_k$ (see Theorem \ref{R=T})
		becomes
		$R_k'\otimes_{\mathcal{O}_{K_0}}\mathcal{O}_{K_k}\simeq\mathbf{T}_k$.
		Likewise, the isomorphism in Theorem \ref{KahlerSelmer} becomes
		\[
		\Omega_
		{\tilde{R}'_k\otimes\mathcal{O}_{K_k}/\tilde{A}\otimes\mathcal{O}_{K_k}}
		\bigotimes_{\tilde{R}'\otimes\mathcal{O}_{K_k}}
		(\tilde{A}'\otimes\mathcal{O}_{K_k})
		\simeq
		\mathrm{Sel}(\mathbb{Q}_k,\mathrm{Ad}(\rho^G_{\mu'}))^\ast
		\otimes\mathcal{O}_{K_k}
		\]
		(all the tensors are over $\mathcal{O}_{K_0}$ unless the base ring is given.
		Note also that
		$\widetilde{A'}\otimes\mathcal{O}_{K_k}$
		is equal to
		$\tilde{A}$
		in its fractional field
		$\mathrm{Frac}(\widetilde{A'}\otimes\mathcal{O}_{K_k})
		\simeq\mathrm{Frac}(A)$
		by the fact that
		the inclusion map
		$\mathcal{O}_{K_0}\to\mathcal{O}_{K_k}$
		is a normal ring homomorphism
		and \cite[Proposition 19.1.2, Theorem 19.4.2]{SwansonHuneke}).
		Since $\mathcal{O}_{K_k}$ is faithfully flat over $\mathcal{O}_{K_0}$, 
		we see that above isomorphism is valid if and only if 
		the following one is valid
		\[
		\Omega_{\tilde{R}'/\tilde{A}'}\bigotimes_{\tilde{R}'}\tilde{A}'
		\simeq\mathrm{Sel}(\mathbb{Q}_k,\mathrm{Ad}(\rho^G_\mu)).
		\]
		We see that the tensored sequence is exact if and only if 
		the original sequence is exact. 
		The same remark applies to other statements in the following.
		As a result, we will not distinguish the objects defined over
		$\mathcal{O}_{K_0}$ 
		(with a prime $'$ as above)
		from those defined over 
		$\mathcal{O}_{K_k}$.
	\end{remark}
	
	In the following we write
    $B_\infty=\tilde{A}[[S]]
    \simeq\tilde{A}[[\mathrm{Gal}(\mathbb{Q}_\infty/\mathbb{Q})]]$
    (the isomorphism is given by sending
    $1+S$ to the fixed generator $\gamma$),
    $B_k=B_\infty/((1+S)^{p^k}-1)
    \simeq\tilde{A}[\Delta_k]$ 
    and $B=B_0=\tilde{A}$.
	For any $B_\infty$-module $M$, write $M_k=M_{\Delta_k}=M/(\gamma_k-1)M=M/(\gamma^{p^k}-1)M$
	which is a $B_k$-module.

	Write $E_k$ for the $p$-adic completion of the field $\mathbb{Q}_k$ 
	and $\mathcal{O}_k$ the ring of integers of $E_k$.
	Then we have the norm map
	$E_\infty^\times\to E_k^\times$.
	We define 
	\[R^D_k=\mathcal{O}[[\mathrm{Im}(T^{ss}(E_\infty)\to T^{ss}(E_k))]].\]
	We have an equivalent definition of $R^D_k$ 
	using Galois groups given as follows.
	Let $\Gamma_{\mathfrak{p}_k,p}^{ab}$ be the maximal $p$-profinite abelian quotient of the decomposition group $\Gamma_{\mathfrak{p}_k}$,
	$D_k$ the image of $\Gamma_{\mathfrak{p}_{\infty},p}^{ab}$ in
	$\Gamma_{\mathfrak{p}_k,p}^{ab}$ and
	$R^D_k$ the algebra $\mathcal{O}[[(D_k)^{n-1}]]$.
	So $R^D_k$ is an algebra over $\Lambda_{\infty,k}$ by local class field theory.
	Indeed, recall that
	$\Lambda_k=\mathcal{O}[[T^{ss}(\mathcal{O}_k)]]$.
	As a result,
	\[\Lambda_{\infty,k}=\mathrm{Im}(\Lambda_\infty\to\Lambda_k)=\mathcal{O}[[\mathrm{Im}(T^{ss}(\mathcal{O}_\infty)\to T^{ss}(\mathcal{O}_k))]].\]
	The natural inclusion of 
	$\mathrm{Im}(T^{ss}(\mathcal{O}_\infty)\to T^{ss}(\mathcal{O}_k))$ 
	into 
	$\mathrm{Im}(T^{ss}(E_\infty)\to T^{ss}(E_k))$
	gives the $\Lambda_{\infty,k}$-algebra structure of $R^D_k$.
	We choose the uniformiser $\omega_k$ in $E_k$ in a compatible way
	(i.e. $\omega_{k+1}$ is mapped to $\omega_k$ by the norm map
	$E_{k+1}^{\times}\to E_k^\times$). 
	Then it is easy to see that
	\begin{align*}
	R^D_k & \to\Lambda_{\infty,k}[[T_{k,1},T_{k,2},\ldots,T_{k,n}]]/(\prod_{j=1}^n(1+T_{k,j})=1)\\
	\mathrm{diag}(1_{i-1},\omega_k,1_{n-1-i},\omega_k^{-1}) & 
	\mapsto (1+T_{k,i})/(1+T_{k,n})
	\end{align*}
	is an isomorphism of $\mathcal{O}$-algebras.

	On the other hand, since the extension $\mathbb{Q}_{\infty}/\mathbb{Q}_k$ 
	is totally ramified at $p$, we have the following exact sequence
	\[1+\mathfrak{p}_\infty\mathcal{O}_{\mathfrak{p}_\infty}\xrightarrow
	{\mathrm{Norm}}1+\mathfrak{p}_k\mathcal{O}_{\mathfrak{p}_k}\to u_k^{\mathbb{Z}_p}\to1\]
	where $u_k$ corresponds to $\gamma_k=\gamma^{p^k}$ 
	by local class field theory
	(recall that $\gamma$ is a fixed generator of the Galois group 
	$\mathrm{Gal} (\mathbb{Q}_{\infty}/\mathbb{Q})$).
	\begin{align*}
    \Lambda_k
    &
    \to
    \Lambda_{\infty,k}[[X_{k,1},X_{k,2},\ldots,X_{k,n}]]/(\prod_{j=1}^n(1+X_{k,j})=1)
    \\
    \mathrm{diag}(1_{i-1},u_k,1_{n-1-i},u_k^{-1})
    &
    \mapsto
    (1+X_{k,i})/(1+X_{k,n})
\end{align*}
	
	So we can summarize the above results as
	\begin{lemma}\label{basicIwasawaalgebras}
		We have the following isomorphisms of $\mathcal{O}$-algebras for $k<\infty$:
		
		\begin{align*}
		R^D_k\simeq \Lambda_{\infty,k}[[T_{k,1},\ldots,T_{k,n}]]
		/(\prod_{j=1}^n(1+T_{k,j})=1), \\
		\Lambda_k\simeq\Lambda_{\infty,k}[[X_{k,1},\ldots,X_{k,n}]]
		/(\prod_{j=1}^n(1+X_{k,j})=1).
		\end{align*}
	\end{lemma}
	
	We then define an $R^D_k$-algebra structure on $R_k$ as follows:
	recall that the diagonal entries of the upper triangular Galois representation
	$\rho_{\Pi}|_{\Gamma_{\mathfrak{p}_k}}$ are
	$\psi_{k,1},\psi_{k,2},\ldots,\psi_{k,n}$.
	Define a map $R^D_k\to R_k$ by sending
	$ (\sigma_1,\sigma_2,\ldots,\sigma_n)\in \mathrm{Im}(T^{ss}(E_\infty)\to T^{ss}(E_k))$ to
	$\prod_{i=1}^n\psi_{k,n+1-i}(\sigma_i)$.

	Now we give some simple properties concerning 
	these Hecke algebras $\mathbf{T}_k$ 
	and universal deformation rings $R_k$.
	\begin{lemma}
		Assume $\textbf{Big} (\overline{\rho}_{\Pi})$,
		$\textbf{Dist} (\overline{\rho}_{\Pi})$ and
		$\textbf{RegU} (\overline{\rho}_{\Pi})$, 
		then $\tilde{\mathbf{T}}_0$ is reduced and the map
		$\mathbf{T}_0\to\mathbf{T}$ is surjective.
	\end{lemma}
	\begin{proof}
		The first part is the same as \cite[Corollary 6.2]{HidaTilouine}.
		
		For the second part, it is enough to show that the map 
		$R_0\to R$ is surjective. 
		For this, it suffices to prove that for any object $B$ in
		$\mathrm{CNL}_\mathcal{O}$, 
		for any two lifts $\rho_1,\rho_2$ of $\overline{\rho}_\pi$ to $B$ such that
		they are in the same deformation class
		$[\mathrm{Sym}^{n-1}\rho_1]=[\mathrm{Sym}^{n-1}\rho_2]\in\mathcal{D}_0(B)$,
		we must have $[\rho_1]=[\rho_2]\in\mathcal{D}(B)$.
		We have decompositions of representations of $\Gamma_\mathbb{Q}$,
		$\mathrm{Ad}(\mathrm{Sym}^{n-1}\rho_i)\simeq\oplus_{j=0}^{n-1}\mathrm{Sym}^{2j}\rho_i$ for $i=1,2$.
		Since $\mathrm{Sym}^{n-1}\rho_1\simeq\mathrm{Sym}^{n-1}\rho_2$,
		we have
		$\mathrm{Sym}^{2j}(\rho_1)\simeq\mathrm{Sym}^{2j}(\rho_2)$ for $j=0,1,\ldots,n-1$
		by $\textbf{Big}(\overline{\rho}_\pi)$.
		
		We then need to show
		$\mathrm{Sym}^2\rho_1\simeq\mathrm{Sym}^2\rho_2$ 
		implies $\rho_1\simeq\rho_2$. 
		By $\textbf{Dist}(\overline{\rho}_\pi)$, we may assume that
		$\rho_1(\mathrm{Frob}_p)$ and $\rho_2(\mathrm{Frob}_p)$ are diagonal matrices. 
		Since $\mathrm{Sym}^{n-1}\rho_1(\mathrm{Frob}_p)$ is similar to the matrix
		$\mathrm{Sym}^{n-1}\rho_2(\mathrm{Frob}_p)$, we have
		$\rho_1(\mathrm{Frob}_p)=\rho_2(\mathrm{Frob}_p)$
		by $\textbf{Dist}(\overline{\rho}_\Pi)$.
		By $\textbf{Dist}(\overline{\rho}_\pi)$, 
		there exists a matrix
		$g\in1+\mathfrak{m}_BM_3(B)$ such that
		$\mathrm{Sym}^2\rho_1(\sigma)=g\cdot\mathrm{Sym}^2\rho_2(\sigma)\cdot g^{-1}$
		for any $\sigma\in\Gamma_\mathbb{Q}$.
		Setting $\sigma=\mathrm{Frob}_p$ shows that
		$g$ is in fact a diagonal matrix, say, $\mathrm{diag}(g_1,g_2,g_3)$.
		For any $\sigma$, we write
		$\rho_i(\sigma)=\begin{pmatrix}
		a_i & b_i \\ c_i & d_i
		\end{pmatrix}$ with $i=1,2$. Then a simple calculation shows that
		\[g_2^2(a_1b_1c_1d_1)=g_1g_3(a_2b_2c_2d_2),\ g_1g_3(a_1b_1c_1d_1)=g_2^2(a_2b_2c_2d_2).\]
		By $\textbf{Big}(\overline{\rho}_\pi)$, 
		we can choose $\sigma$ such that
		$a_1b_1c_1d_1\in B^\times$. 
		Comparing two sides of the above identities, we get
		$g_2^2=g_1g_3$ (the other possibility $g_2^2=-g_1g_3$ is excluded by
		$g_1\equiv g_2\equiv g_3\equiv1(\mathrm{mod}\ \mathfrak{m}_B)$). 
		This shows that $g=\mathrm{Sym}^2h$ for some diagonal matrix
		$h\in1+\mathfrak{m}_BM_2(B)$.
		From this, we get
		$\mathrm{Sym}^2(h^{-1}\rho_1(\sigma)h\rho_2(\sigma)^{-1})=1$.
		Clearly, we have
		$h^{-1}\rho_1(\sigma)h\rho_2(\sigma)^{-1}\equiv1(\mathfrak{m}_B)$.
		Some simple calculations show that
		$h^{-1}\rho_1(\sigma)h\rho_2(\sigma)^{-1}=1$.
		As a result $\rho_1(\sigma)=h\cdot\rho_2(\sigma)\cdot h^{-1}$ for any
		$\sigma\in\Gamma_\mathbb{Q}$, which concludes the proof.
		
		If $n$ is an even integer, we can replace the last paragraph by the following:
		we have the decomposition of representations of $\Gamma_\mathbb{Q}$,
		$\mathrm{Sym}^{n-1}\rho_i\otimes\mathrm{Sym}^n\rho_i\simeq\rho_i\oplus\mathrm{Sym}^3\rho_i\oplus\ldots\oplus\mathrm{Sym}^{2n-1}\rho_i$ 
		for $i=1,2$. 
		Thus $\rho_1\simeq\rho_2$ by $\textbf{Big}(\overline{\rho}_\pi)$. 
		The assumption $\textbf{Big}(\overline{\rho}_\pi)$ then shows that
		$[\rho_1]=[\rho_2]$ lie in the same deformation class.

	\end{proof}
	\begin{corollary}
		The algebra $\tilde{\mathbf{T}}_k$ is reduced for $k\geq0$.
	\end{corollary}
	\begin{proof}
		Note that $\tilde{\mathbf{T}}_k\simeq \mathbf{T}_0\otimes_{\Lambda_{k,0}} B \simeq\mathbf{T}_0\otimes_{\Lambda_{k,0}} (\Lambda_0\otimes_{\Lambda_0} B )\simeq  (\mathbf{T}_0\otimes_{\Lambda_{0}} B )\otimes_{\Lambda_{k,0}}\Lambda_{0}$.
		We already know that $\mathbf{T}_0\otimes_{\Lambda_0} B $
		is reduced by the preceding lemma and $\Lambda_{0}$
		is a domain and is finite free over $\Lambda_{k,0}$,
		thus $\tilde{\mathbf{T}}_k$ is reduced.
	\end{proof}

    \begin{lemma}
    	The algebra $\tilde{\mathbf{T}}_k$ is complete intersection over $ B $ 
    	for $k\geq0$.
    \end{lemma}
    \begin{proof}
    	By Theorem \ref{R=T}, $\mathbf{T}_k$ 
    	is complete intersection over $\Lambda_k$. 
    	Moreover, $\Lambda_0$ is finite free over $\Lambda_{k,0}$ and $\Lambda_0$
    	surjects onto $\Lambda=\Lambda^\circ_0$.
    	Note that $ B $ is finite flat over $\Lambda$, 
    	thus we conclude that $\tilde{\mathbf{T}}_k=\mathbf{T}_k\otimes_{\Lambda_k} B $ 
    	is complete intersection over $ B $ for $k\geq0$.
    \end{proof}
	\begin{theorem}\label{KahlerSelmer}
		Assume
		$\textbf{Dist} (\overline{\rho}_{\Pi})$ and
		$\textbf{RegU} (\overline{\rho}_{\Pi})$,
		then as $B_k$-modules
		\[\Omega_{\tilde{R}_k/ B }\bigotimes_{\tilde{R}_k} B \simeq\mathrm{Sel} (\mathbb{Q}_k,\mathrm{Ad}(\rho_\mu^G))^{\ast}.\]
	\end{theorem}
	\begin{proof}
		Note that we have an isomorphism of K\"{a}hler differentials
		$\Omega_{\tilde{R}_k/ B }
		\simeq
		\Omega_{R_k/\Lambda_k}\otimes_{\Lambda_k} B $
		(see \cite[Tag.00RV]{StackProject}).
		As a result,
		$\Omega_{\tilde{R}_k/ B }\otimes_{\tilde{R}_k} B
		\simeq
		(\Omega_{R_k/\Lambda_k}\otimes_{R_k}\tilde{R}_k)\bigotimes_{\tilde{R}_k} B
		\simeq
		\Omega_{R_k/\Lambda_k}\otimes_{R_k} B$.
		From this we obtain
		$(\Omega_{\tilde{R}_k/ B }\otimes_{\tilde{R}_k} B )^\ast\simeq\mathrm{Hom}_{R_k}(\Omega_{R_k/\Lambda_k}, B ^\ast)\simeq\mathrm{Der}_{\Lambda_k}(R_k, B ^\ast)$.
		So we only have to show that
		\[\mathrm{Der}_{\Lambda_k}(R_k, B ^\ast)\simeq\mathrm{Sel}(\mathbb{Q}_k,\mathrm{Ad}(\rho^G_\mu)).\]
		
		Write $R_k[ B ^{\ast}]=R_k\oplus \epsilon B ^{\ast}$ 
		with $\epsilon^2=0$.
		There is a natural map
		$\mathrm{Der}_{\Lambda_k}(R_k,B^\ast)\to\mathrm{Hom}_{\Lambda_k}(R_k,R_k[B^\ast])$ 
		given by sending 
		$d\in\mathrm{Der}_{\Lambda_k}(R_k,B^\ast)$ to an element
		$\phi\in\mathrm{Hom}_{\Lambda_k}(R_k,R_k[B^\ast])$ such that
		$\phi(r)=r+\epsilon d(r)$ for any $r\in R_k$.
		It is easy to see that the image of this map is the set of
		$\phi\in\mathrm{Hom}_{\Lambda_k}(R_k,R_k[B^\ast])$ 
		such that $\phi\equiv\mathrm{Id}_{R_k}(\mathrm{mod}\ B^\ast)$.
		This subset is same as the subset $\mathcal{D}'_k(R_k[B^\ast])$ of
		$\mathcal{D}_k(R_k[B^\ast])$ consisting of those classes $[\rho']$ of the form
		$\rho'(\sigma)=(1+\epsilon c_{\rho'}(\sigma))\rho(\sigma)$ for any $\sigma\in\Gamma_{\mathbb{Q}_k}$. 
		Now consider a map $\mathcal{D}'_k(R_k[B^\ast])\to
		H^1(\Gamma_{\mathbb{Q}_k},\mathfrak{gl}_n(R_k)\otimes_{R_k}B^\ast)$, 
		given by $[\rho']\mapsto[c_{\rho'}]$ 
		where $[c_{\rho'}]$ is the cohomology class of $c_{\rho'}$. 
		We can show that this map is injective. 
		Now we want to show that the image of this map is
		$\mathrm{Sel}(\mathbb{Q}_k,\mathrm{Ad}(\rho^G_\mu))$.
		
		We take any lifting $\rho'= (1+\epsilon c_{\rho'})\rho$. 
		By the assumption $\textbf{Dist} (\overline{\rho}_\mu^{ G})$, 
		the $p$-ordinary filtration is well-defined on $\mathcal{A}^n(\rho_\mu)$. 
		For primes $\mathfrak{q}$ 
		of $\mathbb{Q}_k$ dividing 
		$N$, $c_{\rho'} (I_\mathfrak{q})=0$ 
		by $\mathbf{RegU}(\overline{\rho}_\Pi)$.
		Besides, $c_{\rho'}|_{I_{\mathfrak{p}_k}}$ takes values in
		$F^1\mathcal{A}^n_{\mu_k}$. 
		Moreover, if $\rho'$ is upper triangular on $I_{\mathfrak{p}_k}$, 
		then so it is on $\Gamma_{\mathfrak{p}_k}$. 
		This shows that $c_{\rho'}$ is an element in
		$\mathrm{Sel} (\mathbb{Q}_k,\mathrm{Ad} (\rho_\mu^{ G}))$.
		
		Conversely, suppose that $c$ is an element in 
		$\mathrm{Sel} (\mathbb{Q}_k,\mathrm{Ad} (\rho_\mu^{ G}))$ (a cocycle).  
		Then $\rho'= (1+\epsilon c)\rho$ defines a unique conjugacy class 
		$[\rho']$ of liftings of $\overline{\rho}$. 
		We then check that $[\rho']$ is a deformation in $\mathcal{D}_k (R_k[ B ^\ast])$. 
		It is enough to show that the characters defined by
		$\rho'|_{\Gamma_{\mathfrak{p}_k}}$ are in the right order. 
		This follows from the uniqueness of the $p$-ordinary filtration 
		which in turn is given by the assumption $\textbf{Dist}(\overline{\rho}_\Pi)$.
	\end{proof}
	
	Combining the above result with Corollary \ref{controlofKahler}, we have
	\begin{corollary}\label{controlSelmer}
		For any $k>k'$, as $B_{k'}$-modules,
		\[\mathrm{Sel} (\mathbb{Q}_k,\mathrm{Ad}(\rho_\mu^{G}))^{\ast}_{\Delta_{k,k'}}\simeq\Omega_{R_{k'}/\Lambda_{k,k'}}\bigotimes_{R_{k'}} B .\]
	\end{corollary}
    \begin{remark}
    	Note that $\Omega_{R_{k'}/\Lambda_{k,k'}}\otimes_{R_{k'}} B $ is not the dual Selmer group $\mathrm{Sel}(\mathbb{Q}_{k'},\mathrm{Ad}(\rho^G_\mu))^\ast$.
    	In other words, the dual Selmer groups over the cyclotomic tower
    	$\mathbb{Q}_\infty/\mathbb{Q}$ do not have the descent property.
    	Later we will modify this K\"{a}hler differential such that there is a descent property 
    	Moreover it is related to the dual Selmer group over each layer in
    	$\mathbb{Q}_\infty/\mathbb{Q}$ 
    	(see the modules $M_k$ in Theorem \ref{pL}).
    \end{remark}
    
	We can also use this result to relate Selmer groups to congruence ideals.
	Recall that we have an isomorphism of $\Lambda_k$-algebras
	$R_k\simeq\mathbf{T}_k$.
	Moreover, the exact control theorem in
	\cite[Lemma 2.7]{HidaTilouine} for
	the universal deformation ring $R_k$
	induces an exact control theorem
	for the big Hida-Hecke algebra $\mathbf{T}_k$ as follows:
	$\mathbf{T}_k/\mathfrak{a}_\lambda\mathbf{T}_k
	\simeq(\mathbf{T}_k)_\lambda$
	where $(\mathbf{T}_k)_\lambda$ is the finite level 
	Hecke algebra of weight $\lambda$
	and $\mathfrak{a}_\lambda$
	is the kernel of the morphism
	$\Lambda_k\to\mathcal{O},
	\mathrm{diag}(t_{\tau,1},t_{\tau,2},\ldots,t_{\tau,n})
	_{\tau\colon\mathbb{Q}_k\to\overline{\mathbb{Q}}_p}
	\mapsto\prod_\tau\prod_{i=1}^nt_{\tau,i}^{\lambda_{\tau,i}}$.
	As in the case of elliptic and Hilbert modular forms
	(see \cite[Corollary 3.2]{HidaIwasawaModules}
	and \cite[Theorem 3.4]{HidaOnPAdicHeckeAlgebras}),
	we deduce that there is a unique Hida family $A_k$
	passing through the automorphic form $\Pi_k$.
	This uniqueness induces an action 
	of $\Delta_{k}=\mathrm{Gal}(\mathbb{Q}_k/\mathbb{Q})$
	on $A_k$ from that 
	on $R_k\simeq\mathbf{T}_k$.
	We write temporarily the Galois representation
	$\rho'\colon\Gamma_{\mathbb{Q}_k}
	\to\mathcal{G}_n(R_k)
	\to\mathcal{G}_n(\tilde{A}_k)$,
	the projection
	$\lambda'\colon R_k\to A_k$
	and 
	$\tilde{\lambda}'\colon 
	R_k\otimes_{\Lambda_k}
	\tilde{A}_k\to \tilde{A}_k$
	with $\tilde{A}_k$ the normal closure of
	$A_k$ in its field of fractions.
	Recall that $R_k$ is a complete intersection $\Lambda_k$-algebra,
	thus $R_k\otimes\tilde{A}_k$ is a complete intersection
	$\tilde{A}_k$-algebra.
	Then by a theorem of J.Tate (see \cite[Theorem 8.7]{HidaTilouine}),
	the congruence ideal $\mathfrak{c}_{\tilde{\lambda}'}$ 
	in $\tilde{A}_k$ of 
	the morphism 
	$\tilde{\lambda}'\colon R_k\otimes\tilde{A}_k\to\tilde{A}_k$
	is equal to the Fitting ideal
	$\mathrm{F}_{\tilde{A}_k}
	(\Omega_{R_k\otimes\tilde{A}_k/\tilde{A}_k}\otimes\tilde{A}_k)$.
	Similarly,
	the congruence ideal
	$\mathfrak{c}_{\tilde{\lambda}_k}$
	in $B$ of $\tilde{\lambda}_k$
	is equal to the Fitting ideal
	$\mathrm{F}_B(\Omega_{\tilde{R}_k/B}\otimes_{\tilde{R}_k})B$.

    The same argument as in Theorem \ref{KahlerSelmer}
    gives an isomorphism
    (under the same conditions as in the theorem)
    \[
    \Omega_{R_k\otimes\tilde{A}_k/\tilde{A}_k}
    \bigotimes_{R_k\otimes\tilde{A}_k}\tilde{A}_k
    \simeq\mathrm{Sel}(\mathbb{Q}_k,\mathrm{Ad}(\rho'))^\ast.
    \]
	
	\begin{corollary}\label{congruenceSelmer}
	Assume $\textbf{Big} (\overline{\rho}_{\Pi})$,
	$\textbf{Dist} (\overline{\rho}_{\Pi})$ and
	$\textbf{RegU} (\overline{\rho}_{\Pi})$,
	then as ideals in $\tilde{A}_k$,
	\[
	\mathfrak{c}_{\tilde{\lambda}'}=
	\mathrm{F}_{\tilde{A}_k}
	(\mathrm{Sel}(\mathbb{Q}_k,\mathrm{Ad}(\rho'))^\ast).
	\]
	
	Similarly, under the same hypotheses, 
	for the morphism $\tilde{\lambda}_k\colon\tilde{R}_k\to B$,
	we have, as ideals in $B$,
	\[\mathfrak{c}_{\tilde{\lambda}_k}=
	\mathrm{F}_B (\mathrm{Sel}
	(\mathbb{Q}_k,\mathrm{Ad}(\rho_\mu^{G}))^{\ast}).\]
	\end{corollary}

    Next we will use two $B_k$-modules
    $\Omega_{\Lambda_k/\Lambda_{\infty,k}}\otimes_{\Lambda_k}B$
    and
    $\Omega_{R^D_k/\Lambda_{\infty,k}}\otimes_{R^D_k}B$.
    By Lemma \ref{basicIwasawaalgebras}, 
    we have isomorphisms of $ B $-modules
    \[\Omega_{R^D_k/\Lambda_{\infty,k}}\bigotimes_{R^D_k} B 
    \simeq\bigoplus_{i=1}^n B d\log(1+T_{k,i})/(\sum_{j=1}^nd\log(1+T_{k,j})=0)
    \simeq\bigoplus_{i=1}^{n-1} B d\log(1+T_{k,i})
    \simeq
    \bigoplus_{i=1}^{n-1}BdT_{k,i},\]
    \[
    \Omega_{\Lambda_k/\Lambda_{\infty,k}}\bigotimes_{\Lambda_k} B
    \simeq\bigoplus_{j=1}^n B
    d\log(1+X_{k,j})/(\sum_{i=1}^nd\log(1+X_{k,i})=0)
    \simeq\bigoplus_{j=1}^{n-1} B
    (d\log(1+X_{k,n})-d\log(1+X_{k,j}))
    \simeq
    \bigoplus_{j=1}^{n-1}BdX_{k,j}.
    \]
    
    We will use  both the two basis
    $(d\log(1+T_{k,i}))_{i=1}^{n-1}$
    and $(dT_{k,i})_{i=1}^{n-1}$
    for the $B$-module
    $\Omega_{R^D_k/\Lambda_{\infty,k}}\otimes B$.
    Likewise, we use both the two basis 
    $(d\log(1+X_{k,n})-d\log(1+X_{k,j}))_{j=1}^{n-1}$
    and
    $(dX_{k,j})_{j=1}^{n-1}$
    for the module
    $\Omega_{\Lambda_k/\Lambda_{\infty,k}}\otimes B$.
    In Definition \ref{definitionofL},
    we choose
    $(d\log(1+T_{k,i}))$
    and
    $(d\log(1+X_{k,n})-d\log(1+X_{k,j})$
    as basis
    mainly for the comparison
    in Section 5.
    Under these two basis, we can show that 
    the element
    $\mathcal{L}(\mathbb{Q},\mathrm{Ad}(\rho_\mu^G))$
    is \emph{equal} to the $\mathcal{L}$-invariant given in
    \cite{HarronJorza}.
    
	\begin{proposition}\label{injectiveKahler1}
		Assume $\textbf{Big}(\overline{\rho}_\Pi)$,
		then $\mathrm{Sel} (\mathbb{Q}_k,\mathrm{Ad} (\rho_\mu^{ G}))^{\ast}$
		is a finitely generated torsion 
		$ B $-module and the following sequence is exact,
		on identifying
		$\mathrm{Sel} (\mathbb{Q}_k,\mathrm{Ad} (\rho_\mu^{ G}))^{\ast}
		\simeq\Omega_{R_k/\Lambda_k}\otimes  B $
		\begin{equation}\label{exactsequence1}
		0\to\Omega_{\Lambda_k/\Lambda_{\infty,k}}\bigotimes_{\Lambda_k} B \to\Omega_{R_k/\Lambda_{\infty,k}}\bigotimes_{R_k} B \to \mathrm{Sel} (\mathbb{Q}_k,\mathrm{Ad} (\rho_\mu^{ G}))^{\ast}\to0.
		\end{equation}
	\end{proposition}
	\begin{proof}
		To simplify notations, denote by $M$ the 
		$\Gamma_{\mathbb{Q}_k}$-module $\mathrm{Ad}(\rho_\mu^{G})$, 
		which is a free $ B $-module. 
		Let $P$ be an arithmetic point of
		$\mathrm{Spec}(B)$.
		Denote by $M_P$ the quotient $\mathrm{Ad}(\rho^G_\mu)(\mathrm{mod}\ P)$.
		
		We first show that
		$\mathrm{Sel} (\mathbb{Q}_k,M)^{\ast}\otimes_{ B }
		B/P\simeq\mathrm{Sel} (\mathbb{Q}_k,M_P)^{\ast}$.
		Indeed, by exactness of the Pontryagin duality, 
		the obvious exact sequence 
		${0\to P\to  B \to  B /P\to 0}$ 
		gives the another one
		$0\to ( B /P)^\ast \to B ^\ast\to P^\ast\to0$.
		This shows that $( B /P)^\ast\simeq B ^\ast[P]$ where $ B ^\ast
		[P]$ is the sub-$ B $-module of $ B ^\ast$ annihilated by $P$. 
		As a result, applying $M\otimes_A\bullet$ to this exact sequence and then taking group cohomology $H^\bullet(\mathbb{Q}_k,\bullet)$ of the exact sequence, we get a long exact sequence
		\[H^0(\mathbb{Q}_k,M\bigotimes_A B ^\ast)\to
		H^1(\mathbb{Q}_k,M\bigotimes_A( B /P)^\ast)\to
		H^1(\mathbb{Q}_k,M\bigotimes_A B ^\ast)\to
		H^1(\mathbb{Q}_k,M\bigotimes_A B ^\ast)\]
		By $\textbf{Big}(\rho_\Pi)$, $H^0(\mathbb{Q}_k,M\otimes_A B ^\ast)=0$. Taking the Pontryagin dual of the above sequence, we get
		\[H^1(\mathbb{Q}_k,M\bigotimes_A B ^\ast)^\ast[P]\simeq H^1(\mathbb{Q}_k,M\bigotimes_A( B /P)^\ast)^\ast.\]
		The same argument applies to the local Galois cohomology groups. 
		Thus we obtain the following
		\[\mathrm{Sel} (\mathbb{Q}_k,M)^{\ast}\bigotimes_{ B } B /P\simeq\mathrm{Sel} (\mathbb{Q}_k,M_P)^{\ast}.\]
		
		By definition, $ B /P\simeq\mathcal{O}$. From this we have
		$\mathrm{Sel} (\mathbb{Q}_k,M_P)^{\ast}\simeq\Omega_{\mathbf{T}_P/\mathcal{O}}\otimes_{\mathbf{T}_P}\mathcal{O}$
		where $\mathbf{T}_P$ is the Hecke algebra of the same weight as the arithmetic point $P$.
		Since $\mathbf{T}_P$ is finite free over $\mathcal{O}$, we see that $\mathrm{Sel} (\mathbb{Q}_k,M_P)^{\ast}$
		is a finite $\mathcal{O}$-module,
		thus $\mathrm{Sel} (\mathbb{Q}_k,M)^{\ast}$ is a torsion
		$ B $-module.
		
		By Lemma \ref{basicIwasawaalgebras},
		$\Omega_{\Lambda_k/\Lambda_{\infty,k}}\otimes_{\Lambda_k} B \simeq B ^{n-1}$. Moreover, $R_k\simeq\mathbf{T}_k$ is finite free over $\Lambda_k$.
		The $ B $-torsionness of $\mathrm{Sel} (\mathbb{Q}_k,M)^{\ast}$
		gives the left exactness of (\ref{exactsequence1}).
	\end{proof}
	
	Now we consider another exact sequence derived from the first fundamental exact sequence of K\"{a}hler differentials
	\[\Omega_{R^D_k/\Lambda_{\infty,k}}\bigotimes_{R^D_k} B \to\Omega_{R_k/\Lambda_{\infty,k}}\bigotimes_{R_k} B \to\Omega_{R_k/R^D_k}\bigotimes_{R_k} B \to0.\]
	
	We want to know when this sequence is injective on the left.
	Since $\mathrm{Sel} (\mathbb{Q}_k,\mathrm{Ad} (\rho_\mu^{ G}))^{\ast}$ is
	$ B $-torsion,
	we can find some non-zero element $0\neq\eta\in B $ such that
	${\eta\cdot\mathrm{Sel} (\mathbb{Q}_k,\mathrm{Ad} (\rho_\mu^{ G}))^{\ast}=0}$.
	This shows that $\eta\cdot d\log(1+T_{k,i})=\eta\cdot dT_{k,i}/(1+T_{k,i})$
	lies in the $B$-module $\Omega_{\Lambda_k/\Lambda_{\infty,k}}\otimes B$.
	This means that $\eta\cdot d\log(1+T_{k,i})$
	is a $B$-linear combination of the basis
	$({d\log(1+X_{k,n})-d\log(1+X_{k,j})})_{j=1}^{n-1}$.
	From this, we give the following definition.
	\begin{definition}\label{definitionofL}
		We write  $L(\mathbb{Q},\mathrm{Ad}(\rho^G_\mu))$ 
		for the $(n-1)\times(n-1)$matrix
		\[
		\begin{pmatrix}
		\frac{\partial \log(1+T_{k,1})}{\partial \log(1+X_{k,n})}-
		\frac{\partial \log(1+T_{k,1})}{\partial \log(1+X_{k,1})}
		&
		\cdots
		&
		\frac{\partial \log(1+T_{k,1})}{\partial \log(1+X_{k,n})}-
		\frac{\partial \log(1+T_{k,1})}{\partial \log(1+X_{k,n-1})}
		\\
		\vdots
		&
		\ddots
		&
		\vdots
		\\
		\frac{\partial \log(1+T_{k,n-1})}{\partial \log(1+X_{k,n})}-
		\frac{\partial \log(1+T_{k,n-1})}{\partial \log(1+X_{k,1})}
		&
		\cdots
		&
		\frac{\partial \log(1+T_{k,n-1})}{\partial \log(1+X_{k,n})}-
		\frac{\partial \log(1+T_{k,n-1})}{\partial \log(1+X_{k,n-1})}
		\end{pmatrix}.
		\]
		The determinant
		$\mathrm{det}(L(\mathbb{Q},\mathrm{Ad}(\rho^G_\mu)))\in\mathrm{Frac}(B)$
		is written as
		$\mathcal{L} (\mathbb{Q}_k,\mathrm{Ad}(\rho_\mu^{G}))$.
		The $0$-th Fitting ideal of the $(n-1)\times(2n-2)$-matrix
		$\begin{pmatrix}
		S1_{n-1} & L(\mathbb{Q}_k,\mathrm{Ad}(\rho^G_\mu))
		\end{pmatrix}$ over $B_k$
		is written as 
		$I_k=I^G_k$.
	\end{definition}

    By definition, $I_k=\sum_{i=0}^{n-1}S^{i}\mathrm{F}^{(i)}_B(L)B_k$
    where $L=L(\mathbb{Q}_k,\mathrm{Ad}(\rho^G_\mu))$.
    
    \begin{lemma}\label{L-invariant:L_0ImpliesL_1}
    	We have
    	$\mathcal{L}(\mathbb{Q}_k,\mathrm{Ad}(\rho^G_\mu))
    	=p^{-k(n-1)}\mathcal{L}(\mathbb{Q},\mathrm{Ad}(\rho^G_\mu))$ ($k\geq0$).
    \end{lemma}
    \begin{proof}
    	The map 
    	$\Lambda_k\to\Lambda_0$
    	is induced by the norm $\mathbb{Q}_k^{\times}\to\mathbb{Q}^{\times}$, thus $1+X_{k,i}$ is sent to $ (1+X_{0,i})^{p^k}$ ($1\leq i \leq n$).
    	Similarly, in the map
    	$R^D_k\to R^D_0$,
    	$1+T_{k,j}$ is sent to $1+T_{0,j}$ ($1\leq j\leq n$).
    	So $\mathcal{L}(\mathbb{Q}_k,\mathrm{Ad}(\rho_\mu^G))$ equal to
    	$p^{-k(n-1)}\mathcal{L}(\mathbb{Q},\mathrm{Ad}(\rho_\mu^G))$.
    \end{proof}
    
	\begin{lemma}\label{injectiveKahler2}
		The element 
		$\mathcal{L} (\mathbb{Q}_k,\mathrm{Ad} (\rho_\mu^{ G}))$
		does not vanish
		if and only if the following sequence is exact
		\begin{equation}\label{exactsequence2}
		0\to\Omega_{R^D_k/\Lambda_{\infty,k}}\bigotimes_{R^D_k} B \to\Omega_{R_k/\Lambda_{\infty,k}}\bigotimes_{R_k} B \to\Omega_{R_k/R^D_k}\bigotimes_{R_k} B \to0.
		\end{equation}
	\end{lemma}
	\begin{proof}
		By Proposition \ref{injectiveKahler1},
		the $ B_k  $-module
		$\Omega_{\Lambda_k/\Lambda_{\infty,k}}\otimes B$ 
		injects into $\Omega_{R_k/\Lambda_{\infty,k}}\otimes B $. 
		By the assumption 
		$\mathcal{L}(\mathbb{Q}_k,\mathrm{Ad}(\rho_\mu^G))\neq0$, 
		the above injection implies that
		$\Omega_{R^D_k/\Lambda_{\infty,k}}\otimes B$
		injects into 
		$\Omega_{R_k/\Lambda_{\infty,k}}\otimes B $.
	\end{proof}

    From the above two lemmas, we see that
    if $I_0\neq0$, then the sequence
    \ref{exactsequence2} is exact for all $k\geq0$.

    Now we come to the main technical theorem of this paper.
    
	\begin{theorem}\label{pL}
		Assume $\textbf{Big} (\overline{\rho}_{\Pi})$,
		$\textbf{Dist} (\overline{\rho}_{\Pi})$,
		$\textbf{RegU} (\overline{\rho}_{\Pi})$,
		$\mathcal{L} (\mathbb{Q},\mathrm{Ad}(\rho_\mu^{G}))\neq0$
		and $B$ regular, then
		$\mathrm{Sel} (\mathbb{Q}_{\infty},\mathrm{Ad} (\rho_\mu^{ G}))^{\ast}$
		is a finitely generated torsion $B_\infty$-module.
		Moreover the characteristic power series $L^{alg}_p (\mathrm{Ad} (\rho_\mu^{ G}),S):=\mathrm{char}_{B_\infty}(\mathrm{Sel}
		(\mathbb{Q}_{\infty},\mathrm{Ad} (\rho_\mu^{ G}))^{\ast})$
		has a trivial zero at $S=0$ of order $n-1$.
		For any $k\geq0$ and any prime ideal $P$ in $B_k$ of height one, we have
		\[
		((\frac{L^{alg}_p (\mathrm{Ad}
		(\rho_\mu^{ G}),S)}{S^{n-1}}B_\infty)_{\Gamma_k})_P
		=I_k
		\mathrm{F}_{B_k}(\mathrm{Sel}
		(\mathbb{Q}_k,\mathrm{Ad}(\rho_\mu^{G}))^{\ast})_P.
		\]
	\end{theorem}
	\begin{proof}
		We divide the proof into several steps.
		
		\textbf{Step.1}: We will isolate the trivial zeros of the characteristic power series of the dual Selmer group over $\mathbb{Q}_\infty$ (in a special case, see below (\ref{ImportantInclusion})).
		
		Recall that we have the following exact sequences derived from the exact sequences (\ref{exactsequence1}) and (\ref{exactsequence2})
		\[0\to\bigoplus_{i=1}^{n-1}BdX_{k,i}\xrightarrow{\iota^I_k} \Omega_{R_k/\Lambda_{\infty,k}}\bigotimes_{R_k}B\to\Omega_{R_k/\Lambda_k}\bigotimes_{R_k}B\to0,\]
		\[0\to\bigoplus_{j=1}^{n-1}BdT_{k,j}\xrightarrow{\iota^D_k}\Omega_{R_k/\Lambda_{\infty,k}}\bigotimes_{R_k}B\to\Omega_{R_k/R^D_k}\bigotimes_{R_k}B\to0.\]
		
	    To simplify notations, we write
		\[M_k=\Omega_{R_k/R^D_k}\bigotimes_{R_k}B.\]
		
		Then one has the following commutative diagrams
		\[\xymatrix{\Omega_{R_k/\Lambda_{\infty,k}}\bigotimes B\ar[r]^{b_k}\ar[d]_{g_k} & \Omega_{B/\Lambda_{\infty,k}}\ar[d]^{h_k}\\
		M_k\ar[r]^{d_k} & \Omega_{B/R^D_{k}}}\]
		Recall that $\Lambda_{\infty,k}=\mathrm{Im}(\Lambda_\infty\to\Lambda_k)$.
		This gives
		$\Omega_{B/\Lambda_{\infty,k}}=\Omega_{B/\Lambda_{\infty}}$.
		By Lemma \ref{basicIwasawaalgebras}, we see that
		$\Omega_{B/R^D_k}=\Omega_{B/R^D_\infty}$.
		So $h_k=h_\infty$.
		From this one has the following exact sequence
		\[0\to\Omega_{R_{\infty}/\Lambda_{\infty}}\xrightarrow{ (g_{\infty},b_k)}M_{\infty}\bigoplus  \Omega_{B/\Lambda_{\infty,k}}\xrightarrow{d_{\infty}-h_k}\Omega_{B/R^D_k}\to0.\]
		
		We first assume that in $\Omega_{R_k/\Lambda_{\infty,k}}\otimes B$
		there is an inclusion
		\begin{equation}\label{ImportantInclusion}
		\iota_k\colon  \mathrm{Im} (\iota^D_k)\hookrightarrow\mathrm{Im} (\iota^I_k).
		\end{equation}
		This assumption implies that the matrix $(\frac{\partial T_{k,i}}{\partial X_{k,j}})$ has entries in $B$.
		
		By the hypothesis
		$\mathcal{L} (\mathbb{Q},\mathrm{Ad}(\rho_\mu^{G}))\neq0$,
		$\Omega_{B/R^D_0}$ is $B$-torsion,
		so $\Omega_{B/R^D_{\infty}}$ is $B_\infty$-pseudo-isomorphic to $0$.
		Moreover, $\mathcal{L} (\mathbb{Q},\mathrm{Ad}(\rho_\mu^{G}))\neq0$ also implies that $\Omega_{B/\Lambda_{\infty,0}}/\mathrm{Im} (\iota^D_0)$
		is torsion so it is $B_\infty$-pseudo-isomorphic to $0$.
		So we deduce the following sequence of pseudo-isomorphic
		$B_\infty$-modules 
		(we write $\sim$ for the relation of pseudo-isomorphism)
		\[\mathrm{Sel} (\mathbb{Q}_\infty,\mathrm{Ad}(\rho^G_\mu))^{\ast}
		\sim\Omega_{R_{\infty}/\Lambda_{\infty}}\bigotimes B
		\sim M_{\infty}\bigoplus\Omega_{B/\Lambda_{\infty,0}}
		\sim M_{\infty}\bigoplus\mathrm{Im} (\iota^D_{\mathbb{Q}})
		\sim M_{\infty}\bigoplus  B^{n-1}.\]
		
		By Corollary \ref{controlofKahler}, we have $ (M_{\infty})_{\Gamma_k}\simeq M_k$.
		Now the hypothesis $\mathcal{L}
		(\mathbb{Q},\mathrm{Ad}(\rho_\mu^{G}))\neq 0$ implies that
		$M_0$ is $B$-torsion.
		This shows that
		$\mathrm{Sel} (\mathbb{Q}_{\infty},\mathrm{Ad}(\rho_\mu^{G}))^{\ast}$
		is a finitely generated torsion $B_\infty$-module.
		
		We write $\Phi (S)$ for the characteristic power series of $M_{\infty}$.
		It is clear from the above that
		\[L_p^{alg} (\mathrm{Ad} (\rho_\mu^{ G}),S)=\Phi (S)S^{n-1}\]
		and
		$\Phi (0)\neq0$ (otherwise $M_{\infty}/SM_{\infty}=M_0$ will not be
		$B$-torsion).
		
		\textbf{Step.2}: We will study the relation between $M_k$ and the Selmer groups.
		
		By the assumption (\ref{ImportantInclusion}), we have the following commutative diagram
		\begin{equation}\label{CommutativeDiagram}
		\xymatrix{
			0\ar[r]
			&
			\mathrm{Im} (\iota^D_k)\ar[r]\ar[d]^{\iota_k}
			&
			\Omega_{R_k/\Lambda_{\infty,k}}\bigotimes B\ar[r]\ar[d]^{\mathrm{Id}}
			&
			\Omega_{R_k/R^D_k}\bigotimes B\ar[r]\ar[d]
			&
			0
			\\
			0\ar[r]
			&
			\mathrm{Im} (\iota^I_k)\ar[r]
			&
			\Omega_{R_k/\Lambda_{\infty,k}}\bigotimes B\ar[r]
			&
			\Omega_{R_k/\Lambda_k}\bigotimes B \ar[r]
			&
			0
		}
		\end{equation}
		
		We write
		$N_k=\mathrm{Coker}
		(\mathrm{Im} (\iota^D_k)\hookrightarrow\mathrm{Im} (\iota^I_k))$,
		which is a $B_k$-module. 
		Moreover the definition of $N_k$ shows that the homological dimension $\mathrm{hdim}_{ B } (N_k)\leq1$.
		Applying the snake lemma to the above diagram, 
		we get the following exact sequence
		\[0\to N_k\to M_k\to\Omega_{R_k/\Lambda_k}\otimes B \to0.\]
		By Proposition \ref{PropertiesOfFittingIdeal} (3), we have
		$\mathrm{F}_{B_k}(M_k)\supset
		\mathrm{F}_{B_k}(N_k)
		\mathrm{F}_{B_k}(\Omega_{R_k/\Lambda_k}\otimes B)$.
		Note that if each term $X$ in the diagram
		\ref{CommutativeDiagram}
		is replaced by $(X)_{\Delta_k}$,
		the diagram is still a commutative diagram.
		Applying the snake lemma to this new diagram
		gives 
		$0\to (N_k)_{\Delta_k}\to M_0\to\Omega_{R_k/\Lambda_{k,0}}\otimes B$.
		Since $R_k$ is a complete intersection $\Lambda_k$-algebra
		and $\Lambda_k$ is a complete intersection $\Lambda_{k,0}$-algebra,
		$R_k$ is thus a complete intersection $\Lambda_{k,0}$-algebra. 
		By Proposition \ref{PropertiesOfFittingIdeal} (3), we have
		$\mathrm{F}_B(M_0)
		=\mathrm{F}_B((N_k)_{\Delta_k})
		\mathrm{F}(\Omega_{R_k/\Lambda_{k,0}}\otimes B)$.
		Recall $\mathrm{F}_B(M_0)=(\mathrm{F}_{B_k}(M_k))_{\Delta_k}$
		(similarly for the other factors),
		Nakayama's lemma shows
		\[\mathrm{F}_{B_k}(M_k)=\mathrm{F}_{B_k}(N_k)\mathrm{F}_{B_k}
		(\mathrm{Sel}(\mathbb{Q}_k,\mathrm{Ad}(\rho^G_\mu))^\ast).\]

		Let's explicate the Fitting ideal $\mathrm{F}_{B_k}(N_k)$. 
		Write the morphisms defining $N_k$ as
		${B^{n-1}\xrightarrow{\iota_k}B^{n-1}\xrightarrow{\beta_k}N_k\to0}$.
		It is easy to see that the actions of 
		$\Delta_k$ on 
		$N_k$, $\mathrm{Im}(\iota_k^D)$ and $\iota_k^I$ 
		are trivial.
		Now we can extend the morphism $\iota_k$ to 
		$\iota_k\colon B_k^{n-1}\to B_k^{n-1}$ 
		by sending an element 
		$g=\sum_jg_jS^j$ to $\iota_k(g)=\sum_j\iota_k(g_j)S^j$ 
		where each $g_j$ is an element in $B^{n-1}$. 
		We can also extend $\beta_k$ to $\beta_k\colon B_k^{n-1}\to N_k$ 
		by sending $g=\sum_jg_jS^j$ to $\beta_k(g)=\beta_k(g_0)$. 
		It is easy to verify that $\iota_k$ and $\beta_k$ are morphisms of
		$B_k$-modules. 
		We define define a morphism 
		$\iota_k'\colon B_k^{n-1}\oplus B_k^{n-1}\to B_k^{n-1}$ 
		by sending $(f,g)$ to $Sf+\iota_k(g)$.
		Then we have an exact sequence 
		$B_k^{n-1}\oplus
		B_k^{n-1}
		\xrightarrow{\iota'_k}B_k^{n-1}
		\xrightarrow{\beta_k}N_k\to0$
		of $B_k$-modules.
		So using this presentation for the $B_k$-module $N_k$, 
		we get
		$\mathrm{F}_{B_k}(N_k)=I_k$.

		\textbf{Step.3}: We will relate $\mathrm{F}_{B_k}(M_k)$ to the characteristic power series $L^{alg}_p(\mathrm{Ad}(\rho^G_\mu),S)$. 
		Since $M_0$ is $ B $-torsion, we have the relation of depth
		\[\mathrm{depth}_{B_\infty} (M_{\infty})=\mathrm{depth}_{ B } (M_0)+1.\]
		
		Recall $R_0$ is complete intersection over $\Lambda_0$, 
		so we can write this as 
		$R_0=\Lambda_0[[Y_1,Y_2,\ldots,Y_r]]
		/(f_1,f_2,\ldots,f_r)$,
		where $ (f_1,f_2,\ldots,f_r)$ 
		is a regular sequence of 
		$\Lambda_0[[Y_1,\ldots,Y_r]]$. 
		As in the proof of Proposition \ref{PropertiesOfFittingIdeal} (8), we have the following exact sequence
		\[\bigoplus_i B\cdot df_i\to\bigoplus_i B\cdot
		dY_i\to\mathrm{Sel} (\mathbb{Q},\mathrm{Ad} (\rho_\mu^{ G}))^{\ast}\to0.\]
		The $ B $-torsionness of $\mathrm{Sel} (\mathbb{Q},\mathrm{Ad} (\rho_\mu^{ G}))^{\ast}$ 
		shows that the above exact sequence is in fact injective on the left.
		So the homological dimension
		$\mathrm{hdim}_{ B } (\mathrm{Sel} (\mathbb{Q},\mathrm{Ad} (\rho_\mu^{ G}))^{\ast})\leq1$.
		Combining with the fact that $\mathrm{hdim}_{ B } (N_0)\leq1$ gives
		\[\mathrm{hdim}_{ B } (M_0)\leq1.\]
		By the Auslander-Buchsbaum formula,
		we have $\mathrm{hdim}_{ B } (M_0)
		+\mathrm{depth}_{ B } (M_0)=\mathrm{depth}_{ B } ( B )$.
		Note that this identity uses only the fact that
		$\mathrm{hdim}_{ B } (M_0)<\infty$,
		which follows from $ B $ being regular by \cite[Theorem 19.2]{MatsumuraCommutativeRingTheory}.
		On the other hand, since $B_\infty$ is also regular,
		$\mathrm{hdim}_{B_\infty} (M_{\infty})<\infty$.
		By the Auslander-Buchsbaum formula, we get
		$\mathrm{hdim}_{B_\infty} (M_{\infty})+\mathrm{depth}_{B_\infty} (M_{\infty})=\mathrm{depth}_{B_\infty} (B_\infty)$.
		Using the fact that
		$\mathrm{depth}_{B_\infty} (B_\infty)=\mathrm{depth}_{ B } ( B )+1$,
		we deduce that
		\[\mathrm{hdim}_{B_\infty} (M_{\infty})=\mathrm{hdim}_{ B } (M_0)\leq1.\]
		
		This implies that $M_{\infty}$, resp. $M_0$, has no non-zero $B_\infty$-submodule, 
		resp. $ B $-submodule, pseudo-isomorphic to $0$.
		Now we can apply Lemma \ref{PropertyOfCharacteristicIdeal}, 
		Proposition \ref{PropertiesOfFittingIdeal} (1) and (6): 
		for any prime ideal $P$ in $B$ of height one, we have
		\[
		((\frac{L^{alg}_p(\mathrm{Ad}(\rho^G_\mu),S)}{S^{n-1}}B_\infty)_\Gamma)_P
		=((M_\infty)_\Gamma)_P=(M_0)_P=
		I_0
		\mathrm{F}_{B}
		(\mathrm{Sel}(\mathbb{Q},\mathrm{Ad}(\rho^G_\mu))^\ast)_P
		=I_0
		\mathrm{char}_B
		(\mathrm{Sel}(\mathbb{Q},\mathrm{Ad}(\rho^G_\mu))^\ast)_P.\]
		
		For general $k\geq0$, for any prime ideal $P$ in $B_k$
		of height one, 
		we have the following, 
		which finishes the proof under the assumption (\ref{ImportantInclusion})
		\[
		((\frac{L^{alg}_p(\mathrm{Ad}(\rho^G_\mu),S)}{S^{n-1}}B_\infty)_{\Gamma_k})_P
		=((M_\infty)_{\Gamma_k})_P=(M_k)_P
		=I_k
		\mathrm{F}_{B_k}
		(\mathrm{Sel}(\mathbb{Q}_k,\mathrm{Ad}(\rho^G_\mu))^\ast)_P.
		\]

		\textbf{Step.4}: We will study the general case without assuming
		(\ref{ImportantInclusion}).
		We can choose $n-1$ elements $T_{k,1}',T_{k,2}',\ldots,T_{k,n-1}'$ 
		in the intersection $\Lambda_k\bigcap R^D_k$ (inside $R_k$)
		such that $R^D_k$ is a finite algebra over the formal power series algebra
		${R^d_k=\Lambda_{\infty,k}[[T_{k,1}',\ldots,T_{k,n-1}']]}$ 
		(use that $R_k$ is of relative dimension $n-1$ over
		$\Lambda_{\infty,k}$ and Noether normalization lemma. 
		Moreover for any $k'<k$, we set
		$R^d_{k'}=\Lambda_{\infty,k'}[[T'_{k',1},\ldots,T'_{k',n-1}]]$ 
		where $T'_{k',j}$ are the images of $T'_{k,j}$ under the map
		$\Lambda_k\to\Lambda_{k'}$).
		In the same manner, one can an element
		$\mathcal{L}' (\mathbb{Q}_k,\mathrm{Ad}(\rho_\mu^{G}))=\mathrm{det} (\frac{dT_{k,i}'}{dX_{k,j}})\in B$,
		an ideal
		$I'_k\subset\mathrm{Q}(B)$,
		the cokernel
		${N_k'=\mathrm{Coker} (\mathrm{Im} (\iota^d_k)\to\mathrm{Im}(\iota^I_k))}$
		and the exact sequence
		$0\to N_k'\to\Omega_{R_k/R^d_k}\bigotimes B \to\Omega_{R_k/\Lambda_k}\bigotimes B \to0.$
		From this, we have
		\begin{equation}\label{Step4:AuxiliaryIdentity}
		\mathrm{F}_{B_k}(\Omega_{R_k/R^d_k}\bigotimes B )=\mathrm{F}_{B_k}(N_k')\mathrm{F}_{B_k}(\mathrm{Sel} (\mathbb{Q}_k,\mathrm{Ad}(\rho_\mu^{G}))^{\ast}).
		\end{equation}
		
		One the other hand, we have the following two exact sequences with maps from one to the other (the first sequence is injective on the left because $R^d_k$ injects into $\Lambda_k$)
		\[
		\begin{tikzcd}
		0\arrow[r] & \Omega_{R^d_k/\Lambda_{\infty,k}}\bigotimes B \arrow[r]\arrow[d] & \Omega_{R_k/\Lambda_{\infty,k}}\bigotimes B \arrow[r]\arrow[d,"\mathrm{Id}"] & \Omega_{R_k/R^d_k}\bigotimes B \arrow[r]\arrow[d] & 0\\
		0\arrow[r] & \Omega_{R^D_k/\Lambda_{\infty,k}}\bigotimes B \arrow[r] & \Omega_{R_k/\Lambda_{\infty,k}}\bigotimes B \arrow[r] & \Omega_{R_k/R^D_k}\bigotimes B \arrow[r] & 0
		\end{tikzcd}
		\]
		Applying the snake lemma to this diagram, we get an exact sequence
		\[0\to\Omega_{R^D_k/R^d_k}\bigotimes B \to\Omega_{R_k/R^d_k}\bigotimes B \to\Omega_{R_k/R^D_k}\bigotimes B \to0.\]
		To simplify notations, 
		let's write $V_k=\Omega_{R^D_k/R^d_k}\otimes B $,
		$Q_k=\Omega_{R_k/R^d_k}\otimes B $ 
		and ${W_k=\Omega_{R_k/R^D_k}\otimes B }$. 
		We fix an arbitrary prime ideal $P$ in $B_k$ of height one, 
		then by Proposition \ref{PropertiesOfFittingIdeal} (3),
		we have $\mathrm{F}_{B_k}(Q_k)_P\supset\mathrm{F}_{B_k}(V_k)_P\mathrm{F}_{B_k}(W_k)_P$. 
		Moreover,
		$\mathrm{F}_{B_k}(Q_k)_P/S\mathrm{F}_{B_k}(Q_k)_P=\mathrm{char}_B(Q_0)_P$ 
		by Proposition \ref{PropertiesOfFittingIdeal} (1) and (6). 
		Similar results hold for $V_k$ and $W_k$. 
		So using Nakayama's lemma, one has, 
		for any prime ideal $P$ in $B_k$ of height one,
	    \begin{equation}\label{Step4:FirstIdentity}
		\mathrm{F}_{B_k} (\Omega_{R_k/R^d_k}\bigotimes B )_P=\mathrm{F}_{B_k} (\Omega_{R^D_k/R^d_k}\bigotimes B )_P\mathrm{F}_{B_k} (\Omega_{R_k/R^D_k}\bigotimes B )_P.
		\end{equation}

		Next comparing the following two exact sequences 
		(both are injective on the left because $\Lambda_{\infty,k}$ injects into $R^D_k$, resp. $\Lambda_k$) 
		and using the same argument as above
		\[0\to\Omega_{R^d_k/\Lambda_{\infty,k}}\bigotimes B \to\Omega_{R^D_k/\Lambda_{\infty,k}}\bigotimes B \to\Omega_{R^D_k/R^d_k}\bigotimes B \to0,\]
		\[0\to\Omega_{R^d_k/\Lambda_{\infty,k}}\bigotimes B \to\Omega_{\Lambda_k/\Lambda_{\infty,k}}\bigotimes B \to N_k' \to0,\]
		one gets, for any prime ideal $P$ in $B_k$ of height one,
		\begin{equation}\label{Step4:SecondIdentity}
		\mathrm{F}_{B_k}(N_k')_P=
		I_k
		\mathrm{F}_{B_k}(\Omega_{R^D_k/R^d_k}\bigotimes B)_P.
		\end{equation}

		Now we combine the identities (\ref{Step4:AuxiliaryIdentity}),
		(\ref{Step4:FirstIdentity}) and (\ref{Step4:SecondIdentity})
		to get
		\[\mathrm{F}_{B_k}(M_k)_P=
		I_k
		\mathrm{F}_{B_k}(\mathrm{Sel}(\mathbb{Q}_k,\mathrm{Ad}(\rho^G_\mu))^\ast)_P\]
		and the rest of the proof is the same as in \textbf{Step.3}.
	\end{proof}

    We can show that there is a natural map of torsion $B$-modules
    (see \cite[Proposition 5.28]{HidaHilbert})
    \[
    \mathrm{Sel}(\mathbb{Q}_k,\mathrm{Ad}(\rho^G_\mu))
    \to
    \bigoplus_\psi\mathrm{Sel}(\mathbb{Q},\mathrm{Ad}(\rho^G_\mu)\otimes\psi)
    \]
    whose kernel and cokernel are both annihilated by $p^k$
    (for $k=0$, the kernel and cokernel are both zero). 
    Here $\psi$ runs through the irreducible representations of $\Delta_k$.
    From this we see
    $\mathrm{F}_{B_k}(\mathrm{Sel}
    (\mathbb{Q}_k,\mathrm{Ad}(\rho^G_\mu))^\ast)
    =\mathrm{F}_B(\mathrm{Sel}
    (\mathbb{Q},\mathrm{Ad}(\rho^G_\mu))^\ast)B_k$
    up to (multiplication by)
    a factor a power of $p$. 
    This, combined with
	Corollary \ref{congruenceSelmer} and Theorem \ref{pL}, 
	gives the following.
	\begin{corollary}\label{pLandCongruenceIdeal}
		The assumptions as in the above theorem, 
		then for any prime ideal $P$ in $B_k$ of height one, 
		we have the following identity
		of ideals in $(B_k)_P$, 
		up to a multiplicative factor a power of $p$
		\[
		((\frac{L^{alg}_p (\mathrm{Ad}(\rho^G_\mu),S)}
		{S^{n-1}}B_\infty)_{\Gamma_k})_P
		=I_k
		(\mathfrak{c}_{\tilde{\lambda}_k}B_k)_P.\]
	\end{corollary}

	We can use this theorem to deduce the structures of the Selmer groups
	$\mathrm{Sel} (\mathbb{Q}_{\infty},\mathcal{A}^j(\rho_\mu))$, as we will do in the following subsections.

	\subsection{Case $\mathrm{Sel} (\mathbb{Q}_{\infty},\mathcal{A}^1(\rho_\mu))$}
	
	To treat this case we can take $G=\mathrm{GL}_2$. 
	Note that
	$\mathcal{A}^1(\rho_\mu)
	=\mathrm{Ad}(\rho_\mu)=\mathrm{Ad}(\rho_\mu^{G})$. 
	So we get from Theorem \ref{pL} 
	and Corollary \ref{pLandCongruenceIdeal} 
	the following theorem.
	\begin{theorem}\label{pL1}
		Assume $\textbf{Big} (\overline{\rho}_\pi)$,
		$\textbf{Dist} (\overline{\rho}_\pi)$,
		$\textbf{RegU} (\overline{\rho}_\pi)$ ,
		$\mathcal{L} (\mathbb{Q},\mathcal{A}^1(\rho_\mu))\neq0$
		and that $ B $ is regular, 
		then the dual Selmer group
		$\mathrm{Sel} (\mathbb{Q}_{\infty},\mathcal{A}^1(\rho_\mu))^{\ast}$ 
		is a finitely generated torsion $B_\infty$-module.
		Its characteristic power series $L^{alg}_p (\mathcal{A}^1(\rho_\mu),S)$ 
		has a zero of order $1$ at $S=0$.
		Moreover for any prime ideal $P$ in $B_k$ of height one, we have
		(up to a factor a power of $p$ if $k>0$)
		\[
		((\frac{L^{alg}_p (\mathcal{A}^1(\rho_\mu),S)}{S}B_\infty)_{\Gamma_k})_P
		=I_k^{\mathrm{GL}_2}
		(\mathfrak{c}_{\tilde{\lambda}^\circ_k}B_k)_P.
		\]
	\end{theorem}
	
	In the present case $\mathcal{A}^1(\rho_\mu)$,
	we can give an
	interpretation of the congruence ideal
	$\mathfrak{c}_{\tilde{\lambda}^\circ_0}$ 
	in terms of special values of complex
	adjoint $L$-function $L(\mathrm{Ad}(\pi),s)$ 
	of the automorphic representation 
	$\pi$ 
	(see \cite{Hida88b} or \cite[Corollary 5.8(2)]{HidaArithmetic}).
	
	\begin{proposition}\label{specialvalue1}
		We assume $\textbf{Big}(\overline{\rho}_\Pi)$ and $p\nmid 6\varphi(N)$.
		Suppose that $ B \to B /P_{\pi}$ is the specialization to $\pi$, then the ideal $\mathfrak{c}_{\tilde{\lambda}^\circ_0}(\mathrm{mod}\ P_\pi)$ in $ B /P_\pi$ is generated by
		\[W (\pi)N^{ (a+1)/2}\frac{L (\mathrm{Ad} (\pi),1)}{\Omega (+,\pi)\Omega (-,\pi)}\]
		where $W(\pi)$ is a non-zero complex number measuring the difference between the actions of the involution $\begin{pmatrix}
		0 & -1 \\ N & 0
		\end{pmatrix}$ and the complex conjugation on the representation $\pi$ (see \cite[end of p.31]{HidaArithmetic} for the precise definition of $W(\pi)$).
	\end{proposition}

    \begin{remark}\label{RemarkOnCongruenceIdealAndSpecialLValue}
    	In order to relate congruence ideals to special $L$-values,
    	we need to introduced a notion of cohomological congruence ideals.
    	Let $L^+$ be a totally real number field
    	(not necessarily totally ramified at $p$)
    	and $d=[L^+:\mathbb{Q}]$.
    	We write $H^d$ for the space 
    	$eH^d(\Gamma_1(Np^\infty);\mathcal{O})$
    	of $p$-adic ordinary Hilbert modular forms
    	over $L^+$,
    	localized at the maximal ideal of
    	$\mathbf{T}^\circ_{L^+}$
    	corresponding to the residue Galois representation
    	$\overline{\rho}_\pi$.
    	It is known that $H^d$ is a finitely generated free $\mathbf{T}^\circ_{L^+}$-module. 
    	We then set $\tilde{H}^d=H^d\otimes_{\Lambda^\circ_{L^+}} B $.
    	The specialization to $\pi$ gives a decomposition of the totally quotient ring 
    	$\mathrm{Frac}(\tilde{\mathbf{T}}^\circ_{L^+})$
    	into an algebra direct product
    	$\mathrm{Frac}(\tilde{\mathbf{T}}^\circ_{L^+})=K\times X$.
    	Correspondingly, there is the decomposition of the cohomological group
    	\[\mathrm{Frac}(\tilde{\mathbf{T}}^\circ_{L^+})\tilde{H}^d
    	=\tilde{H}^d_{\pi_{L^+}}\times \tilde{H}^d_X.\]
    	
    	Then we define $M^{coh}_X$ to be the kernel of the map
    	$\tilde{H}^d\to\mathrm{Frac}(\tilde{\mathbf{T}}^\circ_{L^+})\tilde{H}^d
    	\to \tilde{H}^d_{\pi_{L^+}}$.
    	Similarly, we define $M^{coh}_{\pi_{L^+}}$ to be the kernel of the map
    	$\tilde{H}^d\to\mathrm{Frac}
    	(\tilde{\mathbf{T}}^\circ_{L^+})\tilde{H}^d\to \tilde{H}^d_X$.
    	The morphism
    	$\tilde{\lambda}^\circ_{L^+}\colon\tilde{\mathbf{T}}^\circ_{L^+}\to\tilde{A}$
    	is similarly defined as those
    	$\tilde{\lambda}^\circ_k$.
    	Then the cohomological congruence module
    	$C_0^{coh}(\tilde{\lambda}^\circ_{L^+})$ 
    	is defined to be any of the following isomorphic
    	$\tilde{\mathbf{T}}^\circ_{L^+}$-modules
    	\[(\tilde{H}^d_{\pi_{L^+}}\times \tilde{H}^d_X)/\tilde{H}^d
    	\simeq \tilde{H}^d_{\pi_{L^+}}/M^{coh}_{\pi_{L^+}}
    	\simeq \tilde{H}^d_X/M^{coh}_X.\]
    	
    	The cohomological ideal 
    	$\mathfrak{c}^{coh}_{\tilde{\lambda}^\circ_{L^+}}$ 
    	is then defined to be the ideal of annihilators
    	$\mathrm{Ann}_{ B }(C_0^{coh}(\tilde{\lambda}^\circ_{L^+}))$.
    	One can easily show
    	$\mathfrak{c}_{\tilde{\lambda}^\circ_{L^+}}\tilde{H}^d_{\pi_{L^+}}
    	\subset M^{coh}_{\pi_{L^+}}$
    	which implies
    	$\mathfrak{c}^{coh}_{\tilde{\lambda}^\circ_{L^+}}|
    	\mathfrak{c}_{\tilde{\lambda}^\circ_{L^+}}$. 
    	Moreover if
    	$\mathbf{T}^\circ_{L^+}$ is Gorenstein,
    	then 
    	$\mathfrak{c}^{coh}_{\tilde{\lambda}^\circ_{L^+}}
    	=\mathfrak{c}_{\tilde{\lambda}^\circ_{L^+}}$.
    	
    	We use cohomological congruence ideals because
    	it is more accessible to relate the cohomological congruence ideal 
    	to special $L$-values 
    	than relating congruence ideals to special $L$-values. 
    	In fact, if the following conjecture is true, we can show that
    	$\mathfrak{c}^{coh}_{\tilde{\lambda}^\circ_{L^+}}(\mathrm{mod}\ P_\pi)$
    	is generated by the quotient
    	$\frac{L(\mathrm{Ad}(\pi_{L^+}),1)}{\Omega(+,\pi_{L^+})\Omega(-,\pi_{L^+})}$
    	in $\mathcal{O}$.
    	
    	\begin{conjecture}
    		$H^d$ is a free $\mathbf{T}^\circ_{L^+}$-module of rank $2^d$.
    	\end{conjecture}
    	
    	This conjecture has been proved in a recent article 
    	of J.Tilouine and E.Urban
    	(See \cite[Proposition 2]{TilouineUrban})
    	for the automorphic representation 
    	$\pi$ generated by
    	an elliptic primitive cuspidal eigenform $f$ 
    	of weight $k=a+2\geq2$,
    	of conductor $N$, ordinary at $p$,
    	under the hypotheses:
    	$\mathbf{Big}(\overline{\rho}_\pi|_{\Gamma_{L^+}})$,
    	$\mathbf{Dist}(\overline{\rho}_\pi|_{\Gamma_{L^+}})$,
    	$\mathbf{RegU}(\overline{\rho}_\pi|_{\Gamma_{L^+}})$,
    	$p-1>d(a+1)$
    	and $p$ unramified in $L^+/\mathbb{Q}$.
    	For similar results, we refer the reader to
    	\cite{Diamond,Dimitrov,Ghate2002,Ghate2010}.
    \end{remark}

	\subsection{Case $\mathrm{Sel} (\mathbb{Q}_{\infty},\mathcal{A}^{n-1}(\rho_\mu))$ ($n>2$)}
	We write $G=G_{n/\mathbb{Q}}$ and $G'=G_{n-1/\mathbb{Q}}$. 
	For $n>2$, we assume that the Langlands functorialities
	$\mathrm{Sym}^{r-1}\colon \mathrm{GL}_{2/\mathbb{Q}}\to\mathrm{GL}_{r/\mathbb{Q}}$ are established for $r=n, n-1$.
	So we have the following maps induced by these Langlands functorialities.
	\[
	\begin{tikzcd}
	\tilde{\mathbf{T}}_k^{G}\arrow[r,"\tilde{\theta}^{G}_k"]\arrow[rd,"\tilde{\lambda}^{ G}_k"'] & \tilde{\mathbf{T}}_k^\circ \arrow[d,"\tilde{\lambda}_k^\circ"] & \tilde{\mathbf{T}}_k^{ G'}\arrow[l,"\tilde{\theta}^{G'}_k"']\arrow[ld,"\tilde{\lambda}^{G'}_k"] \\
	&  B  &
	\end{tikzcd}
	\]
	
	By Proposition \ref{DecompositionPropertyOfCongruenceIdeal}, we have the following decomposition of congruence ideals
	\[\mathfrak{c}_{\tilde{\lambda}^{G}_k}=\mathfrak{c}_{\tilde{\lambda}^\circ_k}\tilde{\lambda}^\circ_k (\mathfrak{c}_{\tilde{\theta}^{G}_k}), \ \mathfrak{c}_{\tilde{\lambda}^{G'}_k}=\mathfrak{c}_{\tilde{\lambda}^\circ_k}\tilde{\lambda}^\circ_k (\mathfrak{c}_{\tilde{\theta}^{G'}_k}).\]
	
	Recall the following result from representation theory: 
	let $\mathrm{St}$ be the standard representation of the general linear group $\mathrm{GL}_2(W)$ 
	where $W$ is a field of characteristic $0$. 
	Then we have the following decomposition of representations of $\mathrm{GL}_2(W)$:
	\[\mathrm{Ad}(\mathrm{Sym}^{n-1}\mathrm{St})=\bigoplus_{j=1}^{n-1}(\mathrm{Sym}^{2j}\otimes\mathrm{det}^{-j})\mathrm{St}=\bigoplus_{j=1}^{n-1}\mathcal{A}^j(\mathrm{St}).\]
	
	From this, we have the following decomposition of representations of $\Gamma_{\mathbb{Q}_k}$
	\[\mathrm{Ad}(\rho_\mu^{G})=\mathrm{Ad}(\rho^{G_n}_\mu)=\bigoplus_{j=1}^{n-1}\mathcal{A}^j(\rho_\mu)=\mathrm{Ad}(\rho^{G'}_\mu)\bigoplus\mathcal{A}^{n-1}(\rho_\mu),\]
	which gives the following decomposition of Selmer groups

	\[\mathrm{Sel} (\mathbb{Q}_k,\mathrm{Ad}(\rho_\mu^{G}))=\mathrm{Sel} (\mathbb{Q}_k,\mathrm{Ad}(\rho_\mu^{G'}))\bigoplus\mathrm{Sel} (\mathbb{Q}_k,\mathcal{A}^{n-1}(\rho_\mu)).\]
	
	We apply Theorem \ref{pL} and Corollary \ref{pLandCongruenceIdeal} to $G'$ and to $G$. Using the above decomposition of congruence ideals and Selmer groups, one can show the following result.
	\begin{theorem}\label{pLn}
		Assume $\textbf{Big} (\overline{\rho}_{\Pi^{G}})$,
		$\textbf{Dist} (\overline{\rho}_{\Pi^{G}})$,
		$\textbf{RegU} (\overline{\rho}_{\Pi^{G}})$,
		$\mathcal{L} (\mathbb{Q},\mathrm{Ad}(\rho_\mu^{G_r}))\neq0$
		for both $r=n, n-1$ and that $ B $ is regular.
		Then the Selmer group $\mathrm{Sel} (\mathbb{Q}_{\infty},\mathcal{A}^{n-1}(\rho_\mu))^{\ast}$
		is a finitely generated torsion $B_\infty$-module.
		Its characteristic power series $L^{alg}_p (\mathcal{A}^{n-1}(\rho_\mu),S)$
		has a zero of order $1$ at $S=0$.
		Moreover, for any prime ideal $P$ in $B_k$ of height one, we have
		(up to a factor a power of $p$ if $k>0$)
		\[((\frac{L^{alg}_p (\mathcal{A}^{n-1}(\rho_\mu),S)}{S}B_\infty)_{\Gamma_k})_P
		=
		\frac{I^G_k}{I^{G'}_k}
		\frac{(\tilde{\lambda}^\circ_k (\mathfrak{c}_{\tilde{\theta}^{G}_k})B_k)_P}
		{(\tilde{\lambda}^\circ_k (\mathfrak{c}_{\tilde{\theta}^{G'}_k})B_k)_P}\]
	\end{theorem}

	\section{Comparison of $\mathcal{L}$-invariants}
	The $\mathcal{L}$-invariants of $p$-adic $L$-functions were introduced in
	\cite[Conjecture BSD ($p$)-exceptional case]{MazurTateTeitelbaum}  (though the phenomenon already appeared in \cite{FerreoGreenberg}).
	Later R.Greenberg gave a conjectural formula for the $\mathcal{L}$-invariant of a $p$-adic ordinary Galois representation
	(under some more conditions, see \cite[Section 3, The $\mathcal{L}$-invariant]{GreenbergTrivialZero}).
	In \cite{HarronJorza}, the $\mathcal{L}$-invariants of symmetric powers of $p$-adic Galois representations attached to Iwahori level Hilbert modular forms
	(under some technical conditions) are given in terms of logarithmic derivatives of Hecke eigenvalues.
	
	We use notations as in the last subsection.
	Let's first recall the conjecture on the $\mathcal{L}$-invariants
	(for the definitions of various objects in the conjecture, see
	\cite{GreenbergTrivialZero}):
	\begin{conjecture}
		Let $V$ be a $p$-adic ordinary exceptional Galois representations.
		Assume also that the hypotheses S, T and U are all satisfied.
		Write $L_p (V,S)$ for the  (conjectured) $p$-adic $L$-function of $V$.
		Then there is a constant $\mathcal{L}^{\mathrm{Gr}} (V)\in\mathbb{C}_p-\{0\}$ such that
		\[\lim\limits_{S\to0}\frac{L_p (V,S)}{S^e}=\mathcal{L}^{\mathrm{Gr}} (V)\varepsilon' (V)\frac{L_{\infty} (V,0)}{\Omega_V}.\]
		where $e$ is the order of the trivial zero of $L_p (V,S)$ at $S=0$, $\varepsilon' (V)$ is
		a certain product of Euler-like factors and $\Omega_V$ is the Deligne period
		of $V$.
	\end{conjecture}

    We first give a summary of \cite{HarronJorza}, 
    mainly Section 4: computing the $\mathcal{L}$-invariants
    (the authors of \cite{HarronJorza} work with
    eigenvarieties, instead of Hida families.
    We will show how to fit their result 
    to our situation of Hida families).
    This is also to fix some notations.
    For completeness, we reproduce part 
    of \cite[Section 1]{HarronJorza}.
    For a real number $r$ such that $0\leq r<1$,
    let $\mathcal{R}^r$
    be the set of power series
    $f(x)=\sum_{k=-\infty}^\infty a_kx^k$
    holomorphic for $r\leq|x|_p<1$
    with $a_k\in\mathbb{Q}_p$.
    Then we write $\mathcal{R}=\bigcup_{r<1}\mathcal{R}^r$
    for the Robba ring (over $\mathbb{Q}_p$).
    There are actions of a Frobenius $\varphi$
    and $\Gamma=\mathrm{Gal}(\mathbb{Q}_\infty/\mathbb{Q})$ 
    given as follows:
    for any $f(x)\in \mathcal{R}$,
    $(\varphi f)(x)=f((1+x)^p-1)$ and
    $(\tau f)(x)=f((1+x)^{\chi(\tau)}-1)$
    for any $\tau\in\Gamma$
    and $\chi$ the $p$-adic cyclotomic character.
    For any finite extension  $L/\mathbb{Q}_p$,
    we define
    $\mathcal{R}_L=\mathcal{R}\widehat{\otimes}_{\mathbb{Q}_p}L$.
    Similarly, for an affinoid algebra
    $S$ over $\mathbb{Q}_p$,
    we can define $\mathcal{R}_S$.
    We can extend the actions of $\varphi$ and $\Gamma$
    to $\mathcal{R}_L$ and $\mathcal{R}_S$.
    A $(\varphi,\Gamma)$-module over $\mathcal{R}_L$
    is a free $\mathcal{R}_L$-module $D_L$
    of finite rank,
    equipped with a $\varphi$-semi-linear Frobenius map
    $\varphi_{D_L}$ and a semi-linear action of $\Gamma$
    which commute with each other and the induced map
    $\varphi_{D_L}^\ast D_L=D_L\otimes_\varphi\mathcal{R}_L\to D_L$
    is an isomorphism.
    Similarly, a $(\varphi,\Gamma)$-module over $\mathcal{R}_S$
    is a vector bundle $D_S$
    over $\mathcal{R}_S$ of finite rank,
    equipped with a semi-linear Frobenius $\varphi_{D_S}$
    and a semi-linear action of $\Gamma$,
    which commute with each other and
    the induced map 
    $\varphi_{D_S}^\ast D_S\to D_S$
    is an isomorphism.
    For any character $\delta\colon\mathbb{Q}_p^\times\to S^\times$,
    we define the twisted Robba module
    $\mathcal{R}_S(\delta)$
    to be the $(\varphi,\Gamma)$-module
    $\mathcal{R}_Se_\delta$ with basis $e_\delta$
    such that the semi-linear Frobenius $\varphi$ 
    on $\mathcal{R}_S(\delta)$ is
    $\varphi(xe_\delta)=\varphi(x)\delta(p)e_\delta$
    and for any $\tau\in\Gamma$,
    $\tau(xe_\delta)=\tau(x)\delta(\chi(\tau))e_\delta$.
    We say that a $(\varphi,\Gamma)$-module $D_S$
    is trianguline if there is an increasing filtration
    $F^iD_S$ on $D_S$ such that the graded pieces
    are isomorphic to some
    $\mathcal{R}_S\otimes_SM$
    for a character $\delta$ 
    and an invertible sheaf $M$ over $S$.
    Let $D_{\mathrm{rig},L}^\dagger$ be the functor
    associating to an $L$-linear continuous representation 
    of $\Gamma_{\mathbb{Q}_p}$
    a $(\varphi,\Gamma)$-module over 
    $\mathcal{R}_L$.
    Then $D_{\mathrm{rig},L}^\dagger$ induces 
    an equivalence of categories between 
    the category of $L$-linear continuous representations
    of $\Gamma_{\mathbb{Q}_p}$
    and the category 
    of slope $0$ $(\varphi,\Gamma)$-modules over $\mathcal{R}_L$.
    Similarly, we can define the functor
    $D_\mathrm{st}$
    which associates to $V$
    a semi-stable $(\varphi,\Gamma)$-modules over $\mathcal{R}_L$
    (for the definitions of these two functors, see \cite{Berger}).

    Let $\rho_V\colon\Gamma_\mathbb{Q}\to\mathrm{GL}(V)$
    be a continuous $p$-adic representation of $\Gamma_\mathbb{Q}$,
    unramified outside a finite set of places of $\mathbb{Q}$
    and potentially semi-stable at $p$.
    For any finite prime $l\nmid p\infty$, we set
    $\mathrm{H}^1_f(\mathbb{Q}_l,V)
    =\mathrm{Ker}(\mathrm{H}^1(\mathbb{Q}_l,V)\to
    \mathrm{H}^1(I_l,V))$.
    We then define
    $\mathrm{H}^1_f(\mathbb{Q}_p,V)
    =\mathrm{Ker}(\mathrm{H}^1(\mathbb{Q}_p,V)
    \to\mathrm{H}^1(\mathbb{Q}_p,V\otimes_{\mathbb{Q}_p}B_\mathrm{cris}))
    \simeq \mathrm{H}^1_f(D_\mathrm{rig}^\dagger(V))$
    and
    $\mathrm{H}^1_f(\mathbb{R},V)=\mathrm{H}^1(\mathbb{R},V)$.
    Write $S$ for the finite set of places of $\mathbb{Q}$
    including the ramified places of $V$ and $p$, $\infty$.
    Now we set 
    ($\Gamma_{\mathbb{Q},S}$ is
    the Galois group for the maximal extension of $\mathbb{Q}$
    unramified outside $S$)
    \[\mathrm{H}^1_f(V)=
    \mathrm{Ker}(\mathrm{H}^1(\Gamma_{\mathbb{Q},S},V)
    \to
    \prod_{v\in S}\mathrm{H}^1(\mathbb{Q}_v,V)
    /\mathrm{H}^1_f(\mathbb{Q}_v,V)).\]

    A regular submodule $D$ of $D_\mathrm{st}(V)$
    is a $(\varphi,\Gamma)$-module
    such that
    $D\simeq D_\mathrm{st}(V)/F^0D_\mathrm{st}(V)$.
    For each such regular submodule $D$, there is a filtration
    $F^iD_\mathrm{st}(V)$ on $D_\mathrm{st}(V)$
    (see \cite[Section 2.1.4]{Benois}).
    This filtration in turn induces a filtration
    $F^iD_\mathrm{rig}^\dagger(V)$ on $D_\mathrm{rig}^\dagger(V)$
    given by the intersection
    $F^iD_\mathrm{rig}^\dagger(V)
    =D_\mathrm{rig}^\dagger(V)\bigcap
    (F^iD\otimes_{\mathbb{Q}_p}\mathcal{R}_L[1/\log(1+x)])$.
    The exceptional subspace $W$ of $V$ is defined to be
    $W=F^1D_\mathrm{rig}^\dagger(V)/F^{-1}D_\mathrm{rig}^\dagger(V)$.

    If we assume conditions \textbf{C1} and \textbf{C2} in
    \cite[Section 2.1.2]{Benois},
    we have an isomorphism
    \[\mathrm{H}^1(\Gamma_{\mathbb{Q},S},V)\simeq\bigoplus_{v\in S}
    \mathrm{H}^1(\mathbb{Q}_v,V)/\mathrm{H}^1_f(\mathbb{Q}_v,V). \]
    
    With this isomorphism, we define
    $\mathrm{H}^1(D,V)$
    to be the unique subspace of 
    $\mathrm{H}^1(\Gamma_{\mathbb{Q},S},V)$
    whose image under this isomorphism
    is 
    $\mathrm{H}^1(F^1D_\mathrm{rig}^\dagger(V))/
    \mathrm{H}^1_f(\mathbb{Q}_p,V)$. 
    The natural map 
    $\mathrm{H}^1(\Gamma_{\mathbb{Q},S},V)\to\mathrm{H}^1(\mathbb{Q}_p,V)$
    gives rise to a commutative diagram
    
    \[
    \begin{tikzcd}
    \mathrm{H}^1(D,V)
    \arrow[d,"\kappa_D"]
    \arrow[dr,"\overline{\kappa}_D"]
    &
    \\
    \mathrm{H}^1(W)
    \arrow[r]
    &
    \mathrm{H}^1(\mathrm{gr}^1D_\mathrm{rig}^\dagger(V))
    \end{tikzcd}
    \]
    
    From this diagram, 
    assuming the condition \textbf{C5}
    (see \cite[Section 2.1.9]{Benois}),
    we get the following one
    
    \[
    \begin{tikzcd}
    \mathcal{D}_\mathrm{st}(\mathrm{gr}^1D_\mathrm{rig}^\dagger(V))
    \arrow[r,"i_{D,f}","\sim"']
    &
    \mathrm{H}^1_f(\mathrm{gr}^1D_\mathrm{rig}^\dagger(V))
    \\
    \mathrm{H}^1(D,V)
    \arrow[u,"\rho_{D,f}"']
    \arrow[r,"\overline{\kappa}_D"]
    \arrow[d,"\rho_{D,c}"]
    &
    \mathrm{H}^1(\mathrm{gr}^1D_\mathrm{rig}^\dagger(V))
    \arrow[u,"p_{D,f}"']
    \arrow[d,"p_{D,c}"]
    \\
    \mathcal{D}_\mathrm{st}(\mathrm{gr}^1D_\mathrm{rig}^\dagger(V))
    \arrow[r,"i_{D,c}","\sim"']
    &
    \mathrm{H}^1_c(\mathrm{gr}^1D_\mathrm{rig}^\dagger(V))
    \end{tikzcd}
    \]
    here for the isomorphisms
    $i_{D,f}$ and $i_{D,c}$
    and the projections
    $p_{D,f}$ and $p_{D,c}$, 
    see \cite[Section 1.5.10]{Benois}.
    
    Then D.Benois gave the definition of the $\mathcal{L}$-invariant as follows
    (see \cite[pp.1599-1600]{Benois})
    \begin{definition}
    	The following determinant 
    	is the $\mathcal{L}$-invariant associated to $V$ and $D$
    	\[\mathcal{L}(D,V)
    	=
    	\det\bigg(\rho_{D,f}\circ\rho_{D,c}^{-1}\bigg).\]
    \end{definition}
    
    Explicitly, we have the following formula
    (see \cite[Lemma 3]{HarronJorza}).
    
    \begin{lemma}\label{HJa/b}
    	Let $V_p=V|_{\Gamma_{\mathbb{Q}_p}}$.
    	Suppose that $\mathrm{gr}^1D_{\mathrm{rig},L}^\dagger(V_p)
    	\simeq \mathcal{R}$.
    	Let $\mathrm{H}^1(\mathcal{R})\simeq
    	\mathrm{H}^1_f(\mathcal{R})
    	\oplus\mathrm{H}^1_c(\mathcal{R})$
    	with basis $x=(-1,0)$ and $y=(0,\log_p\chi(\gamma))$.
    	Suppose that
    	$c\in\mathrm{H}^1(D,\rho_V)$
    	is such that the image of $c$
    	in $\mathrm{H}^1(\mathrm{gr}_1D_{\mathrm{rig}}^\dagger(V_p))$
    	is $\xi_p=a_px+b_py$
    	with $b_p\neq0$.
    	Then
    	\[\mathcal{L}(D,V)=\frac{a_p}{b_p}.\]
    \end{lemma}

    With this formula, we can start 
    the computation of $\mathcal{L}$-invariants
    as in \cite[Section 4]{HarronJorza}
    in terms of derivatives of Hecke eigenvalues
    with respect to weights.
    Recall that we place ourselves
    over the base field $\mathbb{Q}$
    and we are working with 
    $\pi$, a non-CM cuspidal holomorphic automorphic form
    over $\mathrm{GL}_2(\mathbb{A}_\mathbb{Q})$,
    cohomological for a local system of highest weight $a>0$.
    Let $\Pi_0$ be the transfer from 
    $\mathrm{GL}_{2/\mathbb{Q}}$ to $G_{n+1}$.
    Then $\Pi_0$ is of cohomological weight
    $(na,(n-1)a,\ldots,a,0)$.
    Then there is an eigenvariety $\mathcal{E}$ 
    over the weight space $\Lambda_0=\Lambda_0^{G_{n+1}}$
    together with a Galois representation
    $\rho_{\mathcal{E}}\colon\Gamma_\mathbb{Q}
    \to\mathrm{GL}_n(\mathcal{O}_\mathcal{E})$
    interpolating the Galois representations attached to the
    classical regular points on $\mathcal{E}$
    (including the Galois representation
    $\rho_{\Pi_0}$ associated to the automorphic form $\Pi_0$).
    Then $\rho_\mathcal{E}|_{\Gamma_{F_\mathfrak{p}}}$
    admits a triangulation with graded pieces
    $\mathcal{R}(\delta_i)$ such that
    $\delta_i(u)=u^{\kappa_i}$
    for $u\in\mathbb{Z}_p^\times$
    and $\delta_i(p)=F_i$
    where $\kappa_i=-(n+1-i)a+i$
    and $F_i=p^{(n-1)/2-i}a_i$
    where $a_i=\theta(U_i)$ is analytic over $\mathcal{E}$
    (for this result, see \cite[Corollary 24]{HarronJorza} or \cite{BellaicheChenevier}).
    
    Now let $D$ be any regular submodule of the Galois representation
    $D_\mathrm{st}(\mathcal{A}(\rho_\pi))$.
    Suppose that there is a point $z_0\in\mathcal{E}$
    specializing to $\Pi_0$
    (i.e. $z_0\circ\mathcal{E}\simeq\mathrm{Sym}^n\rho_\pi$).
    Assume that
    $z_0$ corresponds to the $p$-stabilization
    of $(\Pi_0)_p$ coming from the $p$-stabilization
    of $\pi_p$ giving rise to the regular submodule $D$.
    Assume moreover that there exists 
    global analytic triangulation of
    $D^\dagger_{\mathrm{rig}}(\rho_\mathcal{E}|_{\Gamma_{\mathbb{Q}_p}})$
    with graded pieces
    $\mathcal{R}(\delta_i)$.

    First it is easy to show the following
    (see \cite[Lemma 25]{HarronJorza})
    
    \begin{lemma}
    	Let $\vec{Y}$ be a direction in $\Lambda_0$ and let
    	$\nabla_{\vec{Y}}\rho_\mathcal{E}$ 
    	be the tangent space to $\rho_\mathcal{E}$
    	in the $\vec{Y}$-direction,
    	which makes sense under the assumption
    	that $\mathcal{E}\to\Lambda_{0}$ is \'{e}tale at $z_0$.
    	Then 
    	$c_{\vec{Y}}=(z_0\circ\rho_\mathcal{E})^{-1}\nabla_{\vec{Y}}\rho_\mathcal{E}$
    	is a cohomology class in
    	$\mathrm{H}^1(\mathbb{Q},\mathrm{End}(z_0\circ\rho_\mathcal{E}))$
    	and the natural projection
    	$c_{\vec{Y},j}\in\mathrm{H}^1(\mathbb{Q},\mathcal{A}^j(\rho_\pi))$
    	lies in fact in the cohomology group
    	$\mathrm{H}^1(D,\mathcal{A}^j(\rho_\pi))$
    	for any $1\leq j\leq n$.
    \end{lemma}

    Before giving an explicit formula for $c_{\vec{Y},j}$,
    we need to understand
    what this projection looks like.
    This is given by the following lemma
    (see \cite[Proposition 33]{HarronJorza})
    
    \begin{proposition}
    	Let $V$ be the standard representation of $\mathrm{GL}_2$
    	with basis $(e_1,e_2)$
    	and $V^m=\mathrm{Sym}^mV$
    	has basis $(g_{m,0},g_{m,1},\ldots,g_{m,m})$
    	where $g_{m,l}=(e_1)^{m-l}(e_2)^l$.
    	Suppose that $T\in\mathrm{End}(\mathrm{Sym}^mV)$
    	has an upper triangular matrix with diagonal entries
    	$(a_0,a_1,\ldots,a_m)$
    	with respect to this basis.
    	Then the projection of $T$ to $V_{2k}$ is
    	\[\bigg(\ast\ \cdots\ \ast\ ,\sum_{i=0}^mM_{m,k,i}a_i,\ 0\ \cdots\ 0\bigg)\]
    	with respect to the basis
    	$(g_{2k,i})_i$ of $V_{2k}$ where
    	\begin{equation}\label{M}
    	M_{m,k,i}=\sum_a(-1)^a\frac{m!(m-i+a)!(i+k-a)!}{a!(i-a)!(k-a)!(m-i-k+a)!}.
    	\end{equation}
    	
    	\[M_{m,k,i}=\sum_a(-1)^a\frac{m!(m-i+a)!(i+k-a)!}{a!(i-a)!(k-a)!(m-i-k+a)!}.\]
    	Here the explicit coordinate is the middle one, i.e.
    	the coefficient of $g_{2k,k}$ and $a$ runs through $\mathbb{N}$
    	such that
    	$\max(0,i+k-m)\leq  a\leq \min(i,k)$.
    \end{proposition}

    With this explicit expression in hand, we can give
    the formula for the $\mathcal{L}$-invariants
    $\mathcal{L}(D,\mathcal{A}^j(\rho_\pi))$
    (see \cite[Proposition 26]{HarronJorza}).
    
    \begin{proposition}\label{HJL}
    	Suppose that in the graded pieces $\mathcal{R}(\delta_i)$,
    	we have $\delta_i(u)=u^{\kappa_i}$
    	and $\delta_i(p)=F_i$. Then
    	\[\mathcal{L}(D,\mathcal{A}^j(\rho_\pi))
    	=-\frac{\sum_iM_{n,j,i-1}(\nabla_{\vec{Y}}F_i)/F_i}
    	{\sum_iM_{n,j,i-1}\nabla_{\vec{Y}}\kappa_i} \]
    	as long as this formula makes sense
    	(i.e. the denominator does not vanish).
    \end{proposition}

    Now that we have given a summary of 
    the main computation in \cite{HarronJorza},
    we can use their results to
    calculate our $\mathcal{L}$-invariants.
    Indeed, by \cite[Corollary 2.7.8]{Geraghty},
    the Galois representation
    ${V=\rho_{R_0}=\rho_{\mathbf{T}_0}
    \colon\Gamma_\mathbb{Q}\to\mathcal{G}_{n+1}(\mathbf{T}_0)}$
    is upper triangular when restricted to
    $\Gamma_{\mathbb{Q}_p}$.
    More precisely,
    $V_p=V|_{\Gamma_{\mathbb{Q}_p}}$
    admits a triangulation with
    graded pieces $\mathcal{R}(\delta_i)$
    such that
    $\delta_i(u)=u^{\kappa_i}$ for $u\in\mathbb{Z}_p^\times$
    and $\delta_i(p)=F_i$
    where $\kappa_i=-\lambda_i+i$
    and $F_i=p^{(n-1)/2-i}u_i$.
    Next we should choose a regular submodule $D$
    of $D_\mathrm{st}(\mathcal{A}^j(\rho_\pi))$.
    By \cite[Section 2.2]{HarronJorza} for the case
    $V$ non-split and ordinary
    or \cite[Remark 2.2.2(2)]{Benois},
    we can take
    $D=D_\mathrm{st}(F^{-1}\mathcal{A}^j(\rho_\pi))$.
    By \cite[Remark 2.2.2(2)]{Benois}, 
    the $\mathcal{L}$-invariant
    $\mathcal{L}(D,\mathcal{A}^j(\rho_\pi))$
    of Benois reduces to the $\mathcal{L}$-invariant
    of Greenberg
    when $V$ is ordinary
    (note that the $\mathcal{L}$-invariant
    of Benois for an ordinary Galois representation $V$
    uses the Bloch-Kato Selmer groups,
    while the $\mathcal{L}$-invariant of Greenberg
    uses the Greenberg Selmer groups.
    So \cite[Remark 2.2.2(2)]{Benois}
    says that even though the calculations for these
    two $\mathcal{L}$-invariants of an ordinary representation $V$
    use different Selmer groups,
    the results are the same).
    So the $\mathcal{L}$-invariant of $(\mathcal{A}^j(\rho_\pi),D)$
    of Greenberg/Benois
    is given by
    \[\mathcal{L}^\mathrm{Gr}(\mathcal{A}^j(\rho_\pi))
    :=\mathcal{L}(D,\mathcal{A}^j(\rho_\pi))
    =-\frac{\sum_iM_{n,j,i-1}(\nabla_{\vec{Y}}F_i)/F_i}
    {\sum_iM_{n,j,i-1}\nabla_{\vec{Y}}\kappa_i}. \]
    
    Note that in \cite{HarronJorza}, for the existence of 
    a triangulation of the Galois representation $\rho_\mathcal{E}$,
    we should assume that
    $\alpha/\beta$ is not a $p$-power root of unity.
    Here we do not need this condition since
    the existence of the
    triangulation is guaranteed by \cite[Corollary 2.7.8]{Geraghty}.

    Let $w_0=(n(a+1),(n-1)(a+1),...,a+1,0)$ be the weight in 
    $\Lambda_0^{G_{n+1}}$ corresponding to 
    the representation $\Pi^{G_{n+1}}_0$.
    We still write $z_0$ 
    for a point in $\mathrm{Spec}(A_0^{G_{n+1}})$ over $w_0$
    corresponding to the specialization to $\Pi_0^{G_{n+1}}$.
    Then $z_0\circ\rho_\mu^{G_{n+1}}=\mathrm{Sym}^n\rho_\pi$.
    Since the Hida family $A_0^{G_{n+1}}$ is \'{e}tale over $\Lambda_0^{G_{n+1}}$
    at the point $z_0$,
    for any direction $\vec{Y}$ in $\mathrm{Spec}(\Lambda_0^{G_{n+1}})$,
    it makes sense to define the tangent space
    $\nabla_{\vec{Y}}\rho_\mu^{G_{n+1}}$ to $\rho_\mu^{G_{n+1}}$
    in the $\vec{Y}$-direction. 
    We can summarize the above results in the following proposition.

    \begin{proposition}\label{LinvariantGreenberg}
    	Assume that the $A_0^{G_{n+1}}$-linear Galois representation
    	$\mu_0^{G_{n+1}}\circ\rho_{\mathbf{T}_0^{G_{n+1}}}$,
    	when restricted to $\Gamma_p$,
    	is conjugate to an upper triangular representation
    	and has diagonal entries $\delta_1,\delta_2,\ldots,\ldots,\delta_{n+1}$
    	such that $\delta_i(u)=u^{\kappa_i}$ for any $u\in I_p$
    	and $\delta_i(\mathrm{Frob}_p)=F_i$ for all $j=1,2,\ldots,n+1$.
    	Then for any $j=1,2,\ldots,n$, we have
    	\begin{equation}\label{LInvariantOfHJ}
    	\mathcal{L}^{\mathrm{Gr}} (\mathcal{A}^j(\rho_\pi))
    	=z_0 (-\frac{\sum_{i=1}^{n+1}M_{n,j,i-1}\nabla_{\vec{Y}}\log F_i}
    	{\sum_{i=1}^{n+1}M_{n,j,i-1}\nabla_{\vec{Y}}\kappa_i})
    	\end{equation}
    	as long as this formula makes sense.
    \end{proposition}

    Note that the $F_i$ in Proposition \ref{HJL} is an element in
    $\overline{\mathbb{Q}}_p$
    while the $F_i$ in the above proposition
    is an element in $A_0^{G_{n+1}}$.
    Yet the $F_i$ in Proposition \ref{HJL}
    is just the image of the $F_i$
    in the above proposition
    under the specialization
    $z_0\colon 
    A_0^{G_{n+1}}\to\mathcal{O}\hookrightarrow\overline{\mathbb{Q}}_p$.
    So it will be clear from the context which $F_i$
    we'are using.

	By \cite[Corollary 2.7.8]{Geraghty}, the assumption in the above proposition 
	is clearly satisfied.
	
	We now compare our $\mathcal{L} (\mathbb{Q}, \mathrm{Ad}(\rho^{G_{n+1}}_\mu))$
	with $\mathcal{L}^{\mathrm{Gr}}(\mathcal{A}^j(\rho_\pi))$
	in the above proposition,
	which is the main result of this section.
	\begin{theorem}\label{compareLinvariant}
		Fix an integer $n>0$. Assume that all the hypotheses in
		Theorems \ref{pLn} are satisfied, we have
		\[\prod_{j=1}^n\mathcal{L}^{\mathrm{Gr}}(\mathcal{A}^j(\rho_\pi))
		=z_0(\mathcal{L}
		(\mathbb{Q},\mathrm{Ad}(\rho^{G_{n+1}}_\mu))).\]
	\end{theorem}
	\begin{proof}
		To simplify notations, we write $\mathcal{L}_j^{\mathrm{Gr}}=-\frac{\sum_{i=1}^{n+1}M_{n,j,i-1}
		\nabla_{\vec{Y}}\log F_i}{\sum_{i=1}^{n+1}M_{n,j,i-1}\nabla_{\vec{Y}}\kappa_i}$ 
		and 
		$\mathcal{L}=\mathcal{L} (\mathbb{Q},\mathrm{Ad}(\rho_\mu^{(G_{n+1})}))$.

		In the universal deformation $ (R_0^{ G_{n+1}},\rho_0^{ G_n})$,
		the diagonal entries of the Galois representation $\rho_0^{ G_{n+1}}|_{\Gamma_p}$ are $\psi_{0,1},\psi_{0,2},\ldots,\psi_{0,n+1}$.
		So under the morphism $R_0^{ G_{n+1}}\to A_0^{ G_{n+1}}$,
		$\psi_{0,n+2-i} (\mathrm{Frob}_p)$ is taken to $F_i$ for $i=1,2,\ldots,n+1$
		(moreover, under the morphism
		$A_0^{G_{n+1}}\to B$, $F_i$ is taken to
		$1+T_{0,i}$ by the definition of $T_{0,i}$. This will be used later).
		
		Note that $\mathcal{L}_j^{\mathrm{Gr}}$ ($1\leq j\leq n$) are independent of the choice of the direction $\vec{Y}$
		(as long as the expression makes sense).

		We choose $n$ linearly independent directions in the weight space
		$\mathrm{Spec}(\Lambda_0^{G_{n+1}})$:
		$\vec{Y}_1,\vec{Y}_2,\ldots,\vec{Y}_n$,
		each having components as 
		$\vec{Y}_i=\mathrm{diag}(Y_{i,1},\ldots,Y_{i,n+1})
		=\mathrm{diag}(0_{i-1},1,0_{n-i},-1)$
		where $1\leq i\leq n$
		(the choice of these directions 
		$\vec{Y}_i$ corresponds to the basis 
		$(d\log(1+X_{0,n+1})-d\log(1+X_{0,i})_{i=1}^n$
		of the free $B$-module
		$\Omega_{\Lambda_0/\Lambda_{\infty,0}}\otimes B$).
		Since $\mathcal{L}_j^{\mathrm{Gr}}$ 
		is independent of the directions $\vec{Y}_i$, 
		applying the $n$ directions to the expression for $\mathcal{L}_j^{\mathrm{Gr}}$, 
		we have the following identity of matrices
		(we write $\partial'F_i$ for $\partial\log F_i$
		and $\partial''X_{0,i}$ for $\partial\log(1+X_{0,i})$)
		\begin{align*}
		 -\mathrm{diag}(\mathcal{L}_1,\ldots,\mathcal{L}_n)
		 \begin{pmatrix}
		 M_{n,1,0} & \cdots & M_{n,1,n} \\
		 \vdots & \ddots & \vdots \\
		 M_{n,n,0} & \cdots & M_{n,n,n}
		 \end{pmatrix}
		 \begin{pmatrix}
		 Y_{1,1} & \cdots & Y_{n,1} \\
		 \vdots & \ddots & \vdots \\
		 Y_{1,n+1} & \cdots & Y_{n,n+1}
		 \end{pmatrix}=
		 \\
		 \begin{pmatrix}
		 M_{n,1,0} & \cdots & M_{n,1,n} \\
		 \vdots & \ddots & \vdots \\
		 M_{n,n,0} & \cdots & M_{n,n,n}
		 \end{pmatrix}
		 \begin{pmatrix}
		 \frac{\partial'F_1}{\partial''X_{0,1}} 
		 &
		 \cdots 
		 & 
		 \frac{\partial'F_1}{\partial''X_{0,n+1}} \\
		 \vdots 
		 & 
		 \ddots 
		 & 
		 \vdots \\
		 \frac{\partial'F_{n+1}}{\partial''X_{0,1}} 
		 & 
		 \cdots 
		 & 
		 \frac{\partial'F_{n+1}}{\partial''X_{0,n+1}}
		 \end{pmatrix}
		 \begin{pmatrix}
		 Y_{1,1} & \cdots & Y_{n,1} \\
		 \vdots & \ddots & \vdots \\
		 Y_{1,n+1} & \cdots & Y_{n,n+1}
		 \end{pmatrix}
		\end{align*}
		Written in symbols, this identity becomes $-LMY=MFY$.
		Note that
		\[\sum_{j=1}^{n+1}Y_{i,j}=0,\ \sum_{j=1}^{n+1}\log{F_j}=1.\]
		The above identity of matrices becomes
		\begin{align*}
		&
		-\mathrm{diag}(\mathcal{L}_1,\ldots,\mathcal{L}_n)
		\begin{pmatrix}
		M_{n,1,0}-M_{n,1,n} & \cdots & M_{n,1,n-1}-M_{n,1,n} \\
		\vdots & \ddots & \vdots \\
		M_{n,n,0}-M_{n,n,n} & \cdots & M_{n,n,n-1}-M_{n,n,n}
		\end{pmatrix}
		\begin{pmatrix}
		Y_{1,1} & \cdots & Y_{n,1} \\
		\vdots & \ddots & \vdots \\
		Y_{1,n} & \cdots & Y_{n,n}
		\end{pmatrix}=
		\\
		&
		\begin{pmatrix}
		M_{n,1,0}-M_{n,1,n} & \cdots & M_{n,1,n-1}-M_{n,1,n} \\
		\vdots & \ddots & \vdots \\
		M_{n,n,0}-M_{n,n,n} & \cdots & M_{n,n,n-1}-M_{n,n,n}
		\end{pmatrix}
		\begin{pmatrix}
		\frac{\partial'F_1}{\partial''X_{0,1}}-
		\frac{\partial'F_1}{\partial''X_{0,n+1}} 
		& 
		\cdots 
		& 
		\frac{\partial'F_1}{\partial''X_{0,n}}-
		\frac{\partial'F_1}{\partial''X_{0,n+1}} \\
		\vdots 
		& 
		\ddots 
		& 
		\vdots \\
		\frac{\partial'F_n}{\partial''X_{0,1}}-
		\frac{\partial'F_n}{\partial''X_{0,1}} 
		& 
		\cdots 
		& 
		\frac{\partial'F_n}{\partial''X_{0,n}}-
		\frac{\partial'F_n}{\partial''X_{0,n+1}}
		\end{pmatrix}
		\begin{pmatrix}
		Y_{1,1} & \cdots & Y_{n,1} \\
		\vdots & \ddots & \vdots \\
		Y_{1,n} & \cdots & Y_{n,n}
		\end{pmatrix}
		&
		\end{align*}
		Again written in symbols, the above identity becomes $-LM'Y'=M'F'Y'$. 
		Note that now the matrices $L$, $M'$, $Y'$ 
		and $F'$ are all of size $n\times n$.
		
		By the next lemma, we know that $M'$ is invertible. 
		By definition, $Y'$ is also invertible ($Y'=1_{n\times n}$), 
		thus if we take the determinants of both sides of
		$-L'M'Y'=M'F'Y'$,
		we get
		${(-1)^n\det{L'}=\det{F'}}$.
		In other words,
		\[(-1)^n\prod_{j=1}^n\mathcal{L}_j^{\mathrm{Gr}}=\det{F'}.\]
		
		Now let's relate $\det{F'}$ to $\mathcal{L}$. 
		Recall that
		$\det{F'}=\det(\frac{\partial'F_i}{\partial''X_{0,j}}-
		\frac{\partial'F_i}{\partial''X_{0,n+1}})$,
		and by Definition \ref{definitionofL}, we have
		$\mathcal{L}=\det(\frac{\partial''T_{0,i}}{\partial''X_{0,n+1}}-
		\frac{\partial''T_{0,i}}{\partial''X_{0,j}})$.
		Write the morphism $\phi\colon A_0^{G_{n+1}}\to B$.
		Using the correspondence between $F_i$ and $T_{0,i}$
		(that is, ${\phi(F_i)=1+T_{0,i}}$ for $1\leq i\leq n+1$),
		we see that
		$\phi((-1)^n\prod_{j=1}^n\mathcal{L}_j^{\mathrm{Gr}})
		=\phi(\det{F'})
		=(-1)^n\mathcal{L}$.
		Thus we conclude
		\[\prod_{j=1}^n\mathcal{L}^{\mathrm{Gr}}(\mathcal{A}^j(\rho_\pi))
		=z_0(\det{F'})=z_0
		(\mathcal{L}(\mathbb{Q},\mathrm{Ad}
		(\rho_\mu^{G_{n+1}}))).\]
		
	\end{proof}

	\begin{lemma}
		The matrix $M'$ is invertible.
	\end{lemma}
	\begin{proof}
		We first show that the matrix $M$ has non-trivial determinant $\det{M}\neq0$.
		
		We consider the standard representation $V=\mathbb{C}_p^2$
		of the Lie algebra $\mathfrak{sl}_2 (\mathbb{C}_p)$.
		We fix a basis of $\mathfrak{sl}_2 (\mathbb{C}_p)$ as $H=\begin{pmatrix}
		1 & 0\\
		0 & -1
		\end{pmatrix}$, $X=\begin{pmatrix}
		0 & 1\\
		0 & 0
		\end{pmatrix}$ and $Y=\begin{pmatrix}
		0 & 0\\
		1 & 0
		\end{pmatrix}$ and the standard basis of $V$ as $e_{-1}=\begin{pmatrix}
		0\\
		1
		\end{pmatrix}$ and $e_1=\begin{pmatrix}
		1\\
		0
		\end{pmatrix}$. 
		A vector $v$ is said to have $H$-weight $\lambda$ 
		if $Hv=\lambda v$. 
		So $e_{-1}$  (resp. $e_1$) has $H$-weight $-1$  (resp. $1$).
		
		The irreducible representations of dimension $n+1$ of
		$\mathfrak{sl}_2 (\mathbb{C}_p)$ are all isomorphic to
		$\mathrm{Sym}^nV=\mathbb{C}_p\langle (e_{-1})^n,(e_{-1})^{n-1}e_1,\ldots, (e_1)^n\rangle$. Note that $e_{-1}^ie_1^{n-i}$
		has $H$-weight $n-2i$.
		
		We first relate $\mathrm{Sym}^nV$ to its dual. 
		The dual representation $V^{\ast}$ of $V$ has a dual basis
		$e_{-1}^{\ast},e_1^{\ast}$. The vector $e_{-1}^{\ast}$
		(resp. $e_1^{\ast}$) has $H$-weight $1$  (resp. $-1$).
		We can verify that the $\mathbb{C}_p$-linear map
		$\phi\colon V\to V^\ast$ given by
		$e_1\mapsto -e_{-1}^\ast,\ e_{-1}\mapsto e_1^\ast$
		is $\mathfrak{sl}_2(\mathbb{C}_p)$-equivariant.
		This induces an $\mathfrak{sl}_2(\mathbb{C}_p)$-equivariant isomorphism $\phi_n\colon\mathrm{Sym}^nV\to\mathrm{Sym}^nV^\ast$, given by
		$(e_1)^{n-i}(e_{-1})^i\mapsto (-1)^{n-i}(e^\ast_1)^i(e^\ast_{-1})^{n-i}$.
		
		Note that these isomorphisms all preserve $H$-weights.
		
		Next we relate the tensor product
		$\mathrm{Sym}^aV\otimes\mathrm{Sym}^bV$ to $\mathrm{Sym}^rV$
		where $r=a+b,a+b-2,\ldots,|a-b|$.
		We write 
		$\Xi_{a,b,r}\colon\mathrm{Sym}^aV\otimes\mathrm{Sym}^bV
		\to\mathrm{Sym}^rV$ the non-trivial
		$\mathfrak{sl}_2(\mathbb{C}_p)$-equivariant projection 
		as in \cite[Lemma 2.7.4]{CFS}.
		
		From the above, we can define the projection
		$\phi_{n,n,2k}$ from $\mathrm{Sym}^nV\otimes\mathrm{Sym}^nV^\ast
		=\mathrm{End}(\mathrm{Sym}^nV)$ onto 
		$\mathrm{Sym}^{2k}V$
		for any $k=0,1,\ldots,n$. 
		This is given by
		\[\phi_{n,n,2k}=(-1)^{n-k}\frac{n!n!}{(n-k)!}\Xi_{n,n,2k}\circ(1\otimes\phi_n^{-1}).\]
		More explicitly, let's write
		$\phi_{n,n,2k}((e_1)^{n-i}(e_{-1}^i)\otimes(e_1^\ast)^{n-j}(e_{-1}^\ast)^j)
		=(-1)^j\sum_{w=0}^{2k}C_{n,n,2k}^{i,n-j,w}(e_1)^{2k-w}(e_{-1})^w$.
		By definition, $C_{n,n,2k}^{i,n-i,k}$ is just the coefficient of
		$(e_1)^k(e_{-1})^k$ (of $H$-weight $0$) in the projection of the vector
		$(e_1)^{n-i}(e_{-1}^i)\otimes(e_1^\ast)^{n-i}(e_{-1}^\ast)^i$ (of $H$-weight $0$)
		to the space $\mathrm{Sym}^{2k}V$.
		Since we have
		$\mathrm{End}(\mathrm{Sym}^nV)\simeq\oplus_{k=0}^n
		\mathrm{Sym}^{2k}V$
		as representations of $\mathrm{SL}_2(\mathbb{C}_p)$,
		we see that $\oplus_{k=0}^n\phi_{n,n,2k}$ 
		is an isomorphism from 
		$\mathrm{End}(\mathrm{Sym}^nV)$ 
		to the direct sum
		$\oplus_{k=0}^n\mathrm{Sym}^{2k}V$.
		Since all these $\phi_{n,n,2k}$ preserve $H$-weights, we deduce that
		the map $\oplus_{k=0}^n\phi_{n,n,2k}$ induces an isomorphism between
		the subspace
		$\oplus_{i=0}^n\mathbb{C}_p
		(e_1)^{n-i}(e_{-1}^i)\otimes(e_1^\ast)^{n-i}(e_{-1}^\ast)^i$
		of $\mathrm{End}(\mathrm{Sym}^nV)$
		and the subspace
		$\oplus_{k=0}^n\mathbb{C}_p(e_1)^k(e_{-1})^k$
		of $\oplus_{k=0}^n\mathrm{Sym}^{2k}V$.
		Since the matrix $(C_{n,n,2k}^{i,n-i,k})_{i,k=0}^n$ 
		is just the matrix for this isomorphism
		under the given basis of these two subspaces,
		we see that the matrix $(C_{n,n,2k}^{i,n-i,k})_{i,k}$ is invertible.
		
		By \cite[Proposition 33]{HarronJorza}, we have
		$M_{n,k,i}=(-1)^i\binom{n}{i}C_{n,n,2k}^{i,n-i,k}$.
		It is easy to see that the matrix $B$ is also invertible.
		
		By the formula (\ref{M}), we see that
		$M_{n,0,i}=n!$,which is independent of $i$.
		This gives
		\[\det{M}=(-1)^nM_{n,0,n}\det{M'}.\]
		Clearly, this shows that $M'$ is invertible,
		which concludes the proof.
	\end{proof}

    \paragraph*{\textbf{Acknowledgements}}
    This is part of the author's PhD thesis at University Paris 13. 
    I am grateful to my advisor J.Tilouine for suggesting this problem to me
    and his constant support.
    I thank also E.Urban and H.Hida for useful conversations.

\end{document}